\documentclass[12pt,oneside,reqno]{amsart}
\usepackage{geometry}      		
\usepackage[normalem]{ulem}		
\usepackage{graphicx}			
\usepackage{amsmath}
\usepackage{amsfonts}
\usepackage{amssymb}
\usepackage{amsthm}
\usepackage{bbm}
\usepackage{MnSymbol}	
\usepackage{color}
\usepackage{pgf}
\usepackage{tikz}
\usetikzlibrary{arrows,shapes,decorations,automata,backgrounds,petri}
\usepackage[latin1]{inputenc}
\usepackage{verbatim}

\usepackage{mathrsfs}

\usepackage[foot]{amsaddr}

\usepackage{layout}

\newtheorem{thm}{Theorem}[section]

\newtheorem{lem}[thm]{Lemma}
\newtheorem{prop}[thm]{Proposition}
\theoremstyle{definition}
\newtheorem{defn}[thm]{Definition}
\theoremstyle{remark}
\newtheorem{rem}[thm]{Remark}
\newtheorem{ex}[thm]{Example}

\newcommand{\opn}[1]{\operatorname{#1}}

\newcommand{\mbb}[1]{\mathbb{#1}}
\newcommand{\mc}[1]{\mathcal{#1}}

\newcommand{\ms}[1]{\mathsf{#1}}

\newcommand{\bs}[1]{\boldsymbol{#1}}

\def\>{\rangle}
\def\<{\langle}
\def\RR{\rrangle}
\def\LL{\llangle}

\def\tr{\opn{Tr}}
\def\rank{\opn{rank}}

\def\spn{\opn{span}}
\def\ot{\leftarrow}
\def\0{\bs{0}}
\def\1{\mathbbm{1}}
\def\N{\mbb{N}}
\def\Z{\mbb{Z}}

\def\C{\mbb{C}}
\def\T{\mbb{T}}
\def\D{\mbb{D}}

\def\HH{\mathscr{H}}

  \def\XXint#1#2#3{{\setbox0=\hbox{$#1{#2#3}{\int}$}
      \vcenter{\hbox{$#2#3$}}\kern-.47\wd0}}

\newcommand{\perpo}[1]{#1^\perp}

\begin{document}


\title[A quantum dynamical approach to matrix Khrushchev's formulas]
{A quantum dynamical approach to \\ matrix Khrushchev's formulas}
\author{\small C. Cedzich$^1$, F.A. Gr\"unbaum$^2$, L. Vel\'azquez$^3$, A.H. Werner$^4$, R.F. Werner$^1$}
\address{\scriptsize
$^1$Institut f\"ur Theoretische Physik, Leibniz Universit\"at Hannover, Appelstr. 2, 30167 Hannover, Germany}
\address{\scriptsize
$^2$Department of Mathematics, University of California, Berkeley, CA 94720, USA}
\address{\scriptsize
$^3$Departamento de Matem\'{a}tica Aplicada $\&$ IUMA, Universidad de Zaragoza,
Mar\'{\i}a de Luna 3, 50018 Zaragoza, Spain}
\address{\scriptsize
$^4$Dahlem Center for Complex Quantum Systems, Freie Universit\"at Berlin, 14195 Berlin, Germany}



\date{}	


\begin{abstract}

Khrushchev's formula is the cornerstone of the so called Khrushchev theory, a body of results which has revolutionized the theory of orthogonal polynomials on the unit circle. This formula can be understood as a factorization of the Schur function for an orthogonal polynomial modification of a measure on the unit circle. No such formula is known in the case of matrix-valued measures. This constitutes the main obstacle to generalize Khrushchev theory to the matrix-valued setting which we overcome in this paper.

It was recently discovered that orthogonal polynomials on the unit circle and their matrix-valued versions play a significant role in the study of quantum walks, the quantum mechanical analogue of random walks. In particular, Schur functions turn out to be the mathematical tool which best codify the return properties of a discrete time quantum system, a topic in which Khrushchev's formula has profound and surprising implications.

We will show that this connection between Schur functions and quantum walks is behind a simple proof of Khrushchev's formula via `quantum' diagrammatic techniques for CMV matrices. This does not merely give a quantum meaning to a known mathematical result, since the diagrammatic proof also works for matrix-valued measures. Actually, this path counting approach is so fruitful that it provides different matrix generalizations of Khrushchev's formula, some of them new even in the case of scalar measures.

Furthermore, the path counting approach allows us to identify the properties of CMV matrices which are responsible for Khrushchev's formula. On the one hand, this helps to formalize and unify the diagrammatic proofs using simple operator theory tools. On the other hand, this is the origin of our main result which extends Khrushchev's formula beyond the CMV case, as a factorization rule for Schur functions related to general unitary operators.

\end{abstract}

\maketitle

\section{Introduction}
\label{sec:INTRO}

The beginning of this century has seen the development of two striking ideas that have changed the way one looks at the theory of orthogonal polynomials on the unit circle (OPUC) and its connections with other areas. These ideas are nowadays known under the names of Khrushchev theory \cite{Kh01,Kh02,Si-OPUC} and the theory of CMV matrices \cite{Wa-CMV,CMV,Si-OPUC,Si-CMV}. Indicative of the impact of these theories, the key papers that caused their explosion are the two last entries of Barry Simon's list singling out twelve among the most important in the OPUC history \cite[Appendix D: ``Twelve Great Papers"]{Si-OPUC}.

In a sense, the {\it leitmotiv} of our work is the discovery of the natural meeting point for these two subjects of current interest, and the use of this discovery to lay the foundations of further extensions of Khrushchev theory. Curiously enough, the catalyst for the natural combination of both mathematical issues is another hot topic coming from quantum information theory, the so called quantum walks, the quantum analog of the classical random walks \cite{AhDaZa-QW,Amb-QW,Kem-QW,Ken-QW}. So, in the end, the meeting point becomes a triple junction among three highly topical subjects.

The body of \emph{Khrushchev theory} was developed by Sergei Khrushchev in two seminal papers \cite{Kh01,Kh02} published in 2001 and 2002. His view on OPUC theory emphasizes the role of continued fractions and Schur functions. This perspective allowed him to prove deep results linking the behaviour of OPUC, measures, Verblunsky coefficients and related Schur functions.
The core of Khrushchev theory is an identity, known as Khrushchev's formula, which factorizes the Schur function for an OPUC modification of a measure in terms of (iterates and inverse iterates of) the Schur function for the unmodified measure. Khrushchev's proof of this formula is based on the application of analytical techniques combined with continued fraction expansions of Schur functions and some specific results on OPUC.

\emph{CMV matrices} emerged at the beginning of the nineties in the study of diagonalization processes for unitary matrices in numerical linear algebra (see \cite{Wa-CMV} and references therein). Nevertheless,
they did not receive much attention until their rediscovery in 2003 \cite{CMV} in the OPUC context as the unitary analog of Jacobi matrices for orthogonal polynomials on the real line. The amount of new OPUC results coming from the translation to CMV of the largely established theory of Schr\"odinger and Jacobi operators, together with the multiple connections they opened to other pure and applied areas, contributed to the wide recognition of CMV matrices.

\emph{Quantum walks (QWs)} also first appeared in the early nineties \cite{AhDaZa-QW} as simple quantum mechanical models of discrete time evolutions of single particles. In particular, QWs exhibit a large variety of quantum effects, e.g. Anderson localization and ballistic transport,
which is in striking contrast with their classical counterparts, the random walks \cite{anderson, electW, Grimmet,timeRandom, joye1}. Motivated by the widespread use of classical random walks in the design of randomized algorithms, QWs started being used in quantum algorithms, discovering that they can provide an exponential speedup over any classical one \cite{Amb-QW,Kem-QW,Ken-QW}. This triggered a growing and widening interest in the topic in theoretical and mathematical physics \cite{childs,spactRand,molecules}, as well as in experimental physics \cite{Exp-Neut1,Exp-TrappedIons1,Exp-TrappedIons2,Exp-Neut2}.

Let us briefly comment on the connections between these three different areas.
CMV matrices have been used for an operator approach to Khruschev's formula, understood as the unitary analog of Dirichlet decoupling of Schr\"odinger operators \cite[Chapter 4]{Si-OPUC}. Such an approach uses a decoupling of CMV matrices by finite rank perturbations, together with a number of results on matrix Schur functions.

The connection between CMV matrices and QWs rests on the conservation of probability in closed systems, mathematically expressed by the unitarity of the quantum evolution operator. This connection can be seen as a consequence of the natural idea that the simplest non-trivial quantum models are described by the canonical form of the unitaries, i.e. the CMV matrices \cite{CGMV-QW,CGMV-DEFECT,CGMV,GV-RIESZ}. This parallels the role of Jacobi matrices in the simplest classical random walks, the birth-death processes, which only allow for nearest neighbour transitions \cite{birthdeath}.

Very recently, this link between QWs and CMV matrices has led to an unexpected relation between OPUC and quantum mechanics: Schur functions codify the return properties of discrete time quantum systems as generating functions of first time return probability amplitudes \cite{GVWW,BGVW}. Khrushchev's formula enters the game via this relation:
It has been used to compute return probabilities and expected return times with OPUC techniques \cite{GVWW,BGVW}; it also implies a surprising invariance of the return properties with respect to certain local perturbations (see comments  in \cite[Section 6]{GVWW} and \cite[Section 4]{BGVW}).

The above dynamical interpretation of Schur functions is the starting point for the new results we report in this paper.
We use this connection to obtain a dynamical interpretation of Khrushchev's factorization formula. This factorization corresponds to a splitting of a quantum system into subsystems in such a way that its return properties can be reconstructed from those of the subsystems. The important insight in this paper is that this relation goes both ways. In particular, this allows us to apply methods from the theory of QWs to OPUC. In this way we obtain a purely diagrammatic proof of Khrushchev's formula based on path counting for return paths of what we call ``CMV quantum walks". In contrast to previous proofs of Khrushchev's formula, this one is based solely on algebraic arguments and does not require perturbation theory.
Furthermore, the proof has the advantage of being applicable to more general situations, such as matrix-valued OPUC, which are of central interest for this paper.

Matrix-valued measures on the unit circle give rise to two different inner products, leading to two types of matrix-valued OPUC (MOPUC) (see \cite{DaPuSi-MOPUC} and references therein). Matrix Schur functions play for MOPUC a similar role as in the scalar case. There is also an underlying CMV theory, developed in \cite{Si-OPUC,Si-CMV,DaPuSi-MOPUC}, which is based on the use of block CMV matrices, a generalization of CMV matrices in which the coefficients are substituted by square matrices.
MOPUC also appear naturally in the context of QWs: On the one hand, QWs on the whole line can be described by block CMV matrices \cite{CGMV-QW,CGMV-DEFECT,CGMV}. On the other hand, matrix-valued Schur functions codify the return properties for higher dimensional subspaces \cite{BGVW}.

The transition from the scalar to the matrix-valued setting has been carried out successfully despite the problem of non-commutativity.
However, none of the methods already used to obtain Khrushchev's formula for scalar measures have been generalized to matrix measures at present. Actually, D. Damanik, A. Pushnitski and B. Simon declare explicitly in their survey on matrix orthogonal polynomials \cite[page 63]{DaPuSi-MOPUC}:

{``\it Among the deepest and most elegant methods in OPUC are those of Khrushchev [$\dots$]. We have not been able to extend them to MOPUC! We regard their extension as an important open question"}

The main objective of this work is to establish the starting point to fill the gap pointed out in this quote, providing Khrushchev's formulas for MOPUC. As it turns out the path counting arguments from the scalar case carry over directly to the matrix-valued one.
Indeed, we do more than this: We prove an abstract generalization of Khrushchev's formula for unitary operators, from which several versions for MOPUC follow, some of them new even for scalar OPUC.
This abstract Khrushchev's formula amounts to a factorization of operator-valued Schur functions.
A sufficient condition for this Schur factorization is the existence of a certain factorization of the original unitary operator into unitaries on smaller subspaces. This approach differs in two aspects from decouplings previously used to prove the scalar Khrushchev formula: the unitary factors act on subspaces which are orthogonal only up to an ``overlapping" subspace; this instead of a true decoupling is the price to pay for a strict operator factorization into unitaries with no need of perturbations.

This abstract version of Khrushchev's formula on the level of unitary operators suggests some relations of this material with other topics such as the index theory for QWs \cite{indexpap} and  dilation theory via transfer functions and related issues (see \cite{Ka-DT} and references therein).

The paper is organized as follows: In Section~\ref{sec:GEN-K} we briefly recall the connection between return properties and Schur functions before developing an abstract generalization of Khrushchev's formula for unitary operators. This main result consists of two parts: a characterization of unitary operators admitting overlapping factorizations (Theorem~\ref{thm:CHARAC-OVERLAP}); a proof that such an overlapping factorization implies a Khrushchev like factorization of operator-valued Schur functions (Theorem \ref{thm:GEN-K}).
Khrushchev's formulas for MOPUC are proven in Sections~\ref{sec:CMV-K} and \ref{sec:HESS-K} as an application of the abstract result to block CMV and Hessenberg matrices. To complement the abstract operator result, the path counting approach is presented in Sections~\ref{sec:GEN-PC} and \ref{sec:CMV-PC} for several reasons: path counting gives a closer and more intuitive approach to Khrushchev's formula which helps to unravel its physical meaning. Also, path counting is crucial to guess the abstract Khrushchev formula, so it could be important in surmising other results before developing a more formal approach, and diagrammatics for OPUC represents a novel and effective technique which eventually could be used for other purposes. Finally, Section~\ref{sec:MtoS} illustrates the impact of matrix Khrushchev's formulas in the derivation of scalar Schur functions related to scalar OPUC.

\section{Khrushchev's formula for unitary operators and quantum recurrence}
\label{sec:GEN-K}

Before going into the matrix, or more generally, operator versions of Khrushchev's formula we first recall the scalar one. Khrushchev's formula expresses a factorization property of Schur functions, which are those analytic functions $f(z)$ on the unit disk $\D=\{z\in\C:|z|<1\}$ satisfying $|f(z)|\le1$.
Schur functions $f(z)$ are by Herglotz' theorem in one-to-one correspondence with probability measures $\mu$ on the unit circle $\T=\{z\in\C:|z|=1\}$ via the relations
\begin{equation} \label{eq:C-S}
 F(z) = \int_\T \frac{t+z}{t-z}\,d\mu(t),
 \qquad
 f(z) = z^{-1} (F(z)-1)(F(z)+1)^{-1},
\end{equation}
which define the Carath\'eodory function $F(z)$ and Schur function $f(z)$ of the measure $\mu$.

A key role in the theory of Schur functions is played by the Schur algorithm
\begin{equation} \label{eq:SA}
 f_0(z) = f(z);
 \qquad
 f_{j+1}(z) = \frac{1}{z} \frac{f_j(z)-\alpha_j}{1-\overline\alpha_j f_j(z)},
 \quad \alpha_j=f_j(0), \quad j\ge0,
\end{equation}
which characterizes any Schur function $f$ by a finite or infinite sequence $(\alpha_0,\alpha_1,\alpha_2,\dots)$ of parameters in the closed unit disk $\overline\D$, called the Schur parameters of $f$. The function $f_j$, named the $j$-th iterate of $f$, is a Schur function itself which is characterized by the sequence $(\alpha_j,\alpha_{j+1},\alpha_{j+2},\dots)$.
In addition, Geronimus' theorem asserts that the Schur parameters coincide with the Verblunsky coefficients appearing in the recurrence relation for the orthogonal polynomials of the corresponding probability measure.

The above connection establishes a one-to-one correspondence between probability measures supported on infinitely many points of $\T$ and infinite sequences in $\D$. The sequence of Schur parameters terminates when one of them lies on $\T$, meaning that the related Schur function corresponds to a finitely supported probability measure. This is for instance the case for the $j$-th inverse iterate $b_j$ of a Schur function $f$ with Schur parameters $(\alpha_0,\alpha_1,\dots)$, which is defined as the Schur function with Schur parameters $(-\overline\alpha_{j-1},-\overline\alpha_{j-2},\dots,-\overline\alpha_0,1)$.

If $\mu$ is a probability measure on $\T$ with corresponding Schur function $f$ and $\varphi_j$ are the related orthonormal polynomials, Khrushchev's formula states that the Schur function of $|\varphi_j|^2\,d\mu$ factorizes as a product $f_jb_j$ of an iterate $f_j$ and an inverse iterate $b_j$ of $f$. This simple result has profound consequences in the theory of orthogonal polynomials. Roughly speaking, the reason for this is that it contains in a single ``formula" the trinity which constitutes the core of OPUC theory: probability measure, orthogonal polynomials and Verblunsky coefficients/Schur parameters.

Recently, the impact of this formula in the study of discrete time quantum systems has been recognized \cite{GVWW,BGVW}. In this context, Khrushchev's factorization becomes a tool to split a physical system into subsystems such that the return properties of these subsystems allow us to recover those of the whole system. An operator interpretation of the measure $|\varphi_j|^2\,d\mu$ involved in Khrushchev's formula is central to understand both novel ideas, the quantum mechanical interpretation of Khrushchev's formula and its eventual extensions to more general contexts. Actually, one of the purposes of the present work is to make evident that these ideas are not only closely related, but reveal a fruitful symbiosis.

Given a probability measure $\mu$ on $\T$, any normalized function $\phi\in L^2_\mu$ defines a new probability measure $d\mu_\phi=|\phi|^2d\mu$ on $\T$ whose moments are given by
\[
 \int z^n\,d\mu_\phi(z) = \langle \phi|U_\mu^n\phi \rangle_\mu,
 \qquad
 n\in\Z,
\]
where $\langle\cdot\,|\,\cdot\rangle_\mu$ is the inner product in $L^2_\mu$ and $U_\mu$ denotes the unitary multiplication operator
\begin{equation} \label{eq:multip-op}
 \begin{array}{c @{\hspace{5pt}} c @{\hspace{2pt}} c @{\hspace{2pt}} c}
 U_\mu\colon & L^2_\mu & \longrightarrow & L^2_\mu
 \\
 & \phi(z) & \mapsto & z\phi(z)\; .
\end{array}
\end{equation}
The measure $\mu_\phi$ is called the spectral measure of $\phi$ with respect to $U_\mu$.

Using this language, Khrushchev's formula for a measure $\mu$ can be viewed as a factorization of the Schur function for a spectral measure with respect to the unitary operator $U_\mu$, namely, the spectral measure $d\mu_{\varphi_j}=|\varphi_j|^2\,d\mu$ of an orthonormal polynomial $\varphi_j$.

This point of view naturally leads to the search for generalizations of Khrushchev's formula to general unitary operators, understood as factorizations of Schur functions for spectral measures related to such operators. Moreover, spectral measures can be associated not only with vectors, but also with subspaces, which give rise to operator-valued instead of scalar measures on $\T$. The corresponding Khrushchev formulas should be by analogy factorizations of operator-valued Schur functions, i.e. analytic functions on the unit disk with values in the set of contractions on a Hilbert space. These functions are related to normalized operator-valued measures exactly as in \eqref{eq:C-S}, where the symbol $1$ stands now for the identity operator $\1$ on an appropriate space.

More precisely, given a unitary operator $U$ on a Hilbert space $(\HH,\<\cdot\,|\,\cdot\>)$, its spectral decomposition $U=\int z\,dE(z)$ induces a spectral measure $\mu_V=PEP$ for any subspace $V\subset \HH$, with $P=P_V$ the orthogonal projection of $\HH$ onto $V$. In other words, $\mu_V$ is the unique measure on $\T$ with values in the set of operators on $V$ whose moments are given by
\begin{equation} \label{eq:SPEC-MEASURE}
 \mu_{V,n} = \int z^n\,d\mu_V(z) = PU^nP, \qquad n\in\Z.
\end{equation}
We define the Carath\'eodory and Schur functions, $F_V$ and $f_V$, of a subspace $V\subset \HH$ as those related to the spectral measure $\mu_V$. The case of a one-dimensional subspace $V=\spn\{\psi\}$ leads to the notion of spectral measure $\mu_\psi=\langle\psi|E\psi\rangle$, Carath\'eodory function $F_\psi$ and Schur function $f_\psi$ of a normalized vector $\psi\in \HH$. These concepts depend on the unitary operator $U$, so we will refer to a $U$-spectral measure, $U$-Carath\'eodory function or $U$-Schur function when it is convenient to make explicit the operator dependence.

The values of $\mu_V$, $F_V$ and $f_V$ must be considered as operators on $V$. For example, due to the normalization of $\mu_V$ we have that $F_V(0)=\1_V$ is the identity operator on $V$. From the spectral decomposition of $U$ we find that the the moments of a spectral measure provide the power expansion of the corresponding $U$-Carath\'eodory function,
\[
 F_V(z) = P(U+z\1_\HH)(U-z\1_\HH)^{-1}P =
 \1_V + 2\sum_{n\ge1} \mu_{V,n}^\dag \, z^n.
\]
The analog of this result for the power expansion of $U$-Schur functions is not so trivial and was obtained only recently.
\begin{lem}[see \cite{GVWW, BGVW}]\label{lem:STay} Let $U$ be a unitary operator on a Hilbert space $\HH$. The $U$-Schur function of  a closed subspace $V\subset \HH$ satisfies
{\rm
\begin{equation} \label{eq:f-U}
 f_V(z) = P(U-z\perpo{P})^{-1}P =
 \sum_{n\ge1} a_{V,n}^\dag \, z^{n-1},
 \kern17pt
 a_{V,n} = PU(\perpo{P}U)^{n-1}P,
 \kern17pt
  P^\perp= \1_\HH-P.
\end{equation}
}
\end{lem}
For the convenience of the reader we give a short alternative proof of the lemma.
\begin{proof}
We only need to show that the expression of $f_V$ given by \eqref{eq:f-U} is related to the Carath\'eodory function $F_V$ by $\frac{1}{2}(\1_V-zf_V(z))(F_V(z)+\1_V)=\1_V$, which is equivalent to \eqref{eq:C-S}. Since the values of $f_V$ and $F_V$ must be understood as operators on $V$, the above identity reads in $\HH$ as
\begin{align*}
[P-zP(U-z\perpo{P})^{-1}P][P(U-z\1_\HH)^{-1}(U+z\1_\HH)P+P]=2P\;.
\end{align*}
The validity of this equality is easily checked by rewriting its left-hand side as
\begin{multline*}
P[\1_\HH-z(U-z\perpo{P})^{-1}P][(U-z)^{-1}(U+z)+\1_\HH]P \\= 2P(U-z\perpo{P})^{-1}UP
 = 2P(1-zU^\dag\perpo{P})^{-1}P = 2P\; ,
\end{multline*}
where the last step is due to the fact that $\perpo{P}P=\0_\HH$ is the null operator on $\HH$.
\end{proof}

If we interpret $U$ as the one-step operator of a discrete time quantum evolution, \eqref{eq:f-U} gives a quantum mechanical meaning to the Taylor coefficients of $f_V$ \cite{GVWW, BGVW}: assuming that the evolution starts at some initial state $\psi\in V$, $\|a_{V,n}\psi\|^2$ corresponds to the probability to return for the first time to the subspace $V$ after exactly $n$ steps, because the projection $\perpo{P}$ conditions on the event ``no return" during the first $n-1$ time steps. This connects Lemma \ref{lem:STay} to the issue of quantum recurrence, which was indeed the context in which this result was first proven. Following quantum mechanical terminology, we refer to the operator coefficients $a_{V,n}$ as the first return amplitudes of $V$. The generating function of these amplitudes
\begin{equation} \label{eq:a-U}
 a_V(z) = \sum_{n\ge1} a_{V,n}z^n = zUP(\1_V-z\perpo{P}U)^{-1}P = zf_V^\dag(z),
 \qquad
 f^\dag(z) = f(\overline{z})^\dag,
\end{equation}
will be called the first return generating function of $V$.

A particularly simple but useful example of a constant Schur function for any unitary operator $U$ is given when $V=\HH$ is the whole Hilbert space, so that
\begin{equation} \label{eq:fH}
 f_\HH(z) = U^\dag\; .
\end{equation}
In addition, the identity operator $U=\1_\HH$ gives the identity Schur function $f_V(z)=\1_V$ for an arbitrary subspace $V$.

Returning to the factorization properties of Schur functions, the naive idea that a factorization of a unitary operator induces a factorization of the related Schur functions is not true in general. For instance, from \eqref{eq:f-U} we find that the choice $U=\frac{1}{\sqrt{2}}\left(\begin{smallmatrix}1&1\\1&-1\end{smallmatrix}\right)$ and $V=\text{span}\{e_1\}$ leads to the Schur function $f_{e_1}(z)=(1+\sqrt{2}z)(\sqrt{2}+z)^{-1}$ for $e_1=\left(\begin{smallmatrix}1\\0\end{smallmatrix}\right)$, while we know that the identity matrix $U^2$ yields for any vector the constant Schur function $1\ne f_{e_1}^2$. Nevertheless, we will see that such Schur factorizations appear in the case of what we will call ``overlapping" factorizations of unitaries.

\begin{defn} \label{def:overlap}
Let $U$ be a unitary operator on a Hilbert space $\HH$. Suppose that $\HH$ admits an orthogonal decomposition $\HH = \HH_L \oplus \HH_C \oplus \HH_R$ such that
\begin{equation} \label{eq:overlap}
 U = (U_{LC} \oplus \1_{\HH_R})(\1_{\HH_L} \oplus U_{CR}),
\end{equation}
with $U_{LC}$ and $U_{CR}$ unitary operators on $\HH_{LC} = \HH_L \oplus \HH_C$ and $\HH_{CR} = \HH_C \oplus \HH_R$ respectively. Then, we say that \eqref{eq:overlap} is an $\HH_C$-overlapping factorization of $U$. We call $\HH_C$, $\HH_L$ and $\HH_R$ the overlapping, left and right subspaces, and we will refer to $U_{LC}$ and $U_{CR}$ as the left and right operators respectively.
\end{defn}

\begin{ex} \label{ex:G3.G4}
The unitary operators which in some orthonormal basis $\{\psi_j\}$ have the matrix representations
\[
\begin{aligned}
 & \left(\begin{smallmatrix}
 \frac{-1}{3}&\frac{2}{3}&\frac{-1}{3}&\frac{1}{3}&\frac{1}{3}&\frac{1}{3}
 \\[1pt]
 \frac{2}{3}&\frac{-1}{3}&\frac{-1}{3}&\frac{1}{3}&\frac{1}{3}&\frac{1}{3}
 \\[1pt]
 \frac{2}{3}&\frac{2}{3}&\frac{1}{6}&\frac{-1}{6}&\frac{-1}{6}&\frac{-1}{6}
 \\[1pt]
 0&0&\frac{1}{2}&\frac{-1}{2}&\frac{1}{2}&\frac{1}{2}
 \\[1pt]
 0&0&\frac{1}{2}&\frac{1}{2}&\frac{-1}{2}&\frac{1}{2}
 \\[1pt]
 0&0&\frac{1}{2}&\frac{1}{2}&\frac{1}{2}&\frac{-1}{2}
 \end{smallmatrix}\right)
 =
 \left(\begin{smallmatrix}
 \frac{-1}{3}&\frac{2}{3}&\frac{2}{3}&0&0&0
 \\[1pt]
 \frac{2}{3}&\frac{-1}{3}&\frac{2}{3}&0&0&0
 \\[1pt]
 \frac{2}{3}&\frac{2}{3}&\frac{-1}{3}&0&0&0
 \\[1pt]
 0&0&0&\kern2pt 1 \kern2pt&0&0
 \\[3pt]
 0&0&0&0&\kern2pt 1 \kern2pt&0
 \\[3pt]
 0&0&0&0&0&\kern2pt 1 \kern2pt
 \end{smallmatrix}\right)
 \left(\begin{smallmatrix}
 \kern2pt 1 \kern2pt&0&0&0&0&0
 \\[3pt]
 0&\kern2pt 1 \kern2pt&0&0&0&0
 \\[1pt]
 0&0&\frac{-1}{2}&\frac{1}{2}&\frac{1}{2}&\frac{1}{2}
 \\[1pt]
 0&0&\frac{1}{2}&\frac{-1}{2}&\frac{1}{2}&\frac{1}{2}
 \\[1pt]
 0&0&\frac{1}{2}&\frac{1}{2}&\frac{-1}{2}&\frac{1}{2}
 \\[1pt]
 0&0&\frac{1}{2}&\frac{1}{2}&\frac{1}{2}&\frac{-1}{2}
 \end{smallmatrix}\right),
 \\
 & \left(\begin{smallmatrix}
 \frac{-1}{3}&0&0&\frac{2}{3}&\frac{2}{3}
 \\[1pt]
 \frac{2}{3}&\frac{1}{2}&\frac{-1}{2}&\frac{1}{6}&\frac{1}{6}
 \\[1pt]
 \frac{2}{3}&\frac{-1}{2}&\frac{1}{2}&\frac{1}{6}&\frac{1}{6}
 \\[1pt]
 0&\frac{1}{2}&\frac{1}{2}&\frac{-1}{2}&\frac{1}{2}
 \\[1pt]
 0&\frac{1}{2}&\frac{1}{2}&\frac{1}{2}&\frac{-1}{2}
 \end{smallmatrix}\right)
 =
 \left(\begin{smallmatrix}
 \frac{-1}{3}&\frac{2}{3}&\frac{2}{3}&0&0
 \\[1pt]
 \frac{2}{3}&\frac{-1}{3}&\frac{2}{3}&0&0
 \\[1pt]
 \frac{2}{3}&\frac{2}{3}&\frac{-1}{3}&0&0
 \\[1pt]
 0&0&0&\kern2pt 1 \kern2pt&0
 \\[3pt]
 0&0&0&0&\kern2pt 1 \kern2pt
 \end{smallmatrix}\right)
 \left(\begin{smallmatrix}
 \kern2pt 1 \kern2pt&0&0&0&0
 \\[1pt]
 0&\frac{-1}{2}&\frac{1}{2}&\frac{1}{2}&\frac{1}{2}
 \\[1pt]
 0&\frac{1}{2}&\frac{-1}{2}&\frac{1}{2}&\frac{1}{2}
 \\[1pt]
 0&\frac{1}{2}&\frac{1}{2}&\frac{-1}{2}&\frac{1}{2}
 \\[1pt]
 0&\frac{1}{2}&\frac{1}{2}&\frac{1}{2}&\frac{-1}{2}
 \end{smallmatrix}\right),
\end{aligned}
\]
admit overlapping factorizations with overlapping subspace $\HH_C=\spn\{\psi_3\}$ and $\HH_C=\spn\{\psi_2,\psi_3\}$ respectively. The related left and right subspaces are $\HH_L=\spn\{\psi_1,\psi_2\}$, $\HH_R=\spn\{\psi_4,\psi_5,\psi_6\}$ in the first example, and $\HH_L=\spn\{\psi_1\}$, $\HH_R=\spn\{\psi_4,\psi_5\}$ in the second one. The left and right operators $U_{LC}$ and $U_{CR}$ have the same matrix representation in both cases
\[
 U_{LC} = \left(\begin{smallmatrix}
 \frac{-1}{3}&\frac{2}{3}&\frac{2}{3}
 \\[1pt]
 \frac{2}{3}&\frac{-1}{3}&\frac{2}{3}
 \\[1pt]
 \frac{2}{3}&\frac{2}{3}&\frac{-1}{3}
 \end{smallmatrix}\right),
 \qquad
 U_{CR} =
 \left(\begin{smallmatrix}
 \frac{-1}{2}&\frac{1}{2}&\frac{1}{2}&\frac{1}{2}
 \\[1pt]
 \frac{1}{2}&\frac{-1}{2}&\frac{1}{2}&\frac{1}{2}
 \\[1pt]
 \frac{1}{2}&\frac{1}{2}&\frac{-1}{2}&\frac{1}{2}
 \\[1pt]
 \frac{1}{2}&\frac{1}{2}&\frac{1}{2}&\frac{-1}{2}
 \end{smallmatrix}\right),
\]
$U_{LC}$ acting on $\HH_{LC}=\spn\{\psi_1,\psi_2,\psi_3\}$, and $U_{CR}$ acting on $\HH_{CR}=\spn\{\psi_3,\psi_4,\psi_5,\psi_6\}$ in the first example and on $\HH_{CR}=\spn\{\psi_2,\psi_3,\psi_4,\psi_5\}$ in the second one.
\hfill$\scriptstyle\blacksquare$
\end{ex}

Overlapping factorizations of unitaries will be key to establish an abstract operator version of Khrushchev's formula. For easy recognition of those unitaries giving rise to Khrushchev's like formulas, we provide a simple characterization of these overlaping factorizations in the following theorem.

\begin{thm} \label{thm:CHARAC-OVERLAP}
{\bf \small [Characterization of overlapping factorizations of unitaries]}
\newline
Let $\HH=\HH_L\oplus\HH_C\oplus\HH_R$ be an orthogonal decomposition of a Hilbert space, and let $P_A$ denote the orthogonal projection of $\HH$ onto $\HH_A$. Then, a unitary operator $U$ on $\HH$ has an $\HH_C$-overlapping factorization $U=(U_{LC}\oplus\1_{\HH_R})(\1_{\HH_L}\oplus U_{CR})$ if and only if
\begin{align} \label{eq:charac-overlap}
 P_RUP_L=\0_\HH \; \text{ and } \; \rank P_{LC}UP_{CR}=\rank P_C.
\end{align}
\noindent Moreover, in this case the factorization is unique up to a unitary $U_C$ on $\HH_C$, i.e., any other factorization is of the form $\hat{U}_{LC}=U_{LC}(\1_{\HH_L}\oplus U_C)$ and {\rm $\hat{U}_{CR}=(U_C^\dag\oplus\1_{\HH_R})U_{CR}$}. When either $\dim\HH_L<\infty$ or $\dim\HH_R<\infty$ the rank condition in \eqref{eq:charac-overlap} can be omitted.
\end{thm}

\begin{proof} Let us prove separately each of the assertions.

\noindent {\it Factorization\;$\Rightarrow$\,\eqref{eq:charac-overlap}:}
Using the overlapping factorization we get $P_RUP_L=P_RP_L=\0_\HH$, and $P_{LC}UP_{CR}=U_{LC}U_{CR}=U_{LC}(P_L+P_C+P_R)U_{CR}=U_{LC}P_CU_{CR}$, which has the same rank as $P_C$, since unitary factors do not change the rank.

\noindent {\it \eqref{eq:charac-overlap}\,$\Rightarrow$\,{\it Factorization}:}
We split $U$ into three sumands,
\[
\begin{aligned}
 U & = P_{LC}UP_L + P_RUP_{CR} + P_{LC}UP_{CR} \\
   & = UP_L + P_RU + P_{LC}UP_{CR} \\
   & = K_L \kern7pt + K_R \kern7pt + K.
\end{aligned}
\]
Here at the first equality we used $P_RUP_L=\0_\HH$, and again in the second line by adding this zero term twice and using $P_{LC}+P_R=P_L+P_{CR}=\1_\HH$. The third line introduces the abbreviations we want to use in this proof.
Since $K_L^\dag K_L=P_L$ and $K_RK_R^\dag=P_R$ are orthogonal projections, the operators $K_R$ and $K_L$ are partial isometries.
The same is true of $K$. Actually, the identities
\begin{equation} \label{eq:K-PI}
\begin{aligned}
   K^\dag K & =  P_{CR}U^\dag(\1_\HH-P_R)UP_{CR} = P_{CR}-K_R^\dag K_R
   = P_{CR}-U^\dag P_RU,
   \\
   K K^\dag & = P_{LC}U(\1_\HH-P_L)U^\dag P_{LC} = P_{LC}-K_LK_L^\dag
   = P_{LC}-UP_LU^\dag,
\end{aligned}
\end{equation}
show that $K^\dag K$ and $KK^\dag$ are the orthogonal projections onto $\HH_i=\HH_{CR}\ominus\rank(U^\dag P_RU)$ and $\HH_f=\HH_{LC}\ominus\rank(UP_LU^\dag)$ respectively, which are the initial and final subspaces of the partial isometry $K$.
We now have three decompositions of $\1_\HH$ into orthogonal projections,
\[
 P_L + K^\dag K + U^\dag P_RU,
 \qquad
 P_L + P_C + P_R,
 \qquad
 UP_LU^\dag + KK^\dag + P_R,
\]
which yield three orthogonal decompositions of $\HH$,
\begin{equation} \label{eq:OD}
 \HH_L \oplus \HH_i \oplus U^\dag\HH_R,
 \qquad
 \HH_L \oplus \HH_C \oplus \HH_R,
 \qquad
 U\HH_L \oplus \HH_f \oplus \HH_R.
\end{equation}
By assumption $K$ and hence the projections $K^\dag K$ and $KK^\dag$ have the same rank as $P_C$. Therefore, there is a unitary operator $W \colon \HH_i \to\HH_C$. We can use this to define unitary operators connecting the above orthogonal decompositions, in which one piece is given by the identity, one by a restriction of $U$, and one involves $W$. Explicitly, we can define $U_{LC}$ and $U_{CR}$ by
\begin{equation} \label{eq:FPI}
\begin{aligned}
   \1_{\HH_L} \oplus U_{CR} & =
   P_L\1_\HH P_L + P_C W K^\dag K + P_R U (U^\dag P_RU) =
   P_L + WK^\dag K + P_RU,
   \\
   U_{LC} \oplus \1_{\HH_R} & =
   (UP_LU^\dag) U P_L + KK^\dag (KW^\dag) P_C + P_R\1_\HH P_R =
   UP_L + KW^\dag P_C + P_R.
\end{aligned}
\end{equation}
Unitarity of these operators is immediate because the terms in the right hand side of \eqref{eq:FPI} are partial isometries connecting the orthogonal decompositions \eqref{eq:OD}, i.e., when restricted to their initial and final subspaces they become the following kind of unitaries
\[
\begin{gathered}
 P_L \colon \HH_L \to \HH_L,
 \qquad
 W K^\dag K \colon \HH_i \to \HH_C,
 \qquad
 P_R U \colon U^\dag\HH_R \to \HH_R,
 \\
 U P_L \colon \HH_L \to U\HH_L,
 \qquad
 KW^\dag P_C \colon \HH_C \to \HH_f,
 \qquad
 P_R \colon \HH_R \to \HH_R.
\end{gathered}
\]
Furthermore, the operators $U_{LC}$ and $U_{CR}$ provide the desired factorization since
\[
 (U_{LC}\oplus\1_{\HH_R})(\1_{\HH_L}\oplus U_{CR}) =
 UP_L + KW^\dag WK^\dag K + P_RU = K_L + K + K_R = U.
\]

\noindent{\it Uniqueness:} Consider another $\HH_C$-overlapping factorization of $U$ given by the unitaries $\hat{U}_{LC}$ and $\hat{U}_{CR}$. Then, the equality
\[
 (\hat{U}_{LC}\oplus\1_{\HH_R})(\1_{\HH_L}\oplus\hat{U}_{CR}) =
 (U_{LC}\oplus\1_{\HH_R})(\1_{\HH_L}\oplus U_{CR})
\]
leads to
\[
 U_{LC}^\dag\hat{U}_{LC}\oplus\1_{\HH_R} =
 \1_{\HH_L}\oplus U_{CR}\hat{U}_{CR}^\dag.
\]
This identity is satisfied only if $U_{LC}^\dag\hat{U}_{LC}=\1_{\HH_L}\oplus U_C$ and $U_{CR}\hat{U}_{CR}^\dag=U_C\oplus\1_{\HH_R}$ for some unitary $U_C$ on $\HH_C$. Therefore, $\hat{U}_{LC}=U_{LC}(\1_{\HH_L}\oplus U_C)$ and $\hat{U}_{CR}=(U_C^\dag\oplus\1_{\HH_R})U_{CR}$, a relation which obviously always yields an $\HH_C$-overlapping factorization of $U$.

\noindent{\it Omission of rank condition:} We know that the condition $P_RUP_L=\0_\HH$ implies \eqref{eq:K-PI} so that $K$ is a partial isometry and $\rank K = \rank(K^\dag K) = \rank(KK^\dag)$. If $\dim\HH_R<\infty$, taking traces in the first equation of \eqref{eq:K-PI} gives
\[
 \rank K = \rank(K^\dag K) = \tr(K^\dag K) = \tr P_{CR} - \tr P_R = \rank P_C.
\]
From the second equation of \eqref{eq:K-PI} we obtain in a similar fashion that $\rank K = \rank P_C$ if $\dim\HH_L<\infty$.
\end{proof}

The previous characterization shows that the simplest non-trivial operators admitting overlapping factorizations are given by unitary band matrices. In particular, this links Khrushchev's formulas to the index theory of QWs and general unitary matrices \cite{indexpap, kitaev}. However, a band structure is not necessary for the existence of this kind of decomposition, as Example \ref{ex:G3.G4} illustrates.

Now we can state our main result on factorizations of Schur functions.

\begin{thm} \label{thm:GEN-K}
{\bf \small [Khrushchev's formula for overlapping factorizations of unitaries]}
\newline
Let $U = (U_{LC} \oplus \1_{\HH_R})(\1_{\HH_L} \oplus U_{CR})$ be an $\HH_C$-overlapping factorization of a unitary operator $U$ on a Hilbert space $\HH$. Then, any subspace $V = V_L \oplus \HH_C \oplus V_R$ of $\HH$, with $V_A$ a subspace of $\HH_A$, has a $U$-Schur function which factorizes as
\[
 f_V =  (\1_{V_L} \oplus f^R_{V_{CR}})(f^L_{V_{LC}} \oplus \1_{V_R}),
\]
where $f^L_{V_{LC}}$ is the $U_{LC}$-Schur function of $V_{LC} = V_L \oplus \HH_C$ and $f^R_{V_{CR}}$ is the $U_{CR}$-Schur function of $V_{CR} = \HH_C \oplus V_R$.
In particular, the $U$-Schur function of $\HH_C$ factorizes as $f_{\HH_C}=f^R_{\HH_C}f^L_{\HH_C}$.
\end{thm}

\begin{proof}
The relation given in \eqref{eq:a-U} allows us to translate results from first return generating functions to Schur functions and vice versa. For convenience, we will work with the former ones. Then, the factorization property stated in the theorem can be rewritten as
\begin{equation} \label{eq:fac-a}
 za_V(z) = (a^L_{V_{LC}}(z) \oplus z\1_{V_R}) (z\1_{V_L} \oplus a^R_{V_{CR}}(z)),
\end{equation}
where $a_V(z)=zf_V^\dag(z)$, $a^L_{V_{LC}}(z)=zf^{L\dag}_{V_{LC}}(z)$ and $a^R_{V_{CR}}(z)=zf^{R\dag}_{V_{CR}}(z)$ are respectively the first return generating functions of $V$, $V_{LC}$ and $V_{CR}$ with respect to $U$, $U_{LC}$ and $U_{CR}$.

If $P$ and $\perpo{P}=\1_\HH-P$ are the orthogonal projections of $\HH$ onto $V$ and $V^\bot$ respectively, from \eqref{eq:f-U} and \eqref{eq:a-U} we find that\begin{equation} \label{eq:GF-res-2}
 a_V(z) = \sum_{n\ge1} z^n P U (\perpo{P}U)^{n-1} P
 = z P U P
 + z^2 P U \perpo{P} \left(\1_\HH - z\perpo{P}U\perpo{P}\right)^{-1} \perpo{P} U P.
\end{equation}
Denoting by $P_A$ and $Q_A$ the orthogonal projections of $\HH$ onto $\HH_A$ and $V_A$ respectively, we can express $U=U_{LC}P_L+P_RU_{CR}P_{CR}+U_{LC}U_{CR}P_{CR}$ and $\perpo{P}=(P_L-Q_L)+(P_R-Q_R)$. This allows us to write
\[
 \perpo{P}U\perpo{P} = T_L + T_R + T,
 \qquad
 \left\{
 \begin{aligned}
  & T_L = (P_L-Q_L) U_{LC} (P_L-Q_L),
  \\
  & T_R = (P_R-Q_R) U_{CR} (P_R-Q_R),
  \\
  & T = (P_L-Q_L) U_{LC} U_{CR} (P_R-Q_R).
 \end{aligned}
 \right.
\]
Since $T_LT_R=T_RT_L=TT_L=T_RT=T^2=\0_\HH$ and $\|T_L\|,\|T_R\|,\|T\|\le1$, we get for $|z|<1$,
\[
\begin{aligned}
 & (\1_\HH - z\perpo{P}U\perpo{P})^{-1}
 = (\1_\HH - zT_L - zT_R - zT)^{-1}
 = [(\1_\HH-zT_R)(\1_\HH-zT_L) - zT]^{-1}
 \\
 & \kern15pt
 = (\1_\HH-zT_L)^{-1} [\1_\HH - (\1_\HH-zT_R)^{-1}zT(\1_\HH-zT_L)^{-1}]^{-1}
 (\1_\HH-zT_R)^{-1}
 \\
 & \kern15pt
 = (\1_\HH-zT_L)^{-1} (\1_\HH - zT)^{-1} (\1_\HH-zT_R)^{-1}
 = (\1_\HH-zT_L)^{-1} (\1_\HH + zT) (\1_\HH-zT_R)^{-1}
 \\
 & \kern15pt
 = P_L(\1_\HH-zT_L)^{-1}zT(\1_\HH-zT_R)^{-1}P_R + P_L(\1_\HH-zT_L)^{-1}P_L
 + P_R(\1_\HH-zT_R)^{-1}P_R + P_C.
\end{aligned}
\]
Inserting this identity into \eqref{eq:GF-res-2} gives
\[
\begin{aligned}
 a_V(z) & = zPUP
 + z^3 PU(P_L-Q_L) (\1_\HH-zT_L)^{-1} T (\1_\HH-zT_R)^{-1} (P_R-Q_R)UP
 \\
 & \kern-10pt
 + z^2 PU \left[ (P_L-Q_L) (\1_\HH-zT_L)^{-1} (P_L-Q_L)
 + (P_R-Q_R) (\1_\HH-zT_R)^{-1} (P_R-Q_R) \right] UP.
\end{aligned}
\]
Moreover, from the splitting $P=Q_L+P_C+Q_R$ we find that
\[
\begin{gathered}
 PUP=Q_{LC}U_{LC}Q_L + Q_RU_{CR}Q_{CR} + Q_{LC}U_{LC}U_{CR}Q_{CR},
 \\
 (P_L-Q_L)UP = (P_L-Q_L)U_{LC}Q_L + (P_L-Q_L)U_{LC}U_{CR}Q_{CR},
 \\
 PU(P_R-Q_R) = Q_RU_{CR}(P_R-Q_R) + Q_{LC}U_{LC}U_{CR}(P_R-Q_R),
 \\
 PU(P_L-Q_L) = Q_{LC}U_{LC}(P_L-Q_L),
 \quad
 (P_R-Q_R)UP = (P_R-Q_R)U_{CR}Q_{CR}.
\end{gathered}
\]
Rewriting $U_{LC}U_{CR}=U_{LC}Q_{LC}Q_{CR}U_{CR}$ in the above relations and combining them with the preceding expression of $a_V(z)$ yields
\begin{align}
 & \label{eq:factoriz} \kern100pt za_V(z) = (a^L(z)+zQ_R) (zQ_L+a^R(z)),
 \\
 & \label{eq:factores} \begin{aligned}
 	a^L(z) & = z Q_{LC} U_{LC} Q_{LC}
 	+ z^2 Q_{LC} U_{LC} (P_L-Q_L) (\1_\HH-zT_L)^{-1} (P_L-Q_L) U_{LC} Q_{LC},
 	\\
 	a^R(z) & = z Q_{CR} U_{CR} Q_{CR}
 	+ z^2 Q_{CR} U_{CR} (P_R-Q_R) (\1_\HH-zT_R)^{-1} (P_R-Q_R) U_{CR} Q_{CR}.
 \end{aligned}
\end{align}
Let $S_A$ and $\perpo{S}_A=\1_{\HH_A}-S_A$ be the orthogonal projections of $\HH_A$ onto $V_A$ and $V_A^\bot$ respectively. Considering $a^L(z)$ and $a^R(z)$ as functions with values in operators on $\HH_{LC}$ and $\HH_{CR}$ respectively amounts to substituting $Q_{LC} \to S_{LC}$, $Q_{CR} \to S_{CR}$, $P_L-Q_L=P_{LC}-Q_{LC} \to \perpo{S}_{LC}$ and $P_R-Q_R=P_{CR}-Q_{CR} \to \perpo{S}_{CR}$ in \eqref{eq:factores}, so that
\[
\begin{aligned}
 a^L(z) & = z S_{LC} U_{LC} S_{LC}
 + z^2 S_{LC} U_{LC} \perpo{S}_{LC}
 (\1_{\HH_{LC}}-z\perpo{S}_{LC}U_{LC}\perpo{S}_{LC})^{-1}
 \perpo{S}_{LC} U_{LC} S_{LC},
 \\
 a^R(z) & = z S_{CR} U_{CR} S_{CR}
 + z^2 S_{CR} U_{CR} \perpo{S}_{CR}
 (\1_{\HH_{CR}}-z\perpo{S}_{CR}U_{CR}\perpo{S}_{CR})^{-1}
 \perpo{S}_{CR} U_{CR} S_{CR}.
\end{aligned}
\]
According to \eqref{eq:GF-res-2}, this identifies $a^L=a^L_{V_{LC}}$ and $a^R=a^R_{V_{CR}}$.

Finally, since $a_V$, $a^L_{V_{LC}}$ and $a^R_{V_{CR}}$ must be actually understood as functions with values in operators on $V$, $V_{LC}$ and $V_{CR}$ respectively, we should rewrite \eqref{eq:factoriz} as \eqref{eq:fac-a}.
\end{proof}

We point out that the notions of transfer (or characteristic) operator-valued function and its factorization, starting with work of Liv\v{s}ic, Brodski\u{\i}, Potapov, Sakhnovich and others (see \cite{Ka-DT} and references therein), has played a very important role in areas of analysis that stretch all the way from linear system theory to the study of Wiener-Hopf type problems. We expect to explore connections with these topics in a future publication.

Let us illustrate the general result of Theorem \ref{thm:GEN-K} by reconsidering the overlapping factorizations of Example \ref{ex:G3.G4}.

\begin{ex} \label{ex:G3.G4-bis}
Let $U$ be the first of the unitary operators of Example \ref{ex:G3.G4}, with overlapping subspace $\HH_C=\spn\{\psi_3\}$. Using \eqref{eq:f-U} we can compute the Schur function $f_V$ for any subspace $V = V_L \oplus \HH_C \oplus V_R$ with $V_L \subset \HH_L=\spn\{\psi_1,\psi_2\}$ and $V_R \subset \HH_R=\spn\{\psi_4,\psi_5,\psi_6\}$. We present below two of these results,
\[
\begin{aligned}
 & V_L=V_R=\{0\}, & \quad & V=\spn\{\psi_3\}, & \quad &
 f_{\psi_3}(z) = \frac{2z-1}{2-z} \frac{3z-1}{3-z},
 \\[3pt]
 & V_L=\{0\}, V_R=\spn\{\psi_4\}, & & V=\spn\{\psi_3,\psi_4\}, & &
 f_V(z) =
 \begin{pmatrix}
 \frac{z-1}{2} & \frac{z+1}{2} \\[3pt] \frac{z+1}{2} & \frac{z-1}{2}
 \end{pmatrix}
 \begin{pmatrix}
 \frac{3z-1}{3-z} & 0 \\[3pt] 0 & 1
 \end{pmatrix}.
\end{aligned}
\]
In the second case the matrix representation of $f_V$ is calculated in the basis $\{\psi_3,\psi_4\}$. The computation of the Schur functions with respect to $U_{LC}$ and $U_{CR}$ shows that the first factorization is $f_{\psi_3} = f_{\psi_3}^R f_{\psi_3}^L$, while the second one corresponds to $f_V = f_V^R \, (f_{\psi_3}^L \oplus 1)$.

Assume now that $U$ is the second operator of Example \ref{ex:G3.G4}. A Schur factorization appears for any subspace $V = V_L \oplus \HH_C \oplus V_R$, with $\HH_C=\spn\{\psi_2,\psi_3\}$, $V_L \subset \HH_L=\spn\{\psi_1\}$ and $V_R \subset \HH_R=\spn\{\psi_3,\psi_4,\psi_5\}$. For instance, when $V=\HH_C$, similar matrix computations using \eqref{eq:f-U} and the basis $\{\psi_2,\psi_3\}$ give
\[
 f_{\HH_C}(z) =
 \begin{pmatrix}
 \frac{z-1}{2} & \frac{z+1}{2} \\[3pt] \frac{z+1}{2} & \frac{z-1}{2}
 \end{pmatrix}
 \begin{pmatrix}
 \frac{z-1}{z+3} & 2\frac{z+1}{z+3} \\[3pt] 2\frac{z+1}{z+3} & \frac{z-1}{z+3}
 \end{pmatrix},
\]
which corresponds to the factorization $f_{\HH_C}=f_{\HH_C}^Rf_{\HH_C}^L$.
\hfill$\scriptstyle\blacksquare$
\end{ex}

As we will see, the standard Khrushchev formula, as well as new matrix versions of this one, can be derived as mere applications of the previous theorem, which therefore can be viewed as an abstract operator generalization of Khrushchev's formula. The standard operator approach to the scalar Khrushchev formula uses a different strategy: the factorization of a Schur function comes from a decoupling of a unitary operator as a direct sum of other two ones by introducing a perturbation. In contrast, our approach avoids the need for perturbations by allowing for a factorization with an overlapping subspace instead of enforcing a strict decoupling. It is precisely the absence of perturbations that makes the new proof simpler and easier to generalize.

The Schur factorization of Theorem \ref{thm:GEN-K} is a direct consequence of the Taylor expansion for the Schur function of a spectral measure, together with the simple assumption of an overlapping factorization for unitaries. Both, the proof of such an expansion and the proof of the preceding theorem are surprisingly simple and are of purely algebraic nature. Nevertheless, to guess the operator expression for the Taylor coefficients of a Schur function as well as the statement of the above operator generalization of Khrushchev's formula is something that, at least to us, would have been unimaginable without the guide provided by the physical ideas behind quantum recurrence, and the path counting approach to Khrushchev's formula which is illustrated below.

\section{Khrushchev's formula for unitaries via path counting}
\label{sec:GEN-PC}

Although not as efficient as the formal proof of Theorem \ref{thm:GEN-K}, path counting allows us to ``get our hands on" the ``miracle" of Khrushchev's formula and to improve the understanding of its dynamical meaning. The path counting approach requires the choice of an orthonormal basis $\{\psi_j\}$ of $\HH$ obtained by enlarging a basis of the subspace $V$. First, we will identify a unitary operator $U$ on $\HH$ with its matrix representation $U=(U_{j,k})$ in the basis $\{\psi_j\}$. Then, we introduce a diagrammatic representation of $U$ as a discrete time quantum evolution operator: each coefficient $U_{j,k}=\langle \psi_j|U\psi_k\rangle$ will be understood as the amplitude of the one-step transition from $\psi_k$ to $\psi_j$, and it will be pictorially represented as
\begin{equation} \label{tz:1step}
\begin{tikzpicture}[baseline=(current bounding box.center),->,>=stealth',shorten >=1pt,auto,semithick]
  \tikzstyle{every state}=[fill=blue,draw=none,text=white,inner sep=0pt,minimum size=0.8cm]

  \node (V) {\large $\psi_j \ot \psi_k \kern5pt \equiv$};
  \node[state] (j) [right of=V,node distance=1.8cm]{$\boldsymbol{j}$};
  \node[state] (k) [right of=j,node distance=2cm] {$\boldsymbol{k}$};

  \path (k) edge	node {$U_{j,k}$} (j);

\end{tikzpicture}
\end{equation}

The natural tool for this approach is the first return generating function $a_V(z)=zf_V^\dag(z)$, rather than the Schur function $f_V(z)$ itself. Bearing in mind the expression of the Taylor coefficient $a_{V,n}$ given in \eqref{eq:a-U}, its matrix coefficients read as
\begin{equation} \label{eq:PATH}
 \langle\psi_j|a_{V,n}\psi_k \rangle = \sum_{\psi_{k_i}\notin V}
 U_{j,k_{n-1}}U_{k_{n-1},k_{n-2}}\cdots\:U_{k_2,k_1}U_{k_1,k}.
\end{equation}
Resorting again to quantum nomenclature, we will name the quantity
\[
 \ms{A}(\psi_j \ot \psi_{k_{n-1}} \ot \cdots \ot \psi_{k_1} \ot \psi_k) :=
 U_{j,k_{n-1}}U_{k_{n-1},k_{n-2}}\cdots\:U_{k_2,k_1}U_{k_1,k}
\]
the amplitude of the path $\Lambda \equiv \psi_j \ot \psi_{k_{n-1}} \ot \cdots \ot \psi_{k_1} \ot \psi_k$ of length $n$. The sum in \eqref{eq:PATH} is over the amplitudes $\ms{A}(\Lambda)$ of all the $n$-step paths $\Lambda$ connecting the vectors $\psi_k$ and $\psi_j$ with no vector from the subspace $V$ appearing at the intermediate steps. Since each amplitude $\ms{A}(\Lambda)$ appearing in $\langle\psi_j|a_V(z)\psi_k \rangle$ carries a factor $z^{\text{length}(\Lambda)}$, we conclude that the matrix coefficients of the first return generating function can be expressed in terms of path amplitudes as $\langle\psi_j|a_V(z)\psi_k\rangle = \ms{A}_{j,k}^V(z)$, where
\begin{equation} \label{eq:paths}
 \ms{A}_{j,k}^V(z) :=
 \kern-5pt
 \sum_{\Lambda \in \powerset_{\!j,k}(V)}
 \kern-10pt
 \ms{A}(\Lambda) \, z^{\text{length}(\Lambda)},
 \qquad
 \powerset_{\!j,k}(V) = \text{ \parbox{170pt}
 {set of all paths from $\psi_k$ to $\psi_j$
 \break
 avoiding $V$ at intermediate steps.}}
\end{equation}
That is, we get the generating function $a_V(z)$ by adding the amplitudes of all the paths which start at a basis vector of $V$ and return to $V$ only at the last step, with a factor $z$ accompanying each single step along the paths. From this point of view, equality \eqref{eq:fH} has an obvious meaning: If $V=\HH$ is the whole Hilbert space, any vector returns trivially to the subspace in one step, thus $\powerset_{\!j,k}(\HH)$ is simply the one-step path from $\psi_k$ to $\psi_j$ and
\begin{equation} \label{eq:aH}
 a_\HH(z) = zU.
\end{equation}

Identity \eqref{eq:paths}, together with the relation  $a_V(z)=zf_V^\dag(z)$ from \eqref{eq:a-U}, can be used to get a closer look at Khrushchev's formula using simple path counting methods. These methods are based on a splitting of the generating function $a_V(z)$ which is different from the splitting based on the length of the loops leading to the Taylor expansion. This can be illustrated for instance with any of the unitaries given in Example \ref{ex:G3.G4}, but for simplicity we will use a model with fewer transitions.

\begin{ex} \label{ex:PC}
Consider a unitary operator $U$ whose matrix representation with respect to some orthonormal basis $\{\psi_j\}$ has the form
\[
 \left(\begin{smallmatrix}
 \frac{1}{2}&\frac{-1}{2}&\frac{1}{2}&\frac{1}{2}&0&0
 \\
 \frac{1}{\sqrt{2}}&\frac{1}{\sqrt{2}}&0&0&0&0
 \\
 \frac{-1}{2}&\frac{1}{2}&\frac{1}{2}&\frac{1}{2}&0&0
 \\
 0&0&\frac{1}{2}&\frac{-1}{2}&\frac{1}{2}&\frac{1}{2}
 \\
 0&0&0&0&\frac{1}{\sqrt{2}}&\frac{-1}{\sqrt{2}}
 \\
 0&0&\frac{-1}{2}&\frac{1}{2}&\frac{1}{2}&\frac{1}{2}
 \end{smallmatrix}\right)
 = \text{\footnotesize
 $\left(\begin{array}{@{\hspace{1pt}}c@{\hspace{4pt}} c@{\hspace{6pt}} c@{\hspace{7pt}} c@{\hspace{9pt}} c@{\hspace{7pt}} c@{\hspace{2pt}}}
 a&-a&bd&bd&0&0
 \\
 b&b&0&0&0&0
 \\
 -a&a&bd&bd&0&0
 \\
 0&0&c&-c&c&c
 \\
 0&0&0&0&d&-d
 \\
 0&0&-c&c&c&c
 \end{array}\right)$}
 = \text{\footnotesize
 $\left(\begin{array}{@{\hspace{1pt}}c@{\hspace{5pt}} c@{\hspace{8pt}} c | c c c@{\hspace{4pt}}}
 a&-a&b&0&0&0
 \\
 b&b&0&0&0&0
 \\
 -a&a&b&0&0&0
 \\ \hline
 0&0&0&1&0&0
 \\
 0&0&0&0&1&0
 \\
 0&0&0&0&0&1
 \end{array}\right)
 \left(\begin{array}{@{\hspace{3pt}}c c |@{\hspace{2pt}} c@{\hspace{5pt}} c@{\hspace{8pt}} c@{\hspace{7pt}} c@{\hspace{1pt}}}
 1&0&0&0&0&0
 \\
 0&1&0&0&0&0
 \\ \hline
 0&0&d&d&0&0
 \\
 0&0&c&-c&c&c
 \\
 0&0&0&0&d&-d
 \\
 0&0&-c&c&c&c
 \end{array}\right)$},
\]
and its overlapping factorization given above, where $a=1/2=c$ and $b=1/\sqrt{2}=d$. The overlapping subspace is $\HH_C=\spn\{\psi_3\}$, while the left and right subspaces are $\HH_L=\spn\{\psi_1,\psi_2\}$ and $\HH_R=\spn\{\psi_4,\psi_5,\psi_6\}$ respectively.
Following \eqref{tz:1step}, the operator $U$, as well as the left and right operators $U_{LC}$ and $U_{CR}$ implicit in the above overlapping factorization, can be depicted respectively by the diagrams
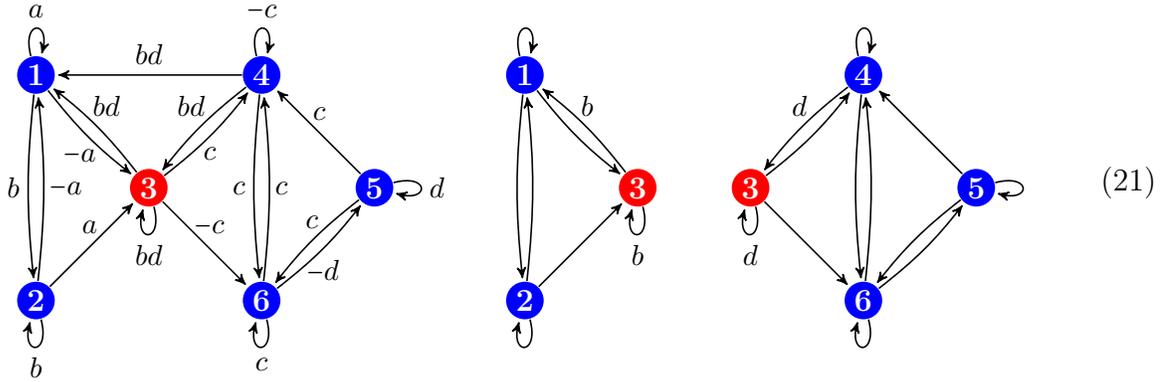
\begin{equation} \label{tz:example}
\begin{tikzpicture}[baseline=(current bounding box.center),->,>=stealth',shorten >=1pt,auto,semithick]
  \tikzstyle{every state}=[node distance=3cm,fill=blue,draw=none,text=white,inner sep=0pt,minimum size=0.5cm]

  \node[state] (1) {$\boldsymbol 1$};
  \node[state] (4) [right of=1] {$\boldsymbol 4$};

  \node[state] (3) at (1.5,-1.5) [fill=red] {$\boldsymbol 3$};
  \node[state] (5) [right of=3] {$\boldsymbol 5$};

  \node[state] (2) at (0,-3) {$\boldsymbol 2$};
  \node[state] (6) [right of=2] {$\boldsymbol 6$};

  \path (1)  edge [loop above]			node {\small$a$} (1)
             edge [bend right=6,left]	node {\small$b$\kern-1pt} (2)
             edge [bend right=7,below]	node {\small$-a$\kern6pt} (3)
        (2)  edge [bend right=6,right]	node {\kern-2pt\small$-a$} (1)
             edge [loop below]			node {\small$b$} (2)
             edge [above]               node {\small$a$\kern2pt} (3)
        (3)  edge [bend right=7,above]	node {\kern7pt\small$bd$} (1)
             edge [loop below]			node {\small$bd$} (3)
             edge [bend right=7,below]	node {\small$c$} (4)
             edge [above]               node {\kern3pt\small$-c$} (6)
        (4)  edge [above]	            node {\small$bd$} (1)
             edge [bend right=7,above]	node {\small$bd$\kern7pt} (3)
             edge [loop above]			node {\small$-c$} (4)
             edge [bend right=6,left] 	node {\small$c$\kern-1pt} (6)
        (5)  edge [above]	            node {\kern1pt\small$c$} (4)
             edge [loop right]			node {\kern-1pt\small$d$} (5)
             edge [bend right=7,above]	node {\small$c$\kern1pt} (6)
        (6)  edge [bend right=6,right]	node {\kern-2pt\small$c$} (4)
             edge [bend right=7,below]	node {\small$-d$} (5)
             edge [loop below]			node {\small$c$} (6);

  \node[state] (1L) at (6.5,0) {$\boldsymbol 1$};
  \node[state] (2L) [below of=1L] {$\boldsymbol 2$};
  \node[state] (3L) at (8,-1.5) [fill=red] {$\boldsymbol 3$};

  \path (1L)  edge [loop above]			node {} (1L)
              edge [bend right=6,left]	node {} (2L)
              edge [bend right=7,below]	node {} (3L)
        (2L)  edge [bend right=6,right]	node {} (1L)
              edge [loop below]			node {} (2L)
              edge [above]              node {} (3L)
        (3L)  edge [bend right=7,above]	node {\kern1pt\small$b$} (1L)
              edge [loop below]			node {\small$b$} (3L);

  \node[state] (4R) at (11,0) {$\boldsymbol 4$};
  \node[state] (6R) [below of=4R] {$\boldsymbol 6$};
  \node[state] (3R) at (9.5,-1.5) [fill=red] {$\boldsymbol 3$};
  \node[state] (5R) [right of=3R] {$\boldsymbol 5$};

  \path (3R)  edge [loop below]			node {\small$d$} (3R)
              edge [bend right=7,below]	node {} (4R)
              edge [above]              node {} (6R)
        (4R)  edge [bend right=7,above]	node {\small$d$\kern2pt} (3R)
              edge [loop above]			node {} (4R)
              edge [bend right=6,left] 	node {} (6R)
        (5R)  edge [above]	            node {} (4R)
              edge [loop right]			node {} (5R)
              edge [bend right=7,above]	node {} (6R)
        (6R)  edge [bend right=6,right]	node {} (4R)
              edge [bend right=7,below]	node {} (5R)
              edge [loop below]			node {} (6R);
\end{tikzpicture}
\end{equation}
The overlapping subspace is highlighted in red and, for convenience, the diagrams of $U_{LC}$ and $U_{CR}$ only show explicitly the amplitudes which differ from the analogous ones in the diagram of $U$. Note the absence of one-step transitions from $\HH_L$ to $\HH_R$ in the diagram of $U$, which is the diagrammatic meaning of the condition $P_RUP_L=\0_\HH$ in \eqref{eq:charac-overlap} characterizing the existence of an $\HH_C$-overlapping factorization in finite dimension.

We can now understand the factorization of $f_{\psi_3}$ by path counting arguments applied to $a_{\psi_3}$. Choosing $V=\spn\{\psi_3\}$ and $V_L = V_R =\{0\}$ in the diagram of $U$, we can identify $a_{\psi_3}=\ms{A}_{3,3}^{\psi_3}$ as a sum like  \eqref{eq:paths} over the loops $\powerset_{\!3,3}(\psi_3)$ which hit $\psi_3$ only at the beginning and the end.
We can split the loops contained in $\powerset_{\!3,3}(\psi_3)$ into three classes: loops avoiding the left and right subspaces all together, loops which pass only through the right or the left subspace and loops which pass through both subspaces. This yields the following decomposition of $a_{\psi_3}$
\begin{equation} \label{eq:ex-fac}
\begin{aligned}
 a_{\psi_3}(z) & = bdz
 + \ms{A}_{3,1}^{\psi_3}(z) \; bdz
 + bdz \; \ms{A}_{4,3}^{\psi_3}(z)
 + \ms{A}_{3,1}^{\psi_3}(z) \; bdz \; \ms{A}_{4,3}^{\psi_3}(z),
 \\
 & = z^{-1} \left(bz + \ms{A}_{3,1}^{\psi_3}(z) \; bz\right)
 \left(dz + dz \; \ms{A}_{4,3}^{\psi_3}(z)\right).
\end{aligned}
\end{equation}
For example, every loop $\Lambda\in\powerset_{\!3,3}(\psi_3)$ hitting only the left subspace must start with the step $\psi_1 \ot \psi_3$, which contributes with a factor $bdz$ to the corresponding term $\ms{A}(\Lambda)\,z^{\text{length}(\Lambda)}$ of $\ms{A}_{3,3}^{\psi_3}(z)$. The contribution of the rest of the steps $\psi_3 \ot \cdots \ot \psi_1$ of this kind of loops is accounted for in $\ms{A}_{3,1}^{\psi_3}(z)$.

Note that the absence of one-step transitions from the left to the right subspace implies that the loops of $\powerset_{\!3,3}(\psi_3)$ hitting both subspaces must start going from $\psi_3$ to the right, crossing to the left at some step and then remaining there until the return to $\psi_3$. This leads to a relation between the left, right and left-right contributions to $a_{\psi_3}$, which ends up in the factorization \eqref{eq:ex-fac}.

A similar analysis for the diagrams of the left and right operators shows that the corresponding generating functions can be expressed as
\[
 a_{\psi_3}^L(z)=bz+\ms{A}_{3,1}^{\psi_3}(z)\;bz,
 \qquad
 a_{\psi_3}^R(z)=dz+dz\;\ms{A}_{4,3}^{\psi_3}(z).
\]
Here we have taken into account that the paths and amplitudes of the initial diagram involved in $\ms{A}_{3,1}^{\psi_3}$ and $\ms{A}_{4,3}^{\psi_3}$ coincide with those of the left and right diagrams respectively because no path can cross from the left subspace to the right one without hitting $\psi_3$. Therefore, \eqref{eq:ex-fac} can be rewritten as
\[
 za_{\psi_3}(z) = a_{\psi_3}^L(z) a_{\psi_3}^R(z),
\]
which, in view of relation \eqref{eq:a-U} between first return generating functions and Schur functions, gives the desired factorization
\[
 f_{\psi_3} = f_{\psi_3}^R f_{\psi_3}^L.
\]
A direct calculation using \eqref{eq:f-U} yields
\begin{equation} \label{eq:ex-1}
f_{\psi_3}^L(z) =
\frac{2z^2-(1+\sqrt{2})z+\sqrt{2}}
{\sqrt{2}z^2-(1+\sqrt{2})z+2},
\qquad
f_{\psi_3}^R(z) =
\frac{2^{\frac{3}{2}}z^3 - 2z^{2} - (\sqrt{2}-1)z + 2}
{2z^3-(\sqrt{2}-1)z^2-2z+2^{\frac{3}{2}}}.
\end{equation}

Path counting arguments also work in factorizing matrix Schur functions. Consider for example the subspace $V=\spn\{\psi_3,\psi_5\}$. Since $\langle\psi_j|a_V(z)\psi_k\rangle=\ms{A}_{j,k}^V(z)$, a splitting of the paths $\powerset_{\!j,k}(V)$, $j,k\in\{3,5\}$, for the initial diagram, similar to that one performed in the previous case, leads to
\[
\text{\small$\begin{aligned}
 & \langle\psi_3|a_V(z)\psi_3\rangle = bdz
 + \ms{A}_{3,1}^{\psi_3}(z) \; bdz
 + bdz \; \ms{A}_{4,3}^V(z)
 + \ms{A}_{3,1}^{\psi_3}(z) \; bdz \; \ms{A}_{4,3}^V(z),
 & \kern5pt & \langle\psi_5|a_V(z)\psi_3\rangle = \ms{A}_{5,3}^V(z),
 \\
 & \langle\psi_3|a_V(z)\psi_5\rangle = bdz \; \ms{A}_{4,5}^V(z)
 + \ms{A}_{3,1}^{\psi_3}(z) \; bdz \; \ms{A}_{4,3}^V(z),
 & & \langle\psi_5|a_V(z)\psi_5\rangle = \ms{A}_{5,5}^V(z),
\end{aligned}$}
\]
where once again $\ms{A}_{3,1}^{\psi_3}$ and $\ms{A}_{4,3}^V$, $\ms{A}_{5,3}^V$, $\ms{A}_{4,5}^V$, $\ms{A}_{5,5}^V$ coincide with those computed with the left and right diagrams respectively. Therefore, identifying the first return generating functions $a_V$ and $a_V^R$ with their representations in the basis $\{\psi_3,\psi_5\}$, we obtain
\[
\text{\small$\begin{aligned}
 za_V(z) & =
 \begin{pmatrix}
 (bz+\ms{A}_{3,1}^{\psi_3}(z)\,bz)(dz+dz\,\ms{A}_{4,3}^V(z))
 & (bz+\ms{A}_{3,1}^{\psi_3}(z)\,bz)\;dz\,\ms{A}_{4,5}^V(z)
 \\[3pt]
 z\,\ms{A}_{5,3}^V(z) & z\,\ms{A}_{5,5}^V(z)
 \end{pmatrix}
 \\
 & = \begin{pmatrix}
 bz+\ms{A}_{3,1}^{\psi_3}(z)\,bz
 & 0
 \\[3pt]
 0 & z
 \end{pmatrix}
 \begin{pmatrix}
 dz+dz\,\ms{A}_{4,3}^V(z)
 & dz\,\ms{A}_{4,5}^V(z)
 \\[3pt]
 \ms{A}_{5,3}^V(z) & \ms{A}_{5,5}^V(z)
 \end{pmatrix}
 = (a_{\psi_3}^L(z) \oplus z)\,a_V^R(z).
\end{aligned}$}
\]
This leads to
\[
 f_V=f_V^R\,(f_{\psi_3}^L\oplus1),
\]
where $f_{\psi_3}^L$ is given in \eqref{eq:ex-1}, while using \eqref{eq:f-U} again yields
\[
 f_V^R(z) = \frac{1}{\sqrt{2}(z^2-2)}
 \begin{pmatrix}
 2z^2-z-2 & -z
 \\
 -z & 2z^2+z-2
 \end{pmatrix}.
\]

Note that not only $P_RUP_L$ vanishes, but so does $\hat{P}_R U \hat{P}_L$, where $\hat{P}_L$ and $\hat{P}_R$ are the orthogonal projections onto $\hat{\HH}_L=\spn\{\psi_5,\psi_6\}$ and $\hat{\HH}_R=\spn\{\psi_1,\psi_2,\psi_3\}$ respectively. This means that $U$ has another overlapping factorization with left and right subspaces $\hat{\HH}_L$ and $\hat{\HH}_R$, and overlapping subspace $\hat{\HH}_C=\spn\{\psi_4\}$. Indeed, using the previous values of $a$, $b$, $c$, $d$,
\[
 \left(\begin{smallmatrix}
 \frac{1}{2}&\frac{-1}{2}&\frac{1}{2}&\frac{1}{2}&0&0
 \\
 \frac{1}{\sqrt{2}}&\frac{1}{\sqrt{2}}&0&0&0&0
 \\
 \frac{-1}{2}&\frac{1}{2}&\frac{1}{2}&\frac{1}{2}&0&0
 \\
 0&0&\frac{1}{2}&\frac{-1}{2}&\frac{1}{2}&\frac{1}{2}
 \\
 0&0&0&0&\frac{1}{\sqrt{2}}&\frac{-1}{\sqrt{2}}
 \\
 0&0&\frac{-1}{2}&\frac{1}{2}&\frac{1}{2}&\frac{1}{2}
 \end{smallmatrix}\right)
 = \text{\footnotesize
 $\left(\begin{array}{@{\hspace{1pt}}c@{\hspace{4pt}} c@{\hspace{2pt}} c@{\hspace{1pt}} c@{\hspace{5pt}} c@{\hspace{6pt}} c@{\hspace{2pt}}}
 a&-a&a&a&0&0
 \\
 b&b&0&0&0&0
 \\
 -a&a&a&a&0&0
 \\
 0&0&bd&-bd&c&c
 \\
 0&0&0&0&d&-d
 \\
 0&0&-bd&bd&c&c
 \end{array}\right)$}
 = \text{\footnotesize
 $\left(\begin{array}{@{\hspace{3pt}} c@{\hspace{10pt}} c@{\hspace{10pt}} c | @{\hspace{2pt}} c@{\hspace{7pt}} c@{\hspace{7pt}} c@{\hspace{2pt}}}
 1&0&0&0&0&0
 \\
 0&1&0&0&0&0
 \\
 0&0&1&0&0&0
 \\ \hline
 0&0&0&d&c&c
 \\
 0&0&0&0&d&-d
 \\
 0&0&0&-d&c&c
 \end{array}\right)
 \left(\begin{array}{@{\hspace{1pt}} c@{\hspace{4pt}} c@{\hspace{7pt}} c@{\hspace{7pt}} c@{\hspace{3pt}} | @{\hspace{5pt}} c@{\hspace{9pt}} c@{\hspace{2pt}}}
 a&-a&a&a&0&0
 \\
 b&b&0&0&0&0
 \\
 -a&a&a&a&0&0
 \\
 0&0&b&-b&0&0
 \\ \hline
 0&0&0&0&1&0
 \\
 0&0&0&0&0&1
 \end{array}\right)$}.
\]
This will generate Khrushchev's formulas for subspaces containing $\psi_4$.
\hfill$\scriptstyle\blacksquare$
\end{ex}

The previous example illustrates the dynamical meaning of the factorization obtained in Theorem \ref{thm:GEN-K}, which has its origin in the fact that the overlapping subspace $\HH_C$ blocks all the possible ways of passing from the left to the right subspace. This allows an easy control of the paths contributing to $a_V(z)$ since they cannot cross from the left to the right because they must avoid $V = V_L \oplus \HH_C \oplus V_R$. These paths, which return to $V$ avoiding it at intermediate steps, either finish after one step in $V$, or visit only one of the left and right subspaces, or move to the right and once crossing over to the left they remain there until their return to $V$. As a consequence, the return properties of the left and right subsystems control the return properties of the whole system, leading to the factorization of first return generating functions. Thus, the operator generalization of Khrushchev's formula provided by Theorem \ref{thm:GEN-K} acquires a quantum dynamical meaning: it codifies the splitting of a first return generating function caused by the fact that the subspace of return $V$ blocks the transitions from a subspace of $\HH \ominus V$ to its orthogonal complement.

\section{Matrix Khrushchev's formulas from CMV matrices}
\label{sec:CMV-K}

Let us start summarizing the state of the art concerning the generalization to the matrix-valued setting of the ideas involved in Khrushchev's formula (see \cite{DaPuSi-MOPUC} and references therein). The Schur algorithm has a version for Schur functions $f$ with values in $d \times d$ matrices,
\begin{equation} \label{eq:matrix-SA}
\begin{aligned}
 & f_0(z)=f(z),
 \\
 & f_{j+1}(z) =
 z^{-1}(\rho_j^R)^{-1} [f_j(z)-\alpha_j] [\1-\alpha_j^\dag f_j(z)]^{-1} \rho_j^L,
 \qquad
 \alpha_j=f_j(0),
 \qquad
 j \ge 0,
 \end{aligned}
\end{equation}
where $\1$ stands for the $d \times d$ unit matrix, the matrix Schur parameters $\alpha_j$ satisfy $\|\alpha_j\|\le1$ and $\rho_j^{L,R}$ are the non-negative square roots
\begin{equation} \label{eq:rhoLR}
 \rho_j^L = (\1-\alpha_j^\dag\alpha_j)^{1/2},
 \qquad
 \rho_j^R = (\1-\alpha_j\alpha_j^\dag)^{1/2}.
\end{equation}
The iterates $f_j$ are again Schur functions with Schur parameters $(\alpha_j,\alpha_{j+1},\alpha_{j+2},\dots)$, while the inverse iterates $b_j$ are defined in the matrix case as the Schur functions with Schur parameters $(-\alpha_{j-1}^\dag,-\alpha_{j-2}^\dag,\dots,-\alpha_1^\dag,\1)$.

The matrix Schur algorithm terminates if $\|\alpha_j\|=1$ for some $j$, which holds exactly when the related matrix measure $\mu$ on $\T$ satifies $\int \phi(z)^\dag\,d\mu(z)\,\phi(z)=0$ for some non null $d$-vector polynomial $\phi(z)$. Otherwise $\mu$ is called a non-trivial measure. Schur functions corresponding to non-trivial matrix measures are hence characterized by infinite sequences of matrix Schur parameters satisfying $\|\alpha_j\|<1$. These parameters at the same time provide the matrix coefficients of the recurrence relation for the orthogonal polynomials with respect to the measure (see below), thus they are also known as the matrix Verblunsky coefficients of the measure.

Any matrix measure $\mu$ defines the left- and right-module sesquilinear functions
\begin{equation} \label{eq:inner-prod}
 \LL g|h \RR_L = \int h(z) \, d\mu(z) \, g(z)^\dag,
 \qquad
 \LL g|h \RR_R = \int g(z)^\dag \, d\mu(z) \, h(z),
\end{equation}
in a Hilbert space of $d \times d$-matrix functions with the matrix Laurent polynomials as a dense subset. When $\mu$ is non-trivial, the application of the Gram-Schmidt process to $\{\1,\1z,\1z^2,\dots\}$ with respect to $\LL\cdot\,|\,\cdot\RR_{L,R}$ yields an infinite sequence of left and right orthonormal polynomials $\varphi_j^{L,R}$ respectively, i.e. satisfying $\LL\varphi_j^L|\varphi_k^L\RR_L = \LL\varphi_j^R|\varphi_k^R\RR_R = \delta_{j,k}\1$. The freedom in the choice of the orthonormal polynomials can be fixed by the conditions $\kappa_{j+1}^L(\kappa_j^L)^{-1}>\0$ and $(\kappa_j^R)^{-1}\kappa_{j+1}^R>\0$ on their leading coefficients $\kappa_j^{L,R}$, where $\0$ stands for the $d \times d$ null matrix. We will assume this choice in what follows.

Finitely supported matrix measures never fall into the class of non-trivial measures but, except for the scalar case, they do not exhaust all the measures outside of this class. We will deal eventually with finitely supported matrix measures on $\T$ having only a finite number of Schur parameters and orthonormal polynomials. Nevertheless, for convenience, we will in general assume that the measure $\mu$ on $\T$ is non-trivial unless it is explicitly stated otherwise.

Bearing in mind the expression for the right and left ``inner products" \eqref{eq:inner-prod}, we should expect a matrix version of Khrushchev's formula relating the Schur functions of the measures $\varphi_j^L\,d\mu\,(\varphi_j^L)^\dag$ and $(\varphi_j^R)^\dag d\mu\,\varphi_j^R$ to the iterate $f_j$ and the inverse iterate $b_j$ of the Schur function $f$ of $\mu$. However, to the best of our knowledge no such relation has been obtained prior to this paper.
Theorem \ref{thm:GEN-K} will be the key to obtain such a matrix Khrushchev formula. For this purpose we will use the so called CMV basis $\chi_j$ and $x_j$ given by
\begin{equation} \label{eq:OP-OLP}
 \begin{aligned}	
	& \chi_{2k}(z) = z^k\varphi_{2k}^{L,\dag}(z^{-1}),
 	& \qquad & x_{2k}(z) = z^{-k}\varphi_{2k}^R(z),
	\\	
	& \chi_{2k+1}(z) = z^{-k}\varphi_{2k+1}^R(z),
 	& & x_{2k+1}(z) = z^k\varphi_{2k+1}^{L,\dag}(z^{-1}).
 \end{aligned}
\end{equation}
These are right orthonormal matrix Laurent polynomials, i.e. $\LL\chi_j|\chi_k\RR_R = \LL x_j|x_k\RR_R = \delta_{j,k}\1$. The analogous left orthonormal Laurent polynomials are not independent of these ones, but they are given by $\chi_j^\dag(z^{-1})$ and $x_j^\dag(z^{-1})$. From \eqref{eq:OP-OLP} we find that
\begin{equation} \label{eq:spec-measure-Vj}
 \chi_j^\dag\,d\mu\,\chi_j = \begin{cases}
 \varphi_j^L\,d\mu\,(\varphi_j^L)^\dag,
 \\
 (\varphi_j^R)^\dag d\mu\,\varphi_j^R,
 \end{cases}
 \qquad
 x_j^\dag\,d\mu\,x_j = \begin{cases}
 (\varphi_j^R)^\dag d\mu\,\varphi_j^R & \qquad \text{ even } j,
 \\
 \varphi_j^L\,d\mu\,(\varphi_j^L)^\dag & \qquad \text{ odd } j.
 \end{cases}
\end{equation}
Therefore, the matrix Khrushchev formula that we are searching for is indeed a statement about the Schur functions of the matrix measures $\chi_j^\dag\,d\mu\,\chi_j$ and $x_j^\dag\,d\mu\,x_j$.

The advantage of the CMV basis relies on the simplicity of their recurrences, given by
\begin{equation} \label{eq:REC-CMV}
 z\chi(z)=\chi(z)\,\mc{C},
 \quad
 \chi=(\chi_0,\chi_1,\chi_2,\dots);
 \qquad
 zx(z)=x(z)\,\hat{\mc{C}},
 \quad
 x=(x_0,x_1,x_2,\dots),
\end{equation}
where $\mc{C}=\mc{L}\mc{M}$ and $\hat{\mc{C}}=\mc{M}\mc{L}$, known as block CMV matrices, are unitary band matrices generated by the unitary factors
\begin{equation} \label{eq:LM}
 \begin{aligned}
 	& \mc{L} = \Theta(\alpha_0) \oplus \Theta(\alpha_2) \oplus \Theta(\alpha_4)
	\oplus \cdots,
	\\
 	& \mc{M} = \1 \oplus \Theta(\alpha_1) \oplus \Theta(\alpha_3) \oplus \cdots,
 \end{aligned}
 \qquad
 \Theta(\alpha_j) =
 \begin{pmatrix}
 	\alpha_j^\dag & \rho_j^L
	\\ \noalign{\vskip3pt}
	\rho_j^R & -\alpha_j
 \end{pmatrix}.
\end{equation}
Here $\alpha_j$ are the Verblunsky coefficients of $\mu$ and $\rho_j^{L,R}$ are defined as in \eqref{eq:rhoLR}.
More explicitly,
\begin{align} \label{eq:CMV}
 & \kern-5pt
 \mc{C} =
 \begin{pmatrix}
 	\alpha_0^\dag & \zeta_1^L & \sigma_1^L
	\\ \noalign{\vskip2pt}
	\rho_0^R & \omega_1^L & \eta_1^L
	\\ \noalign{\vskip2pt}
	& \zeta_2^R & \omega_2^R & \zeta_3^L & \sigma_3^L
	\\ \noalign{\vskip2pt}
	& \sigma_2^R & \eta_2^R & \omega_3^L & \eta_3^L
	\\ \noalign{\vskip2pt}
	& & & \zeta_4^R & \omega_4^R & \zeta_5^L & \dots
	\\ \noalign{\vskip2pt}
	& & & \sigma_4^R & \eta_4^R & \omega_5^L & \dots
	\\
	& & & & & \dots & \dots
 \end{pmatrix},
 \kern20pt
 \hat{\mc{C}} =
 \begin{pmatrix}
 	\alpha_0^\dag & \rho_0^L
	\\ \noalign{\vskip2pt}
	\zeta_1^R & \omega_1^R & \zeta_2^L & \sigma_2^L
	\\ \noalign{\vskip2pt}
	\sigma_1^R & \eta_1^R & \omega_2^L & \eta_2^L
	\\ \noalign{\vskip2pt}
	& & \zeta_3^R & \omega_3^R & \zeta_4^L & \sigma_4^L
	\\ \noalign{\vskip2pt}
	& & \sigma_3^R & \eta_3^R & \omega_4^L & \eta_4^L
	\\ \noalign{\vskip2pt}
	& & & & \zeta_5^R & \omega_5^R & \dots
	\\
	& & & & \dots & \dots & \dots
 \end{pmatrix},
 \\
 \label{eq:OZES}
 & \kern40pt \begin{aligned}
 	& \omega_j^L = -\alpha_{j-1}\alpha_j^\dag,
	& & \quad \zeta_j^L = \rho_{j-1}^L\alpha_j^\dag,
	& & \quad \eta_j^L = -\alpha_{j-1}\rho_j^L,
	& & \quad \sigma_j^L = \rho_{j-1}^L\rho_j^L,
	\\
	& \omega_j^R = -\alpha_j^\dag\alpha_{j-1},
	& & \quad \zeta_j^R = \alpha_j^\dag\rho_{j-1}^R,
	& & \quad \eta_j^R = -\rho_j^R\alpha_{j-1},
	& & \quad \sigma_j^R = \rho_j^R\rho_{j-1}^R.
 \end{aligned}
\end{align}
Note that $\hat{\mc{C}}(\{\alpha_j\})=\mc{C}(\{\alpha_j^T\})^T$ because $\Theta(\alpha_j)^T=\Theta(\alpha_j^T)$.

Consider the $d \times d$-block CMV matrix $\mc{C}$ with Verblunsky coefficients $(\alpha_0,\alpha_1,\dots)$ as a unitary operator $X \mapsto \mc{C}X$ on $\ell^2$. The block structure splits the canonical basis $\{e_j\}$ of $\ell^2$ into subsets of $d$ consecutive vectors as
\[
 \{e_0,e_1,e_2,\dots\} =
 \{e_0,e_1,\dots,e_{d-1}\} \cup \{e_d,e_{d+1},\dots,e_{2d-1}\} \cup \cdots
\]
Let $V_j=\spn\{e_{jd},e_{jd+1},\dots,e_{jd+d-1}\}$ be the $d$-dimensional subspace spanned by the $j$-th of these subsets, and let $P_j$ be the orthogonal projection of $\ell^2$ onto $V_j$. We will refer to $V_j$ as the canonical subspaces of $\ell^2$. From \eqref{eq:REC-CMV} we obtain
\[
 \int z^n \chi_j(z)^\dag \, d\mu(z) \, \chi_k(z) = P_j \, \mc{C}^n P_k,
 \qquad
 n\in\Z.
\]
According to \eqref{eq:SPEC-MEASURE}, this identifies the $\mc{C}$-spectral measure of $V_j$ as $d\mu_{V_j}=\chi_j^\dag\,d\mu\,\chi_j$, whose Schur function $f_{V_j}$ should be involved in a matrix version of Khrushchev's formula. In particular, the orthogonality measure $\mu$ with Verblunsky coefficients $(\alpha_0,\alpha_1,\dots)$ is at the same time the $\mc{C}$-spectral measure $\mu_{V_0}$ of the first canonical subspace $V_0=\spn\{e_0,e_1,\dots,e_{d-1}\}$ since $\chi_0=\1$. Similar results hold for the block CMV matrix $\hat{\mc{C}}$ and the CMV basis $x_j$, so the $\hat{\mc{C}}$-spectral measure of $V_j$ is given by $d\hat\mu_{V_j}=x_j^\dag\,d\mu\,x_j$ and $\mu=\hat\mu_{V_0}$.

Behind the above identification of the $\mc{C}$-spectral measure $\mu_{V_j}$ lies the unitary equivalence between the block CMV operator $\mc{C}$ and a unitary multiplication operator. This last one is the operator $U_\mu$ given in \eqref{eq:multip-op} but considering $L^2_\mu$ as the Hilbert space of (column) $d$-vector functions with inner product
\begin{equation} \label{eq:inner-vector-R}
 \<\phi|\psi\>_R=\int \phi(z)^\dag\,d\mu(z)\,\psi(z).
\end{equation}
The identity $\int \chi_j(z)^\dag\,d\mu(z)\,z\chi_k(z)=P_j\,\mc{C}P_k$ shows that the unitary equivalence assigns the canonical vector $e_{jd+m}$ to the $m$-th $d$-vector column $\chi_{j,m}$ of $\chi_j=(\chi_{j,0},\dots,\chi_{j,d-1})$, so that $V_j$ corresponds to the subspace of $L^2_\mu$ spanned by the columns of $\chi_j$. Hence, the $U_\mu$-spectral measure of this subspace, $\chi_j^\dag\,d\mu\,\chi_j$, coincides with the $\mc{C}$-spectral measure of $V_j$.

Once we know that CMV Schur functions of canonical subspaces are natural candidates to appear in a matrix Khrushchev formula, it remains to apply Theorem \ref{thm:GEN-K} to factorize such Schur functions. This requires factorizations of CMV matrices with an overlapping canonical subspace. Denote by $P_{k<j}$ and $P_{k>j}$ the orthogonal projections of $\ell^2$ onto $\bigoplus_{k<j}V_k$ and $\bigoplus_{k>j}V_k$ respectively. The explicit form of the block CMV matrix $\mc{C}$ shows that $P_{k<j} \mc{C} P_{k>j}$ vanishes for even $j$, while $P_{k>j} \mc{C} P_{k<j}$ vanishes for odd $j$. This ensures the existence of $V_j$-overlapping factorizations of $\mc{C}$ for any canonical subspace $V_j$, with left and right subspaces
\begin{equation} \label{eq:HLR-CMV}
\begin{aligned}
 \HH_L = \bigoplus_{k>j}V_k, \quad  \HH_R = \bigoplus_{k<j}V_k,
 & \qquad \text{ even } j,
 \\
 \HH_L = \bigoplus_{k<j}V_k, \quad  \HH_R = \bigoplus_{k>j}V_k,
 & \qquad \text{ odd } j.
 \end{aligned}
\end{equation}
A similar result holds for the CMV matrix $\hat{\mc{C}}$, but exchanging the parity of the index $j$.

This kind of overlapping factorizations can be explicitly obtained. For instance, bearing in mind \eqref{eq:CMV} and \eqref{eq:OZES}, we can directly check the following $V_2$- and $V_3$-overlapping factorizations,
\begin{equation} \label{eq:ex-fact-CMV}
\begin{aligned}
 \mc{C} & = \text{\small $\begin{pmatrix}
 	\1
	\\ \noalign{\vskip2pt}
	& \1
	\\ \noalign{\vskip2pt}
	& & \alpha_2^\dag & \zeta_3^L & \sigma_3^L
	\\ \noalign{\vskip2pt}
	& & \rho_2^R & \omega_3^L & \eta_3^L
	\\ \noalign{\vskip2pt}
	& & & \zeta_4^R & \omega_4^R & \zeta_5^L & \dots
	\\ \noalign{\vskip2pt}
	& & & \sigma_4^R & \eta_4^R & \omega_5^L & \dots
	\\
	& & & & & \dots & \dots
 \end{pmatrix}
 \begin{pmatrix}
 	\alpha_0^\dag & \zeta_1^L & \sigma_1^L
	\\ \noalign{\vskip2pt}
	\rho_0^R & \omega_1^L & \eta_1^L
	\\ \noalign{\vskip2pt}
	& \rho_1^R & -\alpha_1
	\\ \noalign{\vskip2pt}
	& & & \1
	\\ \noalign{\vskip2pt}
	& & & & \1
	\\ \noalign{\vskip2pt}
	& & & & & \1
	\\
	& & & & & & \ddots
 \end{pmatrix}$}
 \\
 & = \text{\small $\begin{pmatrix}
 	\alpha_0^\dag & \zeta_1^L & \sigma_1^L
	\\ \noalign{\vskip2pt}
	\rho_0^R & \omega_1^L & \eta_1^L
	\\ \noalign{\vskip2pt}
	& \zeta_2^R & \omega_2^R & \rho_2^L
	\\ \noalign{\vskip2pt}
	& \sigma_2^R & \eta_2^R & -\alpha_2
	\\ \noalign{\vskip2pt}
	& & & & \1
	\\ \noalign{\vskip2pt}
	& & & & & \1
	\\
	& & & & & & \ddots
 \end{pmatrix}
 \begin{pmatrix}
 	\1
	\\ \noalign{\vskip2pt}
	& \1
	\\ \noalign{\vskip2pt}
	& & \1
	\\ \noalign{\vskip2pt}
	& & & \alpha_3^\dag & \rho_3^L
	\\ \noalign{\vskip2pt}
	& & & \zeta_4^R & \omega_4^R & \zeta_5^L & \dots
	\\ \noalign{\vskip2pt}
	& & & \sigma_4^R & \eta_4^R & \omega_5^L & \dots
	\\
	& & & & & \dots & \dots
 \end{pmatrix}$}.
\end{aligned}
\end{equation}
To express these factorizations in the general case, we will make explicit the dependence of a block CMV matrix on its Verblunsky coefficients, rewriting for instance the matrices in \eqref{eq:CMV} as $\mc{C}=\mc{C}(\alpha_0,\alpha_1,\dots)$ and $\hat{\mc{C}}=\hat{\mc{C}}(\alpha_0,\alpha_1,\dots)$.

We also need the finite version of block CMV matrices. The finite sequence of Schur parameters $(\alpha_0,\alpha_1,\dots,\alpha_{N-1},\1)$ defines a rational inner Schur function $f(z)$ obtained by inverting the Schur algorithm starting from $f_N(z)=\1$. The related measure $\mu$ is supported on $N+1$ points of $\T$, thus it generates $N+1$ orthonormal polynomials $\varphi_j^{L,R}$ and orthonormal Laurent polynomials $\chi_j$, $x_j$. The recurrence for this finite CMV basis is
\[
 z\chi(z)=\chi(z)\,\mc{C}_N,
 \quad
 \chi=(\chi_0,\chi_1,\dots,\chi_N);
 \qquad
 zx(z)=x(z)\,\hat{\mc{C}}_N,
 \quad
 x=(x_0,x_1,\dots,x_N),
\]
with $\mc{C}_N=\mc{C}_N(\alpha_0,\dots,\alpha_{N-1})$ and $\hat{\mc{C}}_N=\hat{\mc{C}}_N(\alpha_0,\dots,\alpha_{N-1})$ finite block CMV matrices obtained truncating $\mc{C}=\mc{C}(\alpha_0,\alpha_1,\dots)$ and $\hat{\mc{C}}=\hat{\mc{C}}(\alpha_0,\alpha_1,\dots)$ on $V_0 \oplus V_1 \oplus \cdots \oplus V_N$ and setting $\alpha_N\to\1$. For instance,
\[
 \mc{C}_3 =
 \begin{pmatrix}
 	\alpha_0^\dag & \zeta_1^L & \sigma_1^L
	\\ \noalign{\vskip2pt}
	\rho_0^R & \omega_1^L & \eta_1^L
	\\ \noalign{\vskip2pt}
	& \zeta_2^R & \omega_2^R & \rho_2^L
	\\ \noalign{\vskip2pt}
	& \sigma_2^R & \eta_2^R & -\alpha_2
 \end{pmatrix},
 \qquad
 \hat{\mc{C}}_3 =
 \begin{pmatrix}
 	\alpha_0^\dag & \rho_0^L
	\\ \noalign{\vskip2pt}
	\zeta_1^R & \omega_1^R & \zeta_2^L & \sigma_2^L
	\\ \noalign{\vskip2pt}
	\sigma_1^R & \eta_1^R & \omega_2^L & \eta_2^L
	\\ \noalign{\vskip2pt}
	& & \rho_2^R & -\alpha_2
 \end{pmatrix}.
\]
Using \eqref{eq:LM} we find as in the semi-infinite case we may write $\mc{C}_N=\mc{L}_N\mc{M}_N$ and $\hat{\mc{C}}_N=\mc{M}_N\mc{L}_N$, where $\mc{L}_N$ and $\mc{M}_N$ are the unitary factors
\[
\mc{L}_N \kern-1pt = \kern-1pt
 \begin{cases}
 	\text{\footnotesize $\Theta(\alpha_0) \oplus \Theta(\alpha_2)
	\oplus \cdots \oplus
	\Theta(\alpha_{N-1})$},
 	\\
 	\text{\footnotesize $\Theta(\alpha_0) \oplus \Theta(\alpha_2)
	\oplus \cdots \oplus
	\Theta(\alpha_{N-2}) \oplus \1$},
 \end{cases}
 \kern-4pt
 \mc{M}_N \kern-1pt = \kern-1pt
 \begin{cases}
 	\text{\footnotesize $\1 \oplus \Theta(\alpha_1) \oplus \Theta(\alpha_3)
	\oplus \cdots \oplus
	\Theta(\alpha_{N-2}) \oplus \1$},
 	& \kern-1pt \text{odd } N,
	\\
	\text{\footnotesize $\1 \oplus \Theta(\alpha_1) \oplus \Theta(\alpha_3)
	\oplus \cdots \oplus
	\Theta(\alpha_{N-1})$},
	& \kern-1pt \text{even } N.
 \end{cases}
\]
As in the semi-infinite case, the $\mc{C}_N$-spectral measure of $V_j$ is $d\mu_{V_j}=\chi_j^\dag\,d\mu\,\chi_j$ so that $\mu_{V_0}=\mu$, and analogously for $\hat{\mc{C}}_N$ and $x_j$.

With this notation, the factorizations in \eqref{eq:ex-fact-CMV} read as
\[
\begin{aligned}
 \mc{C}(\alpha_0,\alpha_1,\dots)
 & = [\1_2 \oplus \mc{C}(\alpha_2,\alpha_3,\dots)]
 [\mc{C}_2(\alpha_0,\alpha_1) \oplus \1_\infty]
 \\
 & = [\mc{C}_3(\alpha_0,\alpha_1,\alpha_2) \oplus \1_\infty]
 [\1_3 \oplus \hat{\mc{C}}(\alpha_3,\alpha_4,\dots)],
\end{aligned}
\]
where $\1_j$ stands for the identity matrix of order $jd$, so that $\1=\1_1$.

To sum up the discussion, a $V_j$-overlapping factorization of block CMV matrices for an arbitrary canonical subspace $V_j$ is provided by the following proposition.

\begin{prop} \label{prop:CMV-fac}
The block CMV matrices $\mc{C}$, $\hat{\mc{C}}$ with Verblunsky coefficients $(\alpha_0,\alpha_1,\dots)$ admit for any $j\in\N$ the $V_j$-overlapping factorizations
\[
 \mc{C} =
 \begin{cases}
 	(\1_j \oplus \mc{C}^{(j)}) (\mc{C}_j \oplus \1_\infty),
 	\\
 	(\mc{C}_j \oplus \1_\infty) (\1_j \oplus \hat{\mc{C}}^{(j)}),
 \end{cases}
 \quad
 \hat{\mc{C}} =
 \begin{cases}
 	(\hat{\mc{C}}_j \oplus \1_\infty) (\1_j \oplus \hat{\mc{C}}^{(j)}),
	& \quad \text{even } j,
 	\\
 	(\1_j \oplus \mc{C}^{(j)}) (\hat{\mc{C}}_j \oplus \1_\infty),
	& \quad \text{odd } j,
 \end{cases}
\]
where $\mc{C}_j=\mc{C}_j(\alpha_0,\dots,\alpha_{j-1})$, $\mc{C}^{(j)}:=\mc{C}(\alpha_j,\alpha_{j+1},\dots)$ and analogously for $\hat{\mc{C}}_j$, $\hat{\mc{C}}^{(j)}$.
\end{prop}

\begin{proof}
In the factorization $\mc{C}=\mc{L}\mc{M}$, each block $\Theta(\alpha_j)$ acts on the subspace $V_j \oplus V_{j+1}$, thus it commutes with $\Theta(\alpha_k)$ if $k\ne j\pm1$. This allows the following rearrangement for even $j$,
\[
 \begin{aligned}
 	\mc{C} = & \; \text{\small$
 	[\cdots \oplus \Theta(\alpha_{j-2}) \oplus \Theta(\alpha_j) \oplus
 	\Theta(\alpha_{j+2}) \oplus \cdots]
 	[\cdots \oplus \Theta(\alpha_{j-3}) \oplus \Theta(\alpha_{j-1}) \oplus
 	\Theta(\alpha_{j+1}) \oplus \Theta(\alpha_{j+3}) \oplus \cdots]
	$}
 	\\
	= & \; \text{\small$
	[\1_j \oplus \Theta(\alpha_j) \oplus \Theta(\alpha_{j+2}) \oplus
	\Theta(\alpha_{j+4}) \oplus \cdots]
	[\1_j \oplus \1 \oplus \Theta(\alpha_{j+1}) \oplus \Theta(\alpha_{j+3})
	\oplus \cdots]
 	$}
	\\
	& \; \text{\small$
	[\Theta(\alpha_0) \oplus \Theta(\alpha_2) \oplus \cdots \oplus
	\Theta(\alpha_{j-2}) \oplus \1 \oplus \1_\infty]
	[\1 \oplus \Theta(\alpha_1) \oplus \Theta(\alpha_3) \oplus \cdots \oplus
	\Theta(\alpha_{j-1}) \oplus \1_\infty]
	$}
	\\
	= & \; (\1_j \oplus \mc{C}^{(j)}) (\mc{C}_j \oplus \1_\infty),
 \end{aligned}
\]
while for odd $j$ we have,
\[
 \begin{aligned}
 	\mc{C} = & \; \text{\small$
 	[\cdots \oplus \Theta(\alpha_{j-3}) \oplus \Theta(\alpha_{j-1}) \oplus
 	\Theta(\alpha_{j+1}) \oplus \Theta(\alpha_{j+3}) \oplus \cdots]
 	[\cdots \oplus \Theta(\alpha_{j-2}) \oplus \Theta(\alpha_j) \oplus
 	\Theta(\alpha_{j+2}) \oplus \cdots]
	$}
 	\\
	= & \; \text{\small$
	[\Theta(\alpha_0) \oplus \Theta(\alpha_2) \oplus \cdots \oplus
	\Theta(\alpha_{j-1}) \oplus \1_\infty]
	[\1 \oplus \Theta(\alpha_1) \oplus \Theta(\alpha_3) \oplus \cdots \oplus
	\Theta(\alpha_{j-2}) \oplus \1 \oplus \1_\infty]
	$}
	\\
	& \; \text{\small$
	[\1_j \oplus \1 \oplus \Theta(\alpha_{j+1}) \oplus \Theta(\alpha_{j+3})
	\oplus \cdots]
	[\1_j \oplus \Theta(\alpha_j) \oplus \Theta(\alpha_{j+2}) \oplus
	\Theta(\alpha_{j+4}) \oplus \cdots]
	$}
	\\
	= & \; (\mc{C}_j \oplus \1_\infty) (\1_j \oplus \hat{\mc{C}}^{(j)}).
 \end{aligned}
\]
The factorizations of $\hat{\mc{C}}$ follow from those of $\mc{C}$ using the relation $\hat{\mc{C}}(\{\alpha_k\})=\mc{C}(\{\alpha_k^T\})^T$.
\end{proof}

\begin{rem} \label{rem:finCMV-fac}
Proposition \ref{prop:CMV-fac} remains valid for finite block CMV matrices $\mc{C}_N(\alpha_0,\dots,\alpha_{N-1})$ and $\hat{\mc{C}}_N(\alpha_0,\dots,\alpha_{N-1})$ after the obvious replacements $\mc{C}^{(j)} \to \mc{C}_N^{(j)}:=\mc{C}_{N-j}(\alpha_j,\dots,\alpha_{N-1})$ and $\hat{\mc{C}}^{(j)} \to \hat{\mc{C}}_N^{(j)}:=\hat{\mc{C}}_{N-j}(\alpha_0,\dots,\alpha_{N-1})$.
\end{rem}

In spite of their simplicity, the above factorizations of CMV matrices seem not to have been noticed previously, not even in the scalar case. They are different from the decoupling that is used sometimes in the operator approach to OPUC. This kind of decoupling comes for instance by substituting $\Theta(\alpha_{j-1})$ by the identity matrix, giving rise to the relations
\[
\mc{C} =
 \begin{cases}
 	\mc{C}_{j-1} \oplus \mc{C}^{(j)} + K_j,
 	\\
 	\mc{C}_{j-1} \oplus \hat{\mc{C}}^{(j)} + K_j,
 \end{cases}
 \quad
 \hat{\mc{C}} =
 \begin{cases}
 	\hat{\mc{C}}_{j-1} \oplus \hat{\mc{C}}^{(j)}
	+ \hat{K}_j,
	& \quad \text{ even } j,
 	\\
 	\hat{\mc{C}}_{j-1} \oplus \mc{C}^{(j)}
	+ \hat{K}_j,
	& \quad \text{ odd } j,
 \end{cases}
\]
where $K_j$ and $\hat{K}_j$ are rank 2$d$ perturbations. Indeed, a decoupling similar to this one has been used in the case $d=1$ to prove the scalar version of Khrushchev's formula. However, no matrix version of Khrushchev's formula has been obtained with this kind of techniques.

Proposition \ref{prop:CMV-fac} does not provide a decoupling because $\mc{C}_j$, $\hat{\mc{C}}_j$ and $\mc{C}^{(j)}$, $\hat{\mc{C}}^{(j)}$ do not act on orthogonal subspaces, but on subspaces $\bigoplus_{k \le j}V_k$, $\bigoplus_{k \ge j} V_k$ which are orthogonal only up to the $d$-dimensional overlapping subspace $V_j$. However, these overlapping operators yield a clean factorization of CMV matrices which avoids the finite rank perturbations of the usual decoupling. This advantage is crucial for the generalization of Khrushchev's formula to matrix-valued measures.

In what follows we will refer to the factorizations of Proposition \ref{prop:CMV-fac} as the (standard) $V_j$-overlapping factorizations of $\mc{C}$ and $\hat{\mc{C}}$.

\begin{thm} \label{thm:MATRIX-K-1}
{\bf \small [Khrushchev's formula for matrix measures]}
\newline
Let $f$ be the Schur function of a non-trivial matrix measure $\mu$ on $\T$ with left and right orthonormal polynomials $\varphi_j^{L,R}$. If $f_j$ and $b_j$ are the iterates and inverse iterates of $f$, then
\begin{enumerate}
 \item The Schur function of {\rm $\varphi_j^L\,d\mu\,(\varphi_j^L)^\dag$} is $b_jf_j$.
 \item The Schur function of {\rm $(\varphi_j^R)^\dag d\mu\,\varphi_j^R$} is $f_jb_j$.
\end{enumerate}
\end{thm}

\begin{proof}
Consider the $V_j$-overlapping factorization of the block CMV matrix $\mc{C}=\mc{C}(\{\alpha_k\})$ of $\mu$, so that $\HH_C=V_j$ and $\HH_L$, $\HH_R$ are as in \eqref{eq:HLR-CMV}. According to Proposition \ref{prop:CMV-fac}, the left and right operators are
\[
 U_{LC} =
 \begin{cases}
 \mc{C}^{(j)},
 \\
 \mc{C}_j,
 \end{cases}
 U_{CR} =
 \begin{cases}
 \mc{C}_j & \text{ even } j,
 \\
 \hat{\mc{C}}^{(j)} & \text{ odd } j,
 \end{cases}
\]
which act respectively on $\HH_{LC} = \HH_L \oplus \HH_C$ and $\HH_{CR} = \HH_C \oplus \HH_R$, i.e. $\mc{C}_j$ acts on $\bigoplus_{k\le j}V_k$, while $\mc{C}^{(j)}$ and $\hat{\mc{C}}^{(j)}$ act on $\bigoplus_{k\ge j}V_k$.

Applying Theorem \ref{thm:GEN-K} to the above factorization when choosing the subspace $V=V_j$, we find that the $\mc{C}$-Schur function $f_{V_j}$ of $V_j$ factorizes as
\[
 f_{V_j} =
 \begin{cases}
 b_{V_j} f^{(j)}_{V_j} & \text{ even } j,
 \\
 \hat{f}^{(j)}_{V_j} b_{V_j} & \text{ odd } j,
 \end{cases}
\]
where $b_{V_j}$, $f^{(j)}_{V_j}$ and $\hat{f}^{(j)}_{V_j}$ denote the Schur functions of $V_j$ with respect to $\mc{C}_j$, $\mc{C}^{(j)}$ and $\hat{\mc{C}}^{(j)}$ respectively.

Since $V_j$ is the first canonical subspace of the Hilbert space on which $\mc{C}^{(j)}$ acts, we know that the $\mc{C}^{(j)}$-spectral measure of $V_j$ has the same Verblunsky coefficients $(\alpha_j,\alpha_{j+1},\dots)$ as $\mc{C}^{(j)}$, thus $f^{(j)}_{V_j}=f_j$. A similar argument yields $\hat{f}^{(j)}_{V_j}=f_j$.

On the other hand, $V_j$ is the last canonical subspace on which $\mc{C}_j$ acts. Reordering the basis of the underlying Hilbert space $V_0 \oplus \cdots \oplus V_j$ as $B_j \cup B_{j-1} \cup \cdots \cup B_0$, with $B_k=\{e_{kd},e_{kd+1},\dots,e_{kd+d-1}\}$, $V_j$ becomes the first subspace. This reordering transforms
\[
\begin{aligned}
 & \mc{L}_j(\{\alpha_k\}) \to \mc{L}_j(\{-\alpha_{j-k-1}^\dag\}),
 & \quad & \mc{M}_j(\{\alpha_k\}) \to \mc{M}_j(\{-\alpha_{j-k-1}^\dag\}),
 & \quad & \text{ odd } j,
 \\
 & \mc{L}_j(\{\alpha_k\}) \to \mc{M}_j(\{-\alpha_{j-k-1}^\dag\}),
 & & \mc{M}_j(\{\alpha_k\}) \to \mc{L}_j(\{-\alpha_{j-k-1}^\dag\}),
 & & \text{ even } j,
\end{aligned}
\]
thus $\mc{C}_j(\{\alpha_k\}) \to \mc{C}_j(\{-\alpha_{j-k-1}^\dag\})$ and $\mc{C}_j(\{\alpha_k\}) \to \hat{\mc{C}}_j(\{-\alpha_{j-k-1}^\dag\})$ for odd and even $j$ respectively. Therefore, the $\mc{C}_j(\{\alpha_k\})$-spectral measure of the last canonical subspace has the same Verblunsky coefficients $(-\alpha_{j-1}^\dag,-\alpha_{j-2}^\dag,\dots,-\alpha_0^\dag,\1)$ as the spectral measure of the first canonical subspace with respect to $\mc{C}_j(\{-\alpha_{j-k-1}^\dag\})$ or $\hat{\mc{C}}_j(\{-\alpha_{j-k-1}^\dag\})$. Hence, $b_{V_j}=b_j$.

Summarizing, the Schur function of the $\mc{C}$-spectral measure $d\mu_{V_j} = \chi_j^\dag d\mu\,\chi_j$ is given by
\[
 f_{V_j} =
 \begin{cases}
 b_j f_j & \text{ even } j,
 \\
 f_j b_j & \text{ odd } j.
 \end{cases}
\]
A similar analysis starting with the alternative block CMV matrix $\hat{\mc{C}}$ of $\mu$ proves that the Schur function of the $\hat{\mc{C}}$-spectral measure $d\hat\mu_{V_j} = x_j^\dag d\mu\,x_j$ factorizes as
\[
 \hat{f}_{V_j} =
 \begin{cases}
 f_j b_j & \text{ even } j,
 \\
 b_j f_j & \text{ odd } j.
 \end{cases}
\]
Finally, relation \eqref{eq:spec-measure-Vj} between the CMV spectral measures and the measures $\varphi_j^L\,d\mu\,(\varphi_j^L)^\dag$, $(\varphi_j^R)^\dag d\mu\,\varphi_j^R$ allows us to unify the previous results into the statement of the theorem.
\end{proof}

Concerning the study of quantum systems, the relevance of the above result relies on the fact that it gives the first return generating function of a site in a block CMV quantum walk (see Section~\ref{sec:CMV-PC}). This provides a way of studying the recurrence properties  of the sites of such a QW. In particular, this result opens the possibility of computing for any vector of a site the probability of returning to that site and the corresponding expected return time.

The strength of Theorem \ref{thm:GEN-K} is not only illustrated by the simplicity of the previous proof, but it also allows us to obtain more general matrix Khrushchev's formulas which are new already for scalar measures. Since Theorem \ref{thm:GEN-K} can be rewritten as the identification of the Schur function for the CMV spectral measures $\chi_j^\dag\,d\mu\,\chi_j$ and $x_j^\dag\,d\mu\,x_j$, a natural generalization is to find the Schur functions of the measures
\begin{equation} \label{eq:MM-MEASURE}
 \kern-8pt d\mu_{[j,k]} \kern-1pt = \kern-1pt \text{\footnotesize$
 \left(\begin{array}{@{\hspace{-1pt}} c@{\hspace{3pt}} c@{\hspace{3pt}} c@{\hspace{3pt}} c@{\hspace{-1.5pt}}}
 	\chi_j^\dag d\mu\,\chi_j
	& \chi_j^\dag d\mu\,\chi_{j+1}
	& \dots
	& \chi_j^\dag d\mu\,\chi_k
	\\ \noalign{\vskip3pt}
	\chi_{j+1}^\dag d\mu\,\chi_j
	& \chi_{j+1}^\dag d\mu\,\chi_{j+1}
	& \dots
	& \chi_{j+1}^\dag d\mu\,\chi_k
	\\
	\vdots & \vdots & & \vdots
	\\
	\chi_k^\dag d\mu\,\chi_j
	& \chi_k^\dag d\mu\,\chi_{j+1}
	& \dots
	& \chi_k^\dag d\mu\,\chi_k
 \end{array}\right)$},
 \kern8pt
 d\hat{\mu}_{[j,k]} \kern-1pt = \kern-1pt \text{\footnotesize$
 \left(\begin{array}{@{\hspace{-1pt}} c@{\hspace{3pt}} c@{\hspace{3pt}} c@{\hspace{3pt}} c@{\hspace{-1.5pt}}}
 	x_j^\dag d\mu\,x_j
	& x_j^\dag d\mu\,x_{j+1}
	& \dots
	& x_j^\dag d\mu\,x_k
	\\ \noalign{\vskip3pt}
	x_{j+1}^\dag d\mu\,x_j
	& x_{j+1}^\dag d\mu\,x_{j+1}
	& \dots
	& x_{j+1}^\dag d\mu\,x_k
	\\
	\vdots & \vdots & & \vdots
	\\
	x_k^\dag d\mu\,x_j
	& x_k^\dag d\mu\,x_{j+1}
	& \dots
	& x_k^\dag d\mu\,x_k
 \end{array}\right)$}.
\end{equation}

From a quantum mechanical point of view, this question is even more interesting than the one answered by Theorem \ref{thm:MATRIX-K-1}, since it is connected to the recurrence properties of an arbitrary number of sites. Actually, due to its use for the study of return properties of QWs, in the case of a scalar measure $\mu$ (i.e. $d=1$), the $2\times2$ matrix Schur function of $\hat{\mu}_{[j,j+1]}$ was already computed in \cite[Appendix B]{BGVW}. This computation follows an approach which in a sense is the opposite of the one proposed here, since it uses merely OPUC techniques (among others, the scalar Khrushchev formula for $\mu$) to obtain a quantum result. Besides, despite the fact that this computation deals with the simplest matrix case of \eqref{eq:MM-MEASURE}, it is far from being straightforward and seems hardly generalizable. Things will change with the new approach given here.

This general question can be reformulated using the $d \times d$-block CMV matrices $\mc{C}$, $\hat{\mc{C}}$ of $\mu$ due to their unitary equivalence with the multiplication operator $U_\mu$ on the $L^2_\mu$ space of $d$-vector functions. Such a unitary equivalence identifies the subspace generated by the columns of $\chi_j$ or $x_j$ with the canonical subspace $V_j$. This implies that $\mu_{[j,k]}$ and $\hat{\mu}_{[j,k]}$, which are the $U_\mu$-spectral measures of the subspaces spanned by the columns of the vectors $\chi_j,\chi_{j+1},\dots,\chi_k$ and $x_j,x_{j+1},\dots,x_k$, can be understood as the spectral measures of $V_{[j,k]} := V_j \oplus V_{j+1} \oplus \cdots \oplus V_k$ with respect to $\mc{C}$ and $\hat{\mc{C}}$ respectively. Thus, the Schur functions that we wish to characterize are those attached to the the subspaces $V_{[j,k]}$ by $\mc{C}$ and $\hat{\mc{C}}$.

To state this new result we need some other notations concerning CMV matrices.
Let us denote respectively by $\mc{C}_{[j,k]}=\mc{C}_{[j,k]}(\alpha_{j-1},\dots,\alpha_k)$ and $\hat{\mc{C}}_{[j,k]}=\hat{\mc{C}}_{[j,k]}(\alpha_{j-1},\dots,\alpha_k)$ the principal submatrices of $\mc{C}$ and $\hat{\mc{C}}$ corresponding to the subspace $V_{[j,k]}$. For example, using notation \eqref{eq:OZES},
\[
 \mc{C}_{[1,4]} =
 \begin{pmatrix}
 	\omega_1^L & \eta_1^L
	\\ \noalign{\vskip2pt}
	\zeta_2^R & \omega_2^R & \zeta_3^L & \sigma_3^L
	\\ \noalign{\vskip2pt}
	\sigma_2^R & \eta_2^R & \omega_3^L & \eta_3^L
	\\ \noalign{\vskip2pt}
	& & \zeta_4^R & \omega_4^R
 \end{pmatrix},
 \qquad
 \mc{C}_{[2,5]} =
 \begin{pmatrix}
 	\omega_2^R & \zeta_3^L & \sigma_3^L
	\\ \noalign{\vskip2pt}
	\eta_2^R & \omega_3^L & \eta_3^L
	\\ \noalign{\vskip2pt}
	& \zeta_4^R & \omega_4^R & \zeta_5^L
	\\ \noalign{\vskip2pt}
	& \sigma_4^R & \eta_4^R & \omega_5^L
 \end{pmatrix}.
\]
The submatrices $\mc{C}_{[j,k]}$, $\hat{\mc{C}}_{[j,k]}$ become unitary when setting $\alpha_{j-1}\to-\1$ and $\alpha_k\to\1$ because this substitution makes $\mc{C}$, $\hat{\mc{C}}$ to split into direct sums which decouple such submatrices. This defines the finite block CMV unitary truncations
\begin{equation} \label{eq:CMV-TRUNC}
 \mc{C}_{(j,k)} := \mc{C}_{[j,k]}(-\1,\alpha_j\dots,\alpha_{k-1},\1),
 \qquad
 \hat{\mc{C}}_{(j,k)} := \hat{\mc{C}}_{[j,k]}(-\1,\alpha_j\dots,\alpha_{k-1},\1).
\end{equation}
For instance,
\[
 \mc{C}_{(1,4)} =
 \begin{pmatrix}
 	\alpha_1^\dag & \rho_1^L
	\\ \noalign{\vskip2pt}
	\zeta_2^R & \omega_2^R & \zeta_3^L & \sigma_3^L
	\\ \noalign{\vskip2pt}
	\sigma_2^R & \eta_2^R & \omega_3^L & \eta_3^L
	\\ \noalign{\vskip2pt}
	& & \rho_3^R & -\alpha_3
 \end{pmatrix}
 = \hat{\mc{C}}_3^{(1)},
 \qquad
 \mc{C}_{(2,5)} =
 \begin{pmatrix}
 	\alpha_2^\dag & \zeta_3^L & \sigma_3^L
	\\ \noalign{\vskip2pt}
	\rho_2^R & \omega_3^L & \eta_3^L
	\\ \noalign{\vskip2pt}
	& \zeta_4^R & \omega_4^R & \rho_4^L
	\\ \noalign{\vskip2pt}
	& \sigma_4^R & \eta_4^R & -\alpha_4
 \end{pmatrix}
 = \mc{C}_3^{(2)}.
\]
In general, these truncations can be expressed in terms of finite block CMV matrices as
\[
 \mc{C}_{(j,k)} =
 \begin{cases}
 	\mc{C}_{k-j}(\alpha_j,\dots,\alpha_{k-1}) = \mc{C}_{k-j}^{(j)},
	\\ \noalign{\vskip4pt}
	\hat{\mc{C}}_{k-j}(\alpha_j,\dots,\alpha_{k-1}) = \hat{\mc{C}}_{k-j}^{(j)},
 \end{cases}
 \quad
 \hat{\mc{C}}_{(j,k)} =
 \begin{cases}
 	\hat{\mc{C}}_{k-j}(\alpha_j,\dots,\alpha_{k-1}) = \hat{\mc{C}}_{k-j}^{(j)},
	& \quad \text{ even } j,
	\\ \noalign{\vskip4pt}
	\mc{C}_{k-j}(\alpha_j,\dots,\alpha_{k-1}) = \mc{C}_{k-j}^{(j)},
	& \quad \text{ odd } j.
 \end{cases}
\]
Note that the dependence of $\mc{C}_{[j,k]}$ on the extreme Verblunsky coefficients $\alpha_{j-1}$ and $\alpha_k$ can be factorized as
\[
 \mc{C}_{[j,k]} =
 \begin{cases}
 	\left(\1_{k-j} \oplus \alpha_k^\dag \right)
 	\mc{C}_{(j,k)}
 	\left(-\alpha_{j-1} \oplus \1_{k-j} \right),
	& \text{ even } j,k,
	\\ \noalign{\vskip3pt}
	\left(-\alpha_{j-1} \oplus \1_{k-j} \right) 	
	\mc{C}_{(j,k)}
 	\left(\1_{k-j} \oplus \alpha_k^\dag \right),
	& \text{ odd } j,k,
	\\ \noalign{\vskip3pt}
	\mc{C}_{(j,k)}
 	\left(
	-\alpha_{j-1} \oplus \1_{k-j-1} \oplus \alpha_k^\dag
	\right),
	& \text{ even } j, \text{ odd } k,
	\\ \noalign{\vskip3pt}
	\left(
	-\alpha_{j-1} \oplus \1_{k-j-1} \oplus \alpha_k^\dag
	\right)
	\mc{C}_{(j,k)},
	& \text{ odd } j, \text{ even } k.
 \end{cases}
\]
The corresponding relation between $\hat{\mc{C}}_{[j,k]}$ and $\hat{\mc{C}}_{(j,k)}$ is similar to the above one, but with the opposite parity for the indices.

Using this notation we have the following extension of Theorem \ref{thm:MATRIX-K-1}.

\begin{thm} \label{thm:MATRIX-K-2}
{\bf \small
[$1^{\text{st}}$ generalized Khrushchev's formulas for matrix measures]}
\newline
Let $f$ be the Schur function of a non-trivial matrix measure $\mu$ on $\T$ with Verblunsky coefficients $\alpha_j$, CMV basis $\chi_j$, $x_j$ and block CMV matrices $\mc{C}$, $\hat{\mc{C}}$. If $f_j$ and $b_j$ are the iterates and inverse iterates of $f$, then the Schur functions $f_{[j,k]}$, $\hat{f}_{[j,k]}$  of the measures $\mu_{[j,k]}$, $\hat\mu_{[j,k]}$ given in \eqref{eq:MM-MEASURE} are the result of substituting {\rm $-\alpha_{j-1}^\dag \to b_j$} and $\alpha_k \to f_k$ into {\rm $\mc{C}_{[j,k]}^\dag$}, {\rm $\hat{\mc{C}}_{[j,k]}^\dag$} respectively. More explicitly,
{\rm
\[
 f_{[j,k]} =
 \begin{cases}
 	\left(b_j \oplus \1_{k-j}\right)
 	\mc{C}_{(j,k)}^\dag
	\left(\1_{k-j} \oplus f_k\right),
	& \text{ even } j,k,
	\\ \noalign{\vskip5pt}
 	\left(\1_{k-j} \oplus f_k\right)	
	\mc{C}_{(j,k)}^\dag
	\left(b_j \oplus \1_{k-j}\right),
	& \text{ odd } j,k,
	\\ \noalign{\vskip5pt}
	\left(b_j \oplus \1_{k-j-1} \oplus f_k\right)
	\mc{C}_{(j,k)}^\dag,
	& \text{ even } j, \text{ odd } k,
	\\ \noalign{\vskip5pt}
	\mc{C}_{(j,k)}^\dag
	\left(b_j \oplus \1_{k-j-1} \oplus f_k\right),
	& \text{ odd } j, \text{ even } k,
 \end{cases}
\]}

\noindent while the expression for $\hat{f}_{[j,k]}$ is obtained by changing $\mc{C}_{(j,k)} \to \hat{\mc{C}}_{(j,k)}$ and inverting the parity of the indices in the above expression for $f_{[j,k]}$.
\end{thm}

\begin{proof}
The result follows from a triple factorization coming from a two step application of Theorem \ref{thm:GEN-K} to the subspace $V=V_{[j,k]}$. Let us consider the measure $\mu_{[j,k]}$ in the case of even indices $j,k$. Its Schur function $f_{[j,k]}$ is that one $f_{V_{[j,k]}}$ linked to the subspace $V_{[j,k]}$ by the block CMV matrix $\mc{C}$.

According to Theorem \ref{thm:GEN-K}, the $V_j$-overlapping factorization of $\mc{C}$ yields
\[
 \begin{gathered}
 	\HH_L=\bigoplus_{l>j}V_l, \quad \HH_R=\bigoplus_{l<j}V_l, \quad
	V_L=V_{[j+1,k]}, \quad V_R=\{0\}, \quad
 	V_{LC}=V_{[j,k]}=V, \quad V_{CR}=V_j,
	\\
 	U_{LC}=\mc{C}^{(j)}, \quad U_{CR}=\mc{C}_j, \quad
 	f_{V_{[j,k]}} =
 	(b_{V_j} \oplus \1_{k-j}) f^{(j)}_{V_{[j,k]}},
 \end{gathered}
\]
where $f^{(j)}_{V_{[j,k]}}$ is the $\mc{C}^{(j)}$-Schur function of $V_{[j,k]}$ and $b_{V_j}$ is the $\mc{C}_j$-Schur function of $V_j$.

Concerning $f^{(j)}_{V_{[j,k]}}$, the $V_k$-overlapping factorization of $\mc{C}^{(j)}$ with underlying Hilbert space $\HH=\bigoplus_{l\ge j}V_j$ leads to
\[
 \begin{gathered}
 	\HH_L=\bigoplus_{l>k}V_l, \quad \HH_R=V_{[j,k-1]}=V_R, \quad V_L=\{0\}, \quad
 	V_{CR}=V_{[j,k]}=V, \quad V_{LC}=V_k,
	\\
 	U_{LC}=\mc{C}^{(k)}, \quad
	U_{CR}=\mc{C}^{(j)}_{k-j}=\mc{C}_{(j,k)}, \quad
 	f^{(j)}_{V_{[j,k]}} =
 	f^{(j,k)}_{V_{[j,k]}} (\1_{k-j} \oplus f^{(k)}_{V_k}),
 \end{gathered}
\]
with $f^{(j,k)}_{V_{[j,k]}}$ the $\mc{C}_{(j,k)}$-Schur function of $V_{[j,k]}$ and $f^{(k)}_{V_k}$ the $\mc{C}^{(k)}$-Schur function of $V_k$.

Combining the results of both steps we obtain
\[
 f_{V_{[j,k]}} =
 (b_{V_j} \oplus \1_{k-j})
 f^{(j,k)}_{V_{[j,k]}}
 (\1_{k-j} \oplus f^{(k)}_{V_k}).
\]
Since $V_{[j,k]}$ is the whole Hilbert space where $\mc{C}^{(j,k)}$ acts, \eqref{eq:fH} implies that $f^{(j,k)}_{V_{[j,k]}}(z)=\mc{C}_{(j,k)}^\dag$. To prove the theorem for $f_{[j,k]}$ in the case of even indices $j,k$, it only remains to use the arguments in the proof of Theorem \ref{thm:MATRIX-K-1} to show that $f^{(k)}_{V_k}=f_k$ and $b_{V_j}=b_j$.

The proof for the rest of the cases of the theorem follows similar steps.
\end{proof}

In contrast to Theorem \ref{thm:MATRIX-K-1} which is a novel contribution for matrix measures, the above result is new even in the case of scalar measures. In such a case, Theorem \ref{thm:MATRIX-K-2} becomes a novel statement about matrix modifications of scalar measures.

Theorems \ref{thm:MATRIX-K-1} and \ref{thm:MATRIX-K-2} are also valid when $\mu$ is the finitely supported measure associated with $\mc{C}_N$ and $\hat{\mc{C}}_N$ because, as we pointed out in Remark \ref{rem:finCMV-fac}, these finite block CMV matrices have the same kind of $V_j$-overlapping factorizations as the semi-infinite ones. The only quirks of the finite case is that the measures $\mu_{[j,k]}$ and $\hat\mu_{[j,k]}$ make sense just for $k\le N$, and the Schur iterates $f_j$ have a finite number of Schur parameters $(\alpha_j,\alpha_{j+1},\dots,\alpha_{N-1},\1)$.

Moreover, when conveniently rewritten, the previous theorems also hold for doubly infinite block CMV matrices. These unitary matrices are defined for any doubly infinite sequence $(\alpha_j)_{j\in\Z}$ of matrix Verblunsky coefficients with $\|\alpha_j\|<1$ by $\mc{C}=\mc{L}\mc{M}$ and $\hat{\mc{C}}=\mc{M}\mc{L}$, where
\[
 \mc{L} = \bigoplus_{k\in\Z} \Theta(\alpha_{2k}),
 \qquad
 \mc{M} = \bigoplus_{k\in\Z} \Theta(\alpha_{2k-1}),
\]
and $\Theta(\alpha_j)$ acts on $V_j \oplus V_{j+1}$. Schur iterates $f_j$ and inverse Schur iterates $b_j$ can be associated with any doubly infinite block CMV matrix by defining them as the Schur functions with infinite Schur parameters $(\alpha_j,\alpha_{j+1},\dots)$ and $(-\alpha_{j-1}^\dag,-\alpha_{j-2}^\dag,\dots)$ respectively. The standard $V_j$-overlapping factorizations of these matrices are as in Proposition \ref{prop:CMV-fac}, with $\mc{C}_j=\mc{C}_j(\dots,\alpha_{j-2},\alpha_{j-1})$ and $\hat{\mc{C}}_j=\hat{\mc{C}}_j(\dots,\alpha_{j-2},\alpha_{j-1})$ defined by truncating the doubly infinite matrices $\mc{C}$ and $\hat{\mc{C}}$ on $\bigoplus_{k\le j}V_j$ and setting $\alpha_j\to\1$. The submatrices $\mc{C}_{[j,k]}$, $\hat{\mc{C}}_{[j,k]}$ and the unitary truncations $\mc{C}_{(j,k)}$, $\hat{\mc{C}}_{(j,k)}$ are also naturally extended to the doubly infinite case. Theorems \ref{thm:MATRIX-K-1} and \ref{thm:MATRIX-K-2} remain valid in this case, restated so that they deal directly with block CMV matrices, avoiding any reference to orthonormal polynomials or CMV basis.

With the above terminology, the proofs of Theorems \ref{thm:MATRIX-K-1} and \ref{thm:MATRIX-K-2} also serve as proofs of the following more general one.

\begin{thm} \label{thm:MATRIX-K-3}
{\bf [Khrushchev's formulas for block CMV matrices]}
\newline
Let $\mc{C}$ and $\hat{\mc{C}}$ be the block CMV matrices related to a (finite, semi-infinite or doubly infinite) sequence of $d \times d$ Verblunsky coefficients $(\alpha_j)$ with Schur iterates $f_j$ and inverse Schur iterates $b_j$. Denote by $\{e_j\}$ the canonical basis of the corresponding $\ell^2$ space and let $V_j=\spn\{e_{jd},e_{jd+1},\dots,e_{jd+d-1}\}$. Then,
\begin{enumerate}
 \item The $\mc{C}$-Schur function of $V_j$ is
       $\begin{cases}
       b_jf_j, & \text{even } j,
       \\
       f_jb_j, & \text{odd } j.
       \end{cases}$
 \item The $\hat{\mc{C}}$-Schur function of $V_j$ is
       $\begin{cases}
       f_jb_j, & \text{even } j,
       \\
       b_jf_j, & \text{odd } j.
       \end{cases}$
\end{enumerate}
Moreover, the Schur functions of $V_j \oplus V_{j+1} \oplus \cdots \oplus V_k$ with respect to $\mc{C}$, $\hat{\mc{C}}$ are the result of substituting {\rm $-\alpha_{j-1}^\dag \to b_j$} and $\alpha_k \to f_k$ into {\rm $\mc{C}_{[j,k]}^\dag$}, {\rm $\hat{\mc{C}}_{[j,k]}^\dag$} respectively.
\end{thm}

We will continue in the next section with a path counting view on the previous theorems.

\section{Matrix Khrushchev's formulas from CMV via path counting}
\label{sec:CMV-PC}

Considering Theorem \ref{thm:MATRIX-K-1} as a result on the CMV Schur function $f_{V_j}$ of a canonical subspace $V_j$, let us illustrate the path counting approach to the case of a $d \times d$-block CMV matrix $\mc{C}$. As in previous examples of path counting, we will work with the first return generating function $a_{V_j}(z)=zf_{V_j}^\dag(z)$.

Applying \eqref{eq:a-U} to $U=\mc{C}$ and $V=V_j$ we get the first return (matrix) amplitudes
\begin{equation} \label{eq:PATH-MATRIX}
 a_{V_j,n} = \sum_{k_i\ne j}
 \mc{C}_{j,k_{n-1}}\mc{C}_{k_{n-1},k_{n-2}}\cdots\:\mc{C}_{k_2,k_1}\mc{C}_{k_1,j},
 \qquad
 \mc{C}_{j,k} := P_j\,\mc{C}P_k,
\end{equation}
where $P_j$ is the orthogonal projection of $\ell^2$ onto $V_j$.
Note that, in contrast to \eqref{eq:PATH}, $\mc{C}_{j,k}$ does not represent now the $(j,k)$-th element of $\mc{C}$, but the $(j,k)$-th block which results when splitting $\mc{C}$ in $d\times d$ blocks. We will refer to the matrix
\[
 \ms{A}(V_j \ot V_{k_{n-1}} \ot \cdots \ot V_{k_1} \ot V_k) :=
 \mc{C}_{j,k_{n-1}}\mc{C}_{k_{n-1},k_{n-2}}\cdots\:\mc{C}_{k_2,k_1}\mc{C}_{k_1,k}
\]
as the amplitude of the $n$-step subspace path $\Lambda \equiv V_j \ot V_{k_{n-1}} \ot \cdots \ot V_{k_1} \ot V_k$. The amplitudes appearing in the sum of \eqref{eq:PATH-MATRIX} are associated with those loops of length $n$ having $V_j$ as initial and final subspace, but not as any intermediate one. Therefore, $a_{V_j}(z)=\ms{A}_{j,j}^{V_j}(z)$ where
\begin{equation} \label{eq:matrix-paths}
 \ms{A}_{j,k}^V(z) :=
 \kern-5pt
 \sum_{\Lambda \in \powerset_{\!j,k}(V)}
 \kern-10pt
 \ms{A}(\Lambda) \, z^{\text{length}(\Lambda)},
 \qquad
 \powerset_{\!j,k}(V) = \text{ \parbox{170pt}
 {set of all paths from $V_k$ to $V_j$
 \break
 avoiding $V$ at intermediate steps.}}
\end{equation}
In other words, the generating function $a_{V_j}(z)$ comes from the sum of the matrix amplitudes of all the loops of $\powerset_{\!j,j}(V_j)$, with a factor $z$ for each single subspace step of the loops. To obtain a matrix Khrushchev formula it only remains to perform a suitable splitting of $a_{V_j}$ using block path counting tricks, and then translate the results to the Schur function $f_{V_j}$ given by $a_{V_j}(z)=zf_{V_j}^\dag(z)$.

The path counting approach to Khrushchev's formulas uses a pictorial representation of block CMV matrices. Let us think of the canonical subspace $V_j$ as attached to the site $j\in\Z_+=\{0,1,2,\dots\}$ of a semi-infinite lattice, using the symbolic representation
\begin{center}
\begin{tikzpicture}[baseline=(current bounding box.center),->,>=stealth',shorten >=1pt,auto,semithick]
  \tikzstyle{every state}=[fill=blue,draw=none,text=white,inner sep=0pt,minimum size=0.8cm]

  \node (V) {\large $V_j \ot V_k \kern5pt \equiv$};
  \node[state] (j) [right of=V,node distance=1.8cm]{$\boldsymbol{j}$};
  \node[state] (k) [right of=j,node distance=2cm] {$\boldsymbol{k}$};

  \path (k) edge	node {$\mc{C}_{j,k}$} (j);

\end{tikzpicture}
\end{center}
for a one-step subspace path, indicating also its matrix amplitude $\mc{C}_{j,k}$. Then, the steps allowed by the block CMV matrix $\mc{C}$ given in \eqref{eq:CMV} can be diagrammatically depicted as
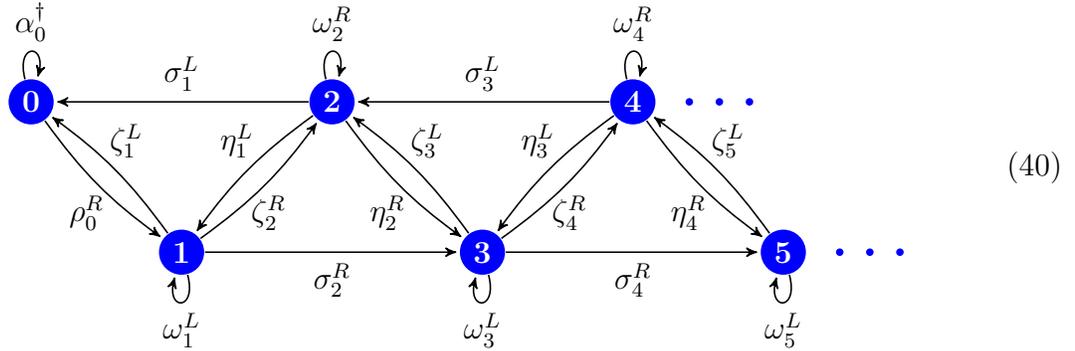
\begin{equation} \label{tz:CMV}
\begin{tikzpicture}[baseline=(current bounding box.center),->,>=stealth',shorten >=1pt,auto,semithick]
  \tikzstyle{every state}=[node distance=4cm,fill=blue,draw=none,text=white,inner sep=0pt,minimum size=0.6cm]
  \tikzstyle{every place}=[node distance=0.4cm,fill=blue,draw=none,inner sep=0pt,minimum size=0.1cm]

  \node[state] (Cu) {$\boldsymbol 0$};
  \node[state] (Du) [right of=Cu] {$\boldsymbol 2$};
  \node[state] (Eu) [right of=Du] {$\boldsymbol 4$};

  \path (Eu) ++(0.35,0) coordinate (E0u);

  \node[place] (E1u) [right of=E0u] {};
  \node[place] (E2u) [right of=E1u] {};
  \node[place] (E3u) [right of=E2u] {};

  \node[state] (Cd) at (2,-2) {$\boldsymbol 1$};
  \node[state] (Dd) [right of=Cd] {$\boldsymbol 3$};
  \node[state] (Ed) [right of=Dd] {$\boldsymbol 5$};

  \path (Ed) ++(0.35,0) coordinate (E0d);

  \node[place,fill=blue] (E1d) [right of=E0d] {};
  \node[place,fill=blue] (E2d) [right of=E1d] {};
  \node[place,fill=blue] (E3d) [right of=E2d] {};

  \path (Cu) edge [loop above]			node {$\alpha_0^\dag$} (Cu)
             edge [bend right=9,below]	node {$\rho_0^R \kern9pt$} (Cd)
        (Du) edge [above]			   		node {$\sigma_1^L$} (Cu)
             edge [bend right=9,above] 	node {$\eta_1^L \kern9pt$} (Cd)
             edge [loop above]			node {$\omega_2^R$} (Du)
             edge [bend right=9,below]	node {$\eta_2^R \kern9pt$} (Dd)
        (Eu) edge [above]					node {$\sigma_3^L$} (Du)
             edge [bend right=9,above] 	node {$\eta_3^L \kern9pt$} (Dd)
             edge [loop above]			node {$\omega_4^R$} (Eu)
             edge [bend right=9,below]	node {$\eta_4^R \kern9pt$} (Ed)
        (Cd) edge [below]					node {$\sigma_2^R$} (Dd)
             edge [bend right=9,below]	node {$\kern3pt \zeta_2^R$} (Du)
             edge [loop below]			node {$\omega_1^L$} (Cd)
             edge [bend right=9,above]	node {$\kern9pt \zeta_1^L$} (Cu)
        (Dd) edge [below]			  		node {$\sigma_4^R$} (Ed)
             edge [bend right=9,below]	node {$\kern3pt \zeta_4^R$} (Eu)
             edge [loop below]			node {$\omega_3^L$} (Dd)
             edge [bend right=9,above]	node {$\kern9pt \zeta_3^L$} (Du)
        (Ed) edge [loop below]			node {$\omega_5^L$} (Ed)
             edge [bend right=9,above]	node {$\kern9pt \zeta_5^L$} (Eu);
\end{tikzpicture}
\end{equation}

In physical language, this picture corresponds to an interpretation of the Hilbert space $\ell^2$ as a quantum state space spanned by basis states $e_{jd+m}=|j,m\>$ labelled by the quantum numbers $j\in\Z_+$ and $m\in\Z_d=\{0,1,\dots,d-1\}$. Each site $j$ of the lattice supports $d$ internal degrees of freedom collected into the subspace $V_j=\spn\{|j,m\>\}_{m\in\Z_d}$. Diagram \eqref{tz:CMV} describes the one-step transition (matrix) amplitudes of a QW in this lattice whose evolution is given by the unitary step matrix $\mc{C}$, something that we will call a block CMV quantum walk.

Let us fix our attention on an odd site $j$ (the analysis for even $j$ is similar), marked below in red, and the corresponding generating function $a_{V_j}$ built out of the loops of $\powerset_{\!j,j}(V_j)$.
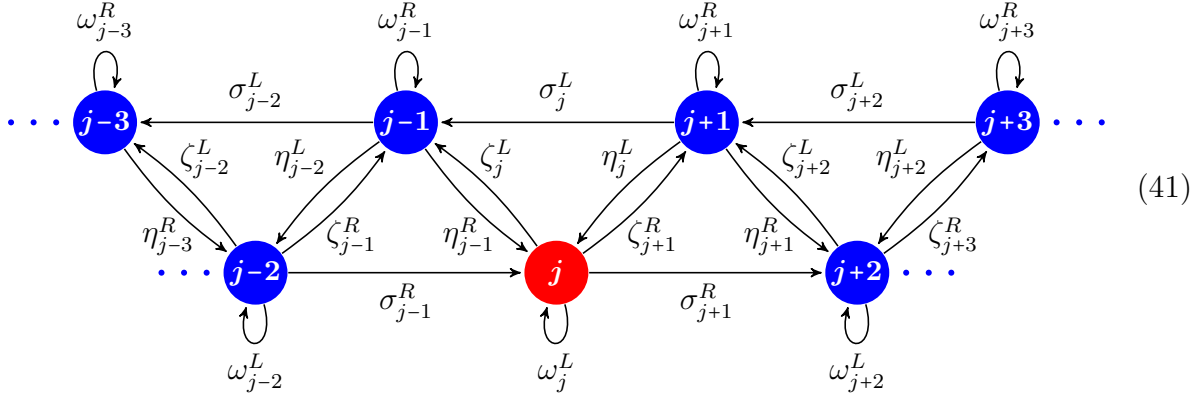
\begin{equation} \label{tz:CMV-ODD}
\begin{tikzpicture}[baseline=(current bounding box.center),->,>=stealth',shorten >=1pt,auto,semithick]
  \tikzstyle{every state}=[node distance=4cm,fill=blue,draw=none,text=white,inner sep=0pt,minimum size=0.85cm]
  \tikzstyle{every place}=[node distance=0.3cm,fill=blue,draw=none,inner sep=0pt,minimum size=0.08cm]

  \node[state] (Bu) {\small $\boldsymbol{j\!-\!3}$};
  \node[state] (Cu) [right of=Bu] {\small $\boldsymbol{j\!-\!1}$};
  \node[state] (Du) [right of=Cu] {\small $\boldsymbol{j\!+\!1}$};
  \node[state] (Eu) [right of=Du] {\small $\boldsymbol{j\!+\!3}$};

  \path (Bu) ++(-0.35,0) coordinate (B0u);

  \node[place] (B1u) [left of=B0u] {};
  \node[place] (B2u) [left of=B1u] {};
  \node[place] (B3u) [left of=B2u] {};

  \path (Eu) ++(0.35,0) coordinate (E0u);

  \node[place] (E1u) [right of=E0u] {};
  \node[place] (E2u) [right of=E1u] {};
  \node[place] (E3u) [right of=E2u] {};

  \node[state] (Bd) at (2,-2) {\small $\boldsymbol{j\!-\!2}$};
  \node[state] (Cd) [right of=Bd,fill=red] {\small $\boldsymbol{j}$};
  \node[state] (Dd) [right of=Cd] {\small $\boldsymbol{j\!+\!2}$};

  \path (Bd) ++(-0.35,0) coordinate (B0d);

  \node[place,fill=blue] (B1d) [left of=B0d] {};
  \node[place,fill=blue] (B2d) [left of=B1d] {};
  \node[place,fill=blue] (B3d) [left of=B2d] {};

  \path (Dd) ++(0.35,0) coordinate (D0d);

  \node[place,fill=blue] (D1d) [right of=D0d] {};
  \node[place,fill=blue] (D2d) [right of=D1d] {};
  \node[place,fill=blue] (D3d) [right of=D2d] {};

  \path (Bu) edge [loop above]			node {$\omega_{j-3}^R$} (Bu)
             edge [bend right=9,below]	node {$\eta_{j-3}^R \kern3pt$} (Bd)
        (Cu) edge [above]			   		node {$\sigma_{j-2}^L$} (Bu)
             edge [loop above]			node {$\omega_{j-1}^R$} (Cu)
             edge [bend right=9,below]	node {$\eta_{j-1}^R \kern3pt$} (Cd)
             edge [bend right=9,above]	node {$\eta_{j-2}^L \kern17pt$} (Bd)
        (Du) edge [above]			   		node {$\sigma_j^L$} (Cu)
             edge [bend right=9,above] 	node {$\eta_j^L \kern4pt$} (Cd)
             edge [loop above]			node {$\omega_{j+1}^R$} (Du)
             edge [bend right=9,below]	node {$\eta_{j+1}^R \kern3pt$} (Dd)
        (Eu) edge [above]					node {$\sigma_{j+2}^L$} (Du)
             edge [bend right=9,above] 	node {$\eta_{j+2}^L \kern17pt$} (Dd)
             edge [loop above]			node {$\omega_{j+3}^R$} (Eu)
        (Bd) edge [below]					node {$\sigma_{j-1}^R$} (Cd)
             edge [loop below]			node {$\omega_{j-2}^L$} (Bd)
             edge [bend right=9,below]	node {$\kern9pt \zeta_{j-1}^R$} (Cu)
             edge [bend right=9,above]	node {$\kern12pt \zeta_{j-2}^L$} (Bu)
        (Cd) edge [below]					node {$\sigma_{j+1}^R$} (Dd)
             edge [bend right=9,below]	node {$\kern9pt \zeta_{j+1}^R$} (Du)
             edge [loop below]			node {$\omega_j^L$} (Cd)
             edge [bend right=9,above]	node {$\kern5pt \zeta_j^L$} (Cu)
        (Dd) edge [bend right=9,below]	node {$\kern9pt \zeta_{j+3}^R$} (Eu)
             edge [loop below]			node {$\omega_{j+2}^L$} (Dd)
             edge [bend right=9,above]	node {$\kern12pt \zeta_{j+2}^L$} (Du);
\end{tikzpicture}
\end{equation}
To perform a factorization of $a_{V_j}=\ms{A}_{j,j}^{V_j}$, let us split the corresponding sum over loops according to the sides of the site $j$ hit by the intermediate steps of the loop. This amounts to writing $a_{V_j} = \Sigma_{V_j}^0 + \Sigma_{V_j}^L + \Sigma_{V_j}^R + \Sigma_{V_j}^{LR}$, where $\Sigma_{V_j}^0$ sums only over the self-loops of the site $j$, $\Sigma_{V_j}^L$ and $\Sigma_{V_j}^R$ include the loops hitting only the left or the right side respectively, and $\Sigma_{V_j}^{LR}$ takes into account the loops which hit both sides. Since the loops of $\powerset_{\!j,j}(V_j)$ cannot hit the site $j$ at intermediate steps, a simple inspection of \eqref{tz:CMV-ODD} shows that
\begin{equation} \label{eq:partial-sums}
 \Sigma_{V_j}^0 = \omega_j^L z,
 \quad
 \Sigma_{V_j}^L = \ms{A}^{V_j}_{j,j-1} \, (\zeta_j^L z) ,
 \quad
 \Sigma_{V_j}^R =  (\eta_j^L z) \, \ms{A}^{V_j}_{j+1,j},
 \quad
 \Sigma_{V_j}^{LR} =
 \ms{A}^{V_j}_{j,j-1} \, (\sigma_j^L z) \, \ms{A}^{V_j}_{j+1,j}.
\end{equation}
For instance, all the paths summed in $\Sigma_{V_j}^L$ must start with the step $V_{j-1} \leftarrow V_j$, which has an amplitude $\zeta_j^L$ and thus contributes with the factor $\zeta_j^L z$. The contribution of the rest of the steps to a return to $V_j$ starting from $V_{j-1}$ are taken into account in $\ms{A}^{V_j}_{j,j-1}$. The remaining identities can be derived in a similar fashion.
Therefore, combining the results of \eqref{eq:partial-sums} and using \eqref{eq:OZES}, we obtain
\begin{equation} \label{eq:GF-SPLIT}
 \begin{aligned}
 	a_{V_j} & =
 	\omega_j^L z +
 	\ms{A}^{V_j}_{j,j-1} \, (\zeta_j^L z) +
 	(\eta_j^L z) \, \ms{A}^{V_j}_{j,j+1} +
 	\ms{A}^{V_j}_{j,j-1} \, (\sigma_j^L z) \, \ms{A}^{V_j}_{j+1,j}
	\\
 	& = z^{-1}
 	[-\alpha_{j-1} z + \ms{A}^{V_j}_{j,j-1} \, (\rho_{j-1}^L z)]
	[\alpha_j^\dag z + (\rho_j^L z) \, \ms{A}^{V_j}_{j+1,j}].
 \end{aligned}
\end{equation}

We can also identify the factors of the above factorization as first return generating functions related to block CMV matrices. With this aim, let us split diagram \eqref{tz:CMV-ODD} at site $j$ into left and right ones obtained by setting $\alpha_j\to\1$ and $\alpha_{j-1}\to-\1$ respectively.
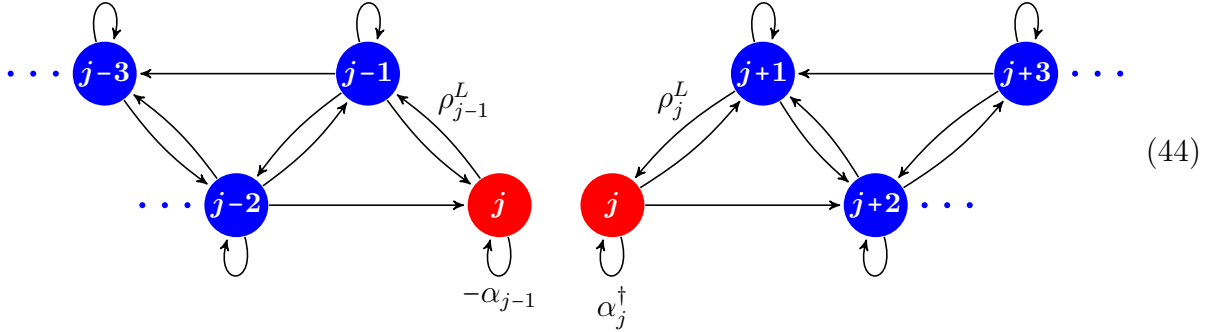
\begin{equation} \label{tz:CMV-ODD-SPLIT}
\begin{tikzpicture}[baseline=(current bounding box.center),->,>=stealth',shorten >=1pt,auto,semithick]
  \tikzstyle{every state}=[node distance=3.5cm,fill=blue,draw=none,text=white,inner sep=0pt,minimum size=0.85cm]
  \tikzstyle{every place}=[node distance=0.3cm,fill=blue,draw=none,inner sep=0pt,minimum size=0.08cm]

  \node[state] (Cu) {\small $\boldsymbol{j\!-\!2}$};
  \node[state] (Du) [right of=Cu,fill=red] {\small $\boldsymbol{j}$};
  \node[state] (Du2) [right of=Du,fill=red,node distance=1.5cm]
  {\small $\boldsymbol{j}$};
  \node[state] (Eu) [right of=Du2] {\small $\boldsymbol{j\!+\!2}$};

  \path (Cu) ++(-0.35,0) coordinate (C0u);

  \node[place] (C1u) [left of=C0u] {};
  \node[place] (C2u) [left of=C1u] {};
  \node[place] (C3u) [left of=C2u] {};

  \path (Eu) ++(0.35,0) coordinate (E0u);

  \node[place] (E1u) [right of=E0u] {};
  \node[place] (E2u) [right of=E1u] {};
  \node[place] (E3u) [right of=E2u] {};

  \node[state] (Bd) at (-1.75,1.75) {\small $\boldsymbol{j\!-\!3}$};
  \node[state] (Cd) [right of=Bd] {\small $\boldsymbol{j\!-\!1}$};
  \node[state] (Dd) at (7,1.75) {\small $\boldsymbol{j\!+\!1}$};
  \node[state] (Ed) [right of=Dd] {\small $\boldsymbol{j\!+\!3}$};

  \path (Bd) ++(-0.35,0) coordinate (B0d);

  \node[place] (B1d) [left of=B0d] {};
  \node[place] (B2d) [left of=B1d] {};
  \node[place] (B3d) [left of=B2d] {};

  \path (Ed) ++(0.35,0) coordinate (E0d);

  \node[place,fill=blue] (E1d) [right of=E0d] {};
  \node[place,fill=blue] (E2d) [right of=E1d] {};
  \node[place,fill=blue] (E3d) [right of=E2d] {};

  \path (Cu) edge			   	node {} (Du)
             edge [loop below]	node {} (Cu)
             edge [bend right=9]	node {} (Cd)
             edge [bend right=9]	node {} (Bd)
        (Du) edge [bend right=9,above] 	node {$\kern17pt \rho_{j-1}^L$} (Cd)
             edge [loop below]	node {$-\alpha_{j-1}$} (Du)
       (Du2) edge			   	node {} (Eu)
             edge [loop below]	node {$\alpha_j^\dag$} (Du2)
             edge [bend right=9]	node {} (Dd)
        (Eu) edge [bend right=9] 	node {} (Dd)
             edge [loop below]	node {} (Eu)
             edge [bend right=9]	node {} (Ed)
        (Bd) edge [loop above]	node {} (Bd)
             edge [bend right=9]	node {} (Cu)
        (Cd) edge 				node {} (Bd)
             edge [bend right=9]	node {} (Du)
             edge [loop above]	node {} (Cd)
             edge [bend right=9]	node {} (Cu)
        (Dd) edge [bend right=9]	node {} (Eu)
             edge [loop above]	node {} (Dd)
             edge [bend right=9,above]	node {$\rho_j^L\kern5pt$} (Du2)
        (Ed) edge			   	node {} (Dd)
             edge [loop above]	node {} (Ed)
             edge [bend right=9]	node {} (Eu);
\end{tikzpicture}
\end{equation}
For convenience, the two diagrams in \eqref{tz:CMV-ODD-SPLIT} show explicitly only the amplitudes which differ from those of \eqref{tz:CMV-ODD}. To distinguish the mathematical objects associated with \eqref{tz:CMV-ODD} from those related to the left and right diagrams of \eqref{tz:CMV-ODD-SPLIT}, we will denote the last ones with a superscript $L$ or $R$, respectively. Then, the coincidence of amplitudes between \eqref{tz:CMV-ODD} and \eqref{tz:CMV-ODD-SPLIT}, together with the impossibility of passing from the left to the right without crossing the site $j$ in \eqref{tz:CMV-ODD}, implies that
\begin{equation} \label{eq:W=LR}
 \ms{A}^{V_j,L}_{j,j-1} = \ms{A}^{V_j}_{j,j-1},
 \qquad
 \ms{A}^{V_j,R}_{j+1,j} = \ms{A}^{V_j}_{j+1,j}.
\end{equation}
This coincidence is key to connect the factorization of $a_{V_j}$ with $a_{V_j}^{L,R}$. Actually, a splitting of the sum over loops for $a_{V_j}^{L,R}$, similar to that one leading to \eqref{eq:GF-SPLIT}, ends up in
\begin{equation} \label{eq:GF-LR}
 a_{V_j}^L = -\alpha_{j-1} z + \ms{A}^{V_j}_{j,j-1} \, (\rho_{j-1}^L z),
 \qquad
 a_{V_j}^R = \alpha_j^\dag z + (\rho_j^L z) \, \ms{A}^{V_j}_{j+1,j}.
\end{equation}

Hence, the factorization \eqref{eq:GF-SPLIT} reads as
\[
 a_{V_j} = z^{-1} a_{V_j}^L \, a_{V_j}^R,
\]
and can be rewritten in terms of the Schur functions of $V_j$ as
\[
 f_{V_j} = f_{V_j}^R \, f_{V_j}^L.
\]
The identification of the left and right diagrams of \eqref{tz:CMV-ODD-SPLIT} as the representations of $\mc{C}_j$ and $\mc{C}^{(j)}$ leads to the equalities $f_{V_j}^R=f_j$ and $f_{V_j}^L=b_j$ following the arguments in the proof of Theorem \ref{thm:MATRIX-K-1}.

The result that connects Schur functions to first return generating functions gives a quantum dynamical meaning to Schur functions and their Taylor coefficients. The previous proof of Khrushchev's formula also uncovers the quantum dynamical content of the Schur iterates and inverse Schur iterates: a CMV quantum walk can be split at any site $j$ into a left and a right CMV walk so that the first return generating function of the site $j$ becomes, up to a factor $z$, the product of the left and right ones; the $j$-th iterate encodes the contributions of the right walk to the first return amplitudes of the site $j$, while the $j$-th inverse iterate comprises the contributions of the left walk.

The path counting method also sheds light on the ``magic" of Theorem \ref{thm:MATRIX-K-2}. This will be shown for instance in the case of the Schur function $f_{[j,k]}$ for the measure $\mu_{[j,k]}$ given in \eqref{eq:MM-MEASURE}, i.e. the $\mc{C}$-Schur function $f_{V_{[j,k]}}$ of the subspace $V_{[j,k]}$. We can recover this Schur function from the corresponding generating function $a_{V_{[j,k]}}(z)=zf_{V_{[j,k]}}^\dag(z)$ which, due to \eqref{eq:a-U}, has the block structure $a_{V_{[j,k]}}=(\ms{A}^{V_{[j,k]}}_{r,s})_{j\le r,s\le k}$, with $\ms{A}^{V_{[j,k]}}_{r,s}$ given in \eqref{eq:matrix-paths}.

Assume that $j$ and $k$ are even (the other cases can be analyzed in a similar fashion). Representing in red the sites related to $V_{[j,k]}$ leads to the following diagram for $\mc{C}$,
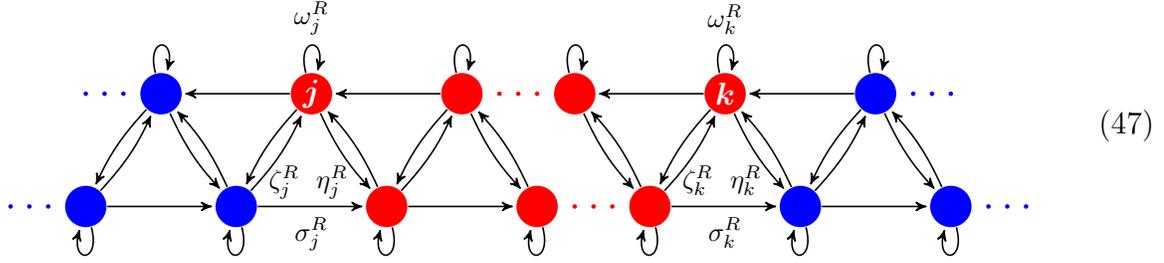
\begin{equation} \label{tz:CMV-EVEN-EVEN}
\begin{tikzpicture}[baseline=(current bounding box.center),->,>=stealth',shorten >=1pt,auto,semithick]
  \tikzstyle{every state}=[node distance=2cm,fill=blue,draw=none,text=white,inner sep=0pt,minimum size=0.55cm]
  \tikzstyle{every place}=[node distance=0.25cm,fill=blue,draw=none,inner sep=0pt,minimum size=0.06cm]

  \node[state] (Cu) {};
  \node[state] (Du) [right of=Cu,fill=red] {$\boldsymbol j$};
  \node[state] (Eu) [right of=Du,fill=red] {};

  \path (Cu) ++(-0.25,0) coordinate (C0u);

  \node[place] (C1u) [left of=C0u] {};
  \node[place] (C2u) [left of=C1u] {};
  \node[place] (C3u) [left of=C2u] {};

  \path (Eu) ++(0.25,0) coordinate (E0u);

  \node[place] (E1u) [right of=E0u,fill=red] {};
  \node[place] (E2u) [right of=E1u,fill=red] {};
  \node[place] (E3u) [right of=E2u,fill=red] {};

  \node[state] (Bd) at (-1,-1.5) {};
  \node[state] (Cd) [right of=Bd] {};
  \node[state] (Dd) [right of=Cd,fill=red] {};
  \node[state] (Ed) [right of=Dd,fill=red] {};

  \path (Bd) ++(-0.25,0) coordinate (B0d);

  \node[place,fill=blue] (B1d) [left of=B0d] {};
  \node[place,fill=blue] (B2d) [left of=B1d] {};
  \node[place,fill=blue] (B3d) [left of=B2d] {};

  \path (Ed) ++(0.25,0) coordinate (E0d);

  \node[place,fill=blue] (E1d) [right of=E0d,fill=red] {};
  \node[place,fill=blue] (E2d) [right of=E1d,fill=red] {};
  \node[place,fill=blue] (E3d) [right of=E2d,fill=red] {};

  \node[state] (Fu) at (5.5,0) [fill=red] {};
  \node[state] (Gu) [right of=Fu,fill=red] {$\boldsymbol k$};
  \node[state] (Hu) [right of=Gu] {};

  \path (Hu) ++(0.25,0) coordinate (H0u);

  \node[place] (H1u) [right of=H0u] {};
  \node[place] (H2u) [right of=H1u] {};
  \node[place] (H3u) [right of=H2u] {};

  \node[state] (Fd) at (6.5,-1.5) [fill=red] {};
  \node[state] (Gd) [right of=Fd] {};
  \node[state] (Hd) [right of=Gd] {};

  \path (Hd) ++(0.25,0) coordinate (H0d);

  \node[place] (H1d) [right of=H0d] {};
  \node[place] (H2d) [right of=H1d] {};
  \node[place] (H3d) [right of=H2d] {};

  \path (Cu) edge [bend right=9,above] 	node {} (Bd)
             edge [loop above]			node {} (Cu)
             edge [bend right=9,below]	node {} (Cd)
        (Du) edge [above]			   		node {} (Cu)
             edge [bend right=9,above] 	node {} (Cd)
             edge [loop above]			node
             		{\footnotesize $\omega_j^R$} (Du)
             edge [bend right=9,below]	node
               {\footnotesize $\eta_j^R\kern8pt$} (Dd)
        (Eu) edge [above]					node {} (Du)
             edge [bend right=9,above] 	node {} (Dd)
             edge [loop above]			node {} (Eu)
             edge [bend right=9,below]	node {} (Ed)
        (Fu) edge [loop above]			node {} (Fu)
             edge [bend right=9,below]	node {} (Fd)
        (Gu) edge [above]			   		node {} (Fu)
             edge [bend right=9,above] 	node {} (Fd)
             edge [loop above]			node
             		{\footnotesize $\omega_k^R$} (Gu)
             edge [bend right=9,below]	node
               {\footnotesize $\eta_k^R\kern8pt$} (Gd)
        (Hu) edge [above]			   		node {} (Gu)
             edge [bend right=9,above] 	node {} (Gd)
             edge [loop above]			node {} (Hu)
             edge [bend right=9,below]	node {} (Hd)
        (Bd) edge [below]					node {} (Cd)
             edge [bend right=9,below]	node {} (Cu)
             edge [loop below]			node {} (Bd)
        (Cd) edge [below]			  		node
             		{\footnotesize $\sigma_j^R$} (Dd)
             edge [bend right=9,below]	node
              {\footnotesize $\kern3pt\zeta_j^R$} (Du)
             edge [loop below]			node {} (Cd)
             edge [bend right=9,above]	node {} (Cu)
        (Dd) edge [below]			  		node {} (Ed)
             edge [bend right=9,below]	node {} (Eu)
             edge [loop below]			node {} (Dd)
             edge [bend right=9,above]	node {} (Du)
        (Ed) edge [loop below]			node {} (Ed)
             edge [bend right=9,above]	node {} (Eu)
        (Fd) edge [below]					node
        				{\footnotesize $\sigma_k^R$} (Gd)
             edge [bend right=9,below]	node
              {\footnotesize $\kern3pt\zeta_k^R$} (Gu)
             edge [loop below]			node {} (Fd)
             edge [bend right=9,above]	node {} (Fu)
        (Gd) edge [below]			  		node {} (Hd)
             edge [bend right=9,below]	node {} (Hu)
             edge [loop below]			node {} (Gd)
             edge [bend right=9,above]	node {} (Gu)
        (Hd) edge [loop below]			node {} (Hd)
             edge [bend right=9,above]	node {} (Hu);
\end{tikzpicture}
\end{equation}
which can be split into left\,$|$\,center\,$|$\,right diagrams by setting $\alpha_j\to\1 \mid \alpha_{j-1}\to-\1$, $\alpha_k\to\1 \mid \alpha_{k-1}\to-\1$ respectively.
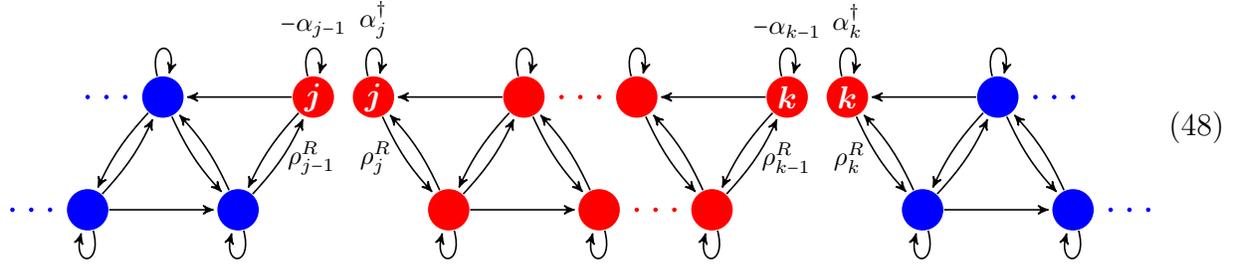
\begin{equation} \label{tz:CMV-EVEN-EVEN-SPLIT}
\kern-6pt
\begin{tikzpicture}[baseline=(current bounding box.center),->,>=stealth',shorten >=1pt,auto,semithick]
  \tikzstyle{every state}=[node distance=2cm,fill=blue,draw=none,text=white,inner sep=0pt,minimum size=0.55cm]
  \tikzstyle{every place}=[node distance=0.25cm,fill=blue,draw=none,inner sep=0pt,minimum size=0.06cm]

  \node[state] (Cu) {};
  \node[state] (Du) [right of=Cu,fill=red] {$\boldsymbol j$};
  \node[state] (Du2) at (2.8,0) [fill=red] {$\boldsymbol j$};
  \node[state] (Eu) [right of=Du2,fill=red] {};

  \path (Cu) ++(-0.25,0) coordinate (C0u);

  \node[place] (C1u) [left of=C0u] {};
  \node[place] (C2u) [left of=C1u] {};
  \node[place] (C3u) [left of=C2u] {};

  \path (Eu) ++(0.25,0) coordinate (E0u);

  \node[place] (E1u) [right of=E0u,fill=red] {};
  \node[place] (E2u) [right of=E1u,fill=red] {};
  \node[place] (E3u) [right of=E2u,fill=red] {};

  \node[state] (Bd) at (-1,-1.5) {};
  \node[state] (Cd) [right of=Bd] {};
  \node[state] (Dd) at (3.8,-1.5) [fill=red] {};
  \node[state] (Ed) [right of=Dd,fill=red] {};

  \path (Bd) ++(-0.25,0) coordinate (B0d);

  \node[place,fill=blue] (B1d) [left of=B0d] {};
  \node[place,fill=blue] (B2d) [left of=B1d] {};
  \node[place,fill=blue] (B3d) [left of=B2d] {};

  \path (Ed) ++(0.25,0) coordinate (E0d);

  \node[place,fill=blue] (E1d) [right of=E0d,fill=red] {};
  \node[place,fill=blue] (E2d) [right of=E1d,fill=red] {};
  \node[place,fill=blue] (E3d) [right of=E2d,fill=red] {};

  \node[state] (Fu) at (6.3,0) [fill=red] {};
  \node[state] (Gu) [right of=Fu,fill=red] {$\boldsymbol k$};
  \node[state] (Gu2) at (9.1,0) [fill=red] {$\boldsymbol k$};
  \node[state] (Hu) [right of=Gu2] {};

  \path (Hu) ++(0.25,0) coordinate (H0u);

  \node[place] (H1u) [right of=H0u] {};
  \node[place] (H2u) [right of=H1u] {};
  \node[place] (H3u) [right of=H2u] {};

  \node[state] (Fd) at (7.3,-1.5) [fill=red] {};
  \node[state] (Gd) at (10.1,-1.5) {};
  \node[state] (Hd) [right of=Gd] {};

  \path (Hd) ++(0.25,0) coordinate (H0d);

  \node[place] (H1d) [right of=H0d] {};
  \node[place] (H2d) [right of=H1d] {};
  \node[place] (H3d) [right of=H2d] {};

  \path (Cu) edge [bend right=9,above] 	node {} (Bd)
             edge [loop above]			node {} (Cu)
             edge [bend right=9,below]	node {} (Cd)
        (Du) edge [above]			   		node {} (Cu)
             edge [bend right=9,above] 	node {} (Cd)
             edge [loop above]			node
             {\footnotesize $-\alpha_{j-1}$} (Du)
       (Du2) edge [loop above] 			node
             {\footnotesize $\alpha_j^\dag$} (Du2)
             edge [bend right=9,left]	node
             {\footnotesize $\rho_j^R\kern1pt$} (Dd)
        (Eu) edge [above]					node {} (Du2)
             edge [bend right=9,above] 	node {} (Dd)
             edge [loop above]			node {} (Eu)
             edge [bend right=9,below]	node {} (Ed)
        (Fu) edge [loop above]			node {} (Fu)
             edge [bend right=9,below]	node {} (Fd)
        (Gu) edge [above]			   		node {} (Fu)
             edge [bend right=9,above] 	node {} (Fd)
             edge [loop above]			node
             	 {\footnotesize $-\alpha_{k-1}$} (Gu)
       (Gu2) edge [bend right=9,left] 	node
               {\footnotesize $\rho_k^R\kern1pt$} (Gd)
             edge [loop above]			node
             	 {\footnotesize $\alpha_k^\dag$} (Gu2)
        (Hu) edge [above]					node {} (Gu2)				
             edge [loop above]			node {} (Hu)
             edge [bend right=9,below]	node {} (Hd)
             edge [bend right=9,above]	node {} (Gd)
        (Bd) edge [below]					node {} (Cd)
             edge [bend right=9,below]	node {} (Cu)
             edge [loop below]			node {} (Bd)
        (Cd) edge [bend right=9,right]	node
          {\footnotesize $\kern-2pt\rho_{j-1}^R$} (Du)
             edge [loop below]			node {} (Cd)
             edge [bend right=9,above]	node {} (Cu)
        (Dd) edge [below]			  		node {} (Ed)
             edge [bend right=9,below]	node {} (Eu)
             edge [loop below]			node {} (Dd)
             edge [bend right=9,above]	node {} (Du2)
        (Ed) edge [loop below]			node {} (Ed)
             edge [bend right=9,above]	node {} (Eu)
        (Fd) edge [bend right=9,right]	node
        	  {\footnotesize $\kern-2pt\rho_{k-1}^R$} (Gu)
             edge [loop below]			node {} (Fd)
             edge [bend right=9,above]	node {} (Fu)
        (Gd) edge [below]					node {} (Hd)
             edge [loop below]			node {} (Gd)
             edge [bend right=9,above]	node {} (Gu2)
             edge [bend right=9,below]	node {} (Hu)
        (Hd) edge [loop below]			node {} (Hd)
             edge [bend right=9,above]	node {} (Hu);
\end{tikzpicture}
\end{equation}
For convenience, in the diagrams \eqref{tz:CMV-EVEN-EVEN} and \eqref{tz:CMV-EVEN-EVEN-SPLIT} we only write explicitly the amplitudes which differ from each other.

Let us use the superscripts $L,C,R$ to denote the mathematical objects related to the left, center and right diagrams of \eqref{tz:CMV-EVEN-EVEN-SPLIT}. These diagrams correspond to the block CMV matrices $\mc{C}_j$, $\mc{C}_{(j,k)}$ and $\mc{C}^{(k)}$ respectively. According to \eqref{eq:aH}, the first return generating function of the whole Hilbert space $V_{[j,k]}$ in the center diagram is
\[
 a_{V_{[j,k]}}^C(z) = z\,\mc{C}_{(j,k)}.
\]
Due to the coincidence of one-step amplitudes, this center diagram has the same blocks $\ms{A}^{V_{[j,k]}}_{r,s}$ as \eqref{tz:CMV-EVEN-EVEN}, except for the indices $(r,s)\in\{(j,j),(j+1,j),(k,k-1),(k,k)\}$. Therefore, $a_{V_{[j,k]}}(z)$ coincides with $z\,\mc{C}_{(j,k)}$, and thus with $z\,\mc{C}_{[j,k]}$, up to the four mentioned blocks, which are the only ones in which $\mc{C}_{(j,k)}$ and $\mc{C}_{[j,k]}$ differ from each other.

Examining the diagrams \eqref{tz:CMV-EVEN-EVEN}, \eqref{tz:CMV-EVEN-EVEN-SPLIT} and using \eqref{eq:OZES} we find that
\begin{equation} \label{eq:a-A-GEN}
 \begin{aligned}
 	& a_{V_j}^L =
 	-\alpha_{j-1} z + (\rho_{j-1}^R z) \, \ms{A}^{V_j}_{j-1,j},
 	& & a_{V_k}^R =
 	\alpha_k^\dag z + \ms{A}^{V_k}_{k,k+1} \, (\rho_k^R z),
	\\
	& \ms{A}^{V_{[j,k]}}_{j,j} =
	\omega_j^R z + (\zeta_j^R z) \, \ms{A}^{V_j}_{j-1,j} =
	\alpha_j^\dag a_{V_j}^L,
	& & \ms{A}^{V_{[j,k]}}_{k,k-1} =
 	\zeta_k^R z + \ms{A}^{V_k}_{k,k+1} \, (\sigma_k^R z) =
	a_{V_k}^R \rho_{k-1}^R,
	\\
	& \ms{A}^{V_{[j,k]}}_{j+1,j} =
	\eta_j^R z + (\sigma_j^R z) \, \ms{A}^{V_j}_{j-1,j} =
	\rho_j^R a_{V_j}^L,
	& & \ms{A}^{V_{[j,k]}}_{k,k} =
 	\omega_k^R z + \ms{A}^{V_k}_{k,k+1} \, (\eta_k^R z) =
	- a_{V_k}^R \alpha_{k-1},
 \end{aligned}
\end{equation}
where we have used that $\ms{A}^{V_j,L}_{j-1,j}=\ms{A}^{V_j}_{j-1,j}$ and $\ms{A}^{V_k,R}_{k,k+1}=\ms{A}^{V_k}_{k,k+1}$ for similar reasons as those giving \eqref{eq:W=LR}.

Taking into account that
\[
\begin{aligned}
 & \left(\mc{C}_{[j,k]}\right)_{j,j} =
 \omega_j^R = -\alpha_j^\dag\alpha_{j-1},
 & \quad & \left(\mc{C}_{[j,k]}\right)_{k,k-1} =
 \zeta_k^R = \alpha_k^\dag\rho_{k-1}^R,
 \\
 & \left(\mc{C}_{[j,k]}\right)_{j+1,j} =
 \eta_j^R = -\rho_j^R\alpha_{j-1},
 & & \left(\mc{C}_{[j,k]}\right)_{k,k} =
 \omega_k^R = -\alpha_k^\dag\alpha_{k-1},
\end{aligned}
\]
\eqref{eq:a-A-GEN} shows that the blocks of $a_{V_{[j,k]}}$ which differ from those of $z\,\mc{C}_{[j,k]}$ are obtained by substituting $-\alpha_{j-1}\to z^{-1}a_{V_j}^L$ and $\alpha_k^\dag\to z^{-1}a_{V_k}^R$ into the analogous ones of $z\,\mc{C}_{[j,k]}$. These blocks capture all the dependence of $\mc{C}_{[j,k]}$ on $\alpha_{j-1}$ and $\alpha_k$, so we conclude that
\[
 a_{V_{[j,k]}}(z) = z\,\mc{C}_{[j,k]}\bigg|_\text{\footnotesize
 $-\alpha_{j-1}\to z^{-1}a_{V_j}^L, \, \alpha_k^\dag\to z^{-1}a_{V_k}^R$}.
\]
Since $a_{V_{[j,k]}}(z)=zf_{V_{[j,k]}}^\dag\!(z)$ and the proof of Theorem~\ref{thm:MATRIX-K-1} shows that $a_{V_j}^L(z)=zb_j^\dag(z)$ and $a_{V_k}^R(z)=zf_k^\dag(z)$, the above identity reproduces the result of Theorem \ref{thm:MATRIX-K-2}.

\section{Matrix Khrushchev's formulas from Hessenberg matrices}
\label{sec:HESS-K}

An advantage of the abstract formulation of Khrushchev's formula given in Theorem \ref{thm:GEN-K} is that it is also applicable to unitary matrices which are not necessarily banded. This kind of matrices appear for instance when studying the analogue of Theorem \ref{thm:MATRIX-K-2} for orthonormal polynomial modifications of a measure. That is, if $\varphi_j^{L,R}$ are the orthonormal polynomials with respect to a matrix measure $\mu$ on $\T$, we search for the Schur functions of the measures
\begin{equation} \label{eq:nu}
\begin{aligned}
  & d\nu_{[j,k]} = \text{\footnotesize$
 \begin{pmatrix}
 	(\varphi_j^R)^\dag\,d\mu\,\varphi_j^R
	& (\varphi_j^R)^\dag\,d\mu\,\varphi_{j+1}^R
	& \dots
	& (\varphi_j^R)^\dag\,d\mu\,\varphi_k^R
	\\ \noalign{\vskip3pt}
	(\varphi_{j+1}^R)^\dag\,d\mu\,\varphi_j^R
	& (\varphi_{j+1}^R)^\dag\,d\mu\,\varphi_{j+1}^R
	& \dots
	& (\varphi_{j+1}^R)^\dag\,d\mu\,\varphi_k^R
	\\
	\vdots & \vdots & & \vdots
	\\
	(\varphi_k^R)^\dag\,d\mu\,\varphi_j^R
	& (\varphi_k^R)^\dag\,d\mu\,\varphi_{j+1}^R
	& \dots
	& (\varphi_k^R)^\dag\,d\mu\,\varphi_k^R
 \end{pmatrix}$},
 \\[2pt]
 & d\hat{\nu}_{[j,k]} = \text{\footnotesize$
 \begin{pmatrix}
 	\varphi_j^L\,d\mu\,(\varphi_j^L)^\dag
	& \varphi_j^L\,d\mu\,(\varphi_{j+1}^L)^\dag
	& \dots
	& \varphi_j^L\,d\mu\,(\varphi_k^L)^\dag
	\\ \noalign{\vskip3pt}
	\varphi_{j+1}^L\,d\mu\,(\varphi_j^L)^\dag
	& \varphi_{j+1}^L\,d\mu\,(\varphi_{j+1}^L)^\dag
	& \dots
	& \varphi_{j+1}^L\,d\mu\,(\varphi_k^L)^\dag
	\\
	\vdots & \vdots & & \vdots
	\\
	\varphi_k^L\,d\mu\,(\varphi_j^L)^\dag
	& \varphi_k^L\,d\mu\,(\varphi_{j+1}^L)^\dag
	& \dots
	& \varphi_k^L\,d\mu\,(\varphi_k^L)^\dag
 \end{pmatrix}$}.
\end{aligned}
\end{equation}
The usefulness of Theorem \ref{thm:GEN-K} for this purpose rests on the identification of the above measures as spectral measures with respect to a unitary multiplication operator. Consider first $U_\mu$ as the multiplication operator \eqref{eq:multip-op} in the space $L^2_\mu$ of column vector functions with inner product $\<\cdot\,|\,\cdot\>_R$ given by \eqref{eq:inner-vector-R}. Then, $\nu_{[j,k]}$ is the $U_\mu$-spectral measure of the subspace spanned by the columns of $\varphi_j^R,\varphi_{j+1}^R,\dots,\varphi_k^R$. Likewise, $\hat{\nu}_{[j,k]}$ is the $U_\mu$-spectral measure of the subspace spanned by the rows of $\varphi_j^L,\varphi_{j+1}^L,\dots,\varphi_k^L$, understanding now $L^2_\mu$ as a space of row vector functions with inner product $\<\phi|\psi\>_L=\int\psi(z)\,d\mu(z)\,\phi(z)^\dag$.

The next step is to search for overlapping factorizations of $U_\mu$ adapted to the subspaces whose spectral measure we wish to study. Like in the CMV case, the starting point is a recurrence relation, but in this case for the orthonormal polynomials. This recurrence is given in terms of the Verblunsky coefficients $\alpha_j$ of $\mu$ and the matrices $\rho_j^{L,R}$ in \eqref{eq:rhoLR} by (see \cite{DaPuSi-MOPUC} and references therein)
\begin{equation} \label{eq:REC-OP}
 z\varphi_j^L - \rho_j^L \varphi_{j+1}^L = \alpha_j^\dag \varphi_j^{R,*},
 \qquad
 z\varphi_j^R - \varphi_{j+1}^R \rho_j^R = \varphi_j^{L,*} \alpha_j^\dag,
 \qquad
 p^*(z) = z^{\deg( p )} p^\dag(z^{-1}).
\end{equation}
Expanding the reversed polynomials $\varphi_j^{R,*}$ and $\varphi_j^{L,*}$ in terms of $\varphi_k^L$ and $\varphi_k^R$ respectively, the recurrence relations in \eqref{eq:REC-OP} become
\begin{equation} \label{eq:HESS}
 \begin{aligned}
  	& \varphi^R\,\mc{H}=z\varphi^R,
	& \quad & \varphi^R=(\varphi_0^R,\varphi_1^R,\dots),
	& \quad & \mc{H} = \lim_{j\to\infty}
	\Theta_0^{\scriptscriptstyle(\infty)}\Theta_1^{\scriptscriptstyle(\infty)}
	\cdots\,\Theta_j^{\scriptscriptstyle(\infty)},
	\\
	& \hat{\mc{H}}\varphi^L = z\varphi^L,
	& & \varphi^L=(\varphi_0^L,\varphi_1^L,\dots)^T,
	& & \hat{\mc{H}} = \lim_{j\to\infty}
	\Theta_j^{\scriptscriptstyle(\infty)}\cdots\,
	\Theta_1^{\scriptscriptstyle(\infty)}\Theta_0^{\scriptscriptstyle(\infty)},
 \end{aligned}
\end{equation}
where the limits must be understood in the strong topology and $\Theta_j^{\scriptscriptstyle(N)}$ is defined by extending $\Theta(\alpha_j)$ to order $(N+1)d$ so that $\Theta(\alpha_j)$ acts on $V_j \oplus V_{j+1}$, i.e.
\begin{equation}
 \Theta_j^{\scriptscriptstyle(N)} =
 \1_j \oplus \Theta(\alpha_j) \oplus \1_{N-j-1}.
\end{equation}
Explicitly, $\mc{H}=\mc{H}(\alpha_0,\alpha_1,\dots)$ and $\hat{\mc{H}}=\hat{\mc{H}}(\alpha_0,\alpha_1,\dots)$ are the block Hessenberg matrices
\begin{equation} \label{eq:H}
\begin{aligned}
 & \mc{H} =
 \text{\small $\begin{pmatrix}
 	\alpha_0^\dag & \rho_0^L\alpha_1^\dag & \rho_0^L\rho_1^L\alpha_2^\dag
 	& \rho_0^L\rho_1^L\rho_2^L\alpha_3^\dag & \dots
 	\\ \noalign{\vskip3pt}
 	\rho_0^R & -\alpha_0\alpha_1^\dag & -\alpha_0\rho_1^L\alpha_2^\dag
 	& -\alpha_0\rho_1^L\rho_2^L\alpha_3^\dag & \dots
 	\\ \noalign{\vskip3pt}
 	& \rho_1^R & -\alpha_1\alpha_2^\dag & -\alpha_1\rho_2^L\alpha_3^\dag & \dots
 	\\ \noalign{\vskip3pt}
 	&& \rho_2^R & -\alpha_2\alpha_3^\dag & \dots
 	\\ \noalign{\vskip3pt}
 	&&& \rho_3^R & \dots
 	\\
 	&&&& \dots
 \end{pmatrix}$},
 \\
 & \hat{\mc{H}} =
 \text{\small $\begin{pmatrix}
 	\alpha_0^\dag & \rho_0^L
	\\ \noalign{\vskip3pt}
	\alpha_1^\dag\rho_0^R & -\alpha_1^\dag\alpha_0 & \rho_1^L
	\\ \noalign{\vskip3pt}
	\alpha_2^\dag\rho_1^R\rho_0^R & -\alpha_2^\dag\rho_1^R\alpha_0
	& -\alpha_2^\dag\alpha_1 & \rho_2^L
	\\ \noalign{\vskip3pt}
	\alpha_3^\dag\rho_2^R\rho_1^R\rho_0^R
	& -\alpha_3^\dag\rho_2^R\rho_1^R\alpha_0 & -\alpha_3^\dag\rho_2^R\alpha_1
	& -\alpha_3^\dag\alpha_2 & \rho_3^L
	\\
	\dots & \dots & \dots & \dots & \dots & \dots
 \end{pmatrix}$},
\end{aligned}
\end{equation}
which, analogously to the CMV case, are related by $\hat{\mc{H}}(\{\alpha_j\})=\mc{H}(\{\alpha_j^T\})^T$.

The expressions of $\mc{H}$ and $\hat{\mc{H}}$ in \eqref{eq:HESS} lead naturally to the $V_j$-overlapping factorizations
\begin{equation} \label{eq:HESS-fac}
 \mc{H}=(\mc{H}_j\oplus\1_\infty)(\1_j\oplus\mc{H}^{(j)}),
 \qquad
 \hat{\mc{H}}=(\1_j\oplus\hat{\mc{H}}^{(j)})(\hat{\mc{H}}_j\oplus\1_\infty),
\end{equation}
where $\mc{H}^{(j)}:=\mc{H}(\alpha_j,\alpha_{j+1},\dots)$ and $\hat{\mc{H}}^{(j)}:=\hat{\mc{H}}(\alpha_j,\alpha_{j+1},\dots)$, while $\mc{H}_j=\mc{H}_j(\alpha_0,\dots,\alpha_{j-1})$ and $\hat{\mc{H}}_j=\hat{\mc{H}}_j(\alpha_0,\dots,\alpha_{j-1})$ are finite unitary Hessenberg matrices with the form
\begin{equation} \label{eq:HESSfin}
 \mc{H}_N := \Theta_0^{\scriptscriptstyle(N)}\Theta_1^{\scriptscriptstyle(N)}
 \cdots\,\Theta_{N-1}^{\scriptscriptstyle(N)},
 \qquad
 \hat{\mc{H}}_N := \Theta_{N-1}^{\scriptscriptstyle(N)}\cdots\,
 \Theta_1^{\scriptscriptstyle(N)}\Theta_0^{\scriptscriptstyle(N)}.
\end{equation}
For instance,
\[
 \mc{H}_2 =
 \begin{pmatrix}
 	\alpha_0^\dag & \rho_0^L\alpha_1^\dag & \rho_0^L\rho_1^L
 	\\ \noalign{\vskip3pt}
 	\rho_0^R & -\alpha_0\alpha_1^\dag & -\alpha_0\rho_1^L
 	\\ \noalign{\vskip3pt}
 	& \rho_1^R & -\alpha_1
 \end{pmatrix},
 \qquad
 \hat{\mc{H}}_2 =
 \begin{pmatrix}
 	\alpha_0^\dag & \rho_0^L
	\\ \noalign{\vskip3pt}
	\alpha_1^\dag\rho_0^R & -\alpha_1^\dag\alpha_0 & \rho_1^L
	\\ \noalign{\vskip3pt}
	\rho_1^R\rho_0^R & -\rho_1^R\alpha_0 & -\alpha_1
 \end{pmatrix}.
\]

In spite of these results, the application of Theorem \ref{thm:GEN-K} to $\mc{H}$ and $\hat{\mc{H}}$ has a serious drawback, since the vector polynomials are not always dense in the Hilbert space $L^2_\mu$ of $d$-vector functions. Thus, $\mc{H}$ and $\hat{\mc{H}}$ do not represent in general the whole multiplication operator $U_\mu$, but its restriction to the closure in $L^2_\mu$ of the set $\C^d[z]$ of $d$-vector polynomials. As a consequence, the operators $X\mapsto\mc{H}X$ and $X^T\mapsto X^T\hat{\mc{H}}$ defined in $\ell^2$ by these block Hessenberg matrices are isometric but not always unitary. The non unitarity prevents us from using all the results previously developed about the connection between Schur functions and spectral measures, which are based on the spectral theorem for unitary operators. We will overcome this difficulty by studying the non unitary cases as limiting situations of the unitary ones.

The lack of unitarity of these block Hessenberg matrices is accounted for by the identities
\[
 \begin{gathered}
 	\hat{\mc{H}}\hat{\mc{H}}^\dag = \mc{H}^\dag\mc{H} = \1_\infty,
 	\\
	\begin{aligned}
 		& \mc{H}\mc{H}^\dag = \1_\infty-\Delta,
		& \qquad & \Delta_{j,k} =
		\lim_{N\to\infty}
		\alpha_{j-1}\left(\rho_j^L\rho_{j+1}^L\cdots\rho_N^L\right)
		\left(\rho_N^L\cdots\rho_{k+1}^L\rho_k^L\right)\alpha_{k-1}^\dag,
    	\\
 		& \hat{\mc{H}}^\dag\hat{\mc{H}} = \1_\infty-\hat\Delta,
    	& & \hat{\Delta}_{j,k} =
		\lim_{N\to\infty}
		\alpha_{j-1}^\dag\left(\rho_j^R\rho_{j+1}^R\cdots\rho_N^R\right)
		\left(\rho_N^R\cdots\rho_{k+1}^R\rho_k^R\right)\alpha_{k-1},	
	\end{aligned}
 \end{gathered}
\]
where we use the convention $\alpha_{-1}=-\1$. Therefore, $\mc{H}$ and $\hat{\mc{H}}$ are unitary only when $\lim_{N\to\infty}\rho_0^L\rho_1^L\cdots\rho_N^L=\0$ and $\lim_{N\to\infty}\rho_N^R\cdots\rho_1^R\rho_0^R=\0$ respectively. Actually, these two conditions are equivalent to each other and also to $\sum_{j\ge0}\|\alpha_j\|^2=\infty$, which we will write $(\alpha_j)\notin\ell^2$.

On the other hand, the unitarity of the block Hessenberg matrices simply reflects  the unitarity of the operator $U_\mu$ restricted to the closure $\overline{\C^d[z]}$ of the vector polynomials, i.e. $z\,\overline{\C^d[z]}=\overline{\C^d[z]}$. From this equality it follows that $\overline{\C^d[z]}$ covers the vector Laurent polynomials and thus coincides with $L^2_\mu$. Hence, the condition $(\alpha_j)\notin\ell^2$ guaranteeing the unitarity of $\mc{H}$ and $\hat{\mc{H}}$ also ensures that they are true representations of $U_\mu$, which becomes unitarily equivalent to the corresponding block Hessenberg operators in $\ell^2$. This unitary equivalence assigns the canonical subspace $V_j$ to the subspace of $L^2_\mu$ spanned by the columns (rows) of $\varphi_j^R$ ($\varphi_j^L$). Therefore, the unitarity condition $(\alpha_j)\notin\ell^2$ allows us to identify the $U_\mu$-spectral measures $\nu_{[j,k]}$, $\hat{\nu}_{[j,k]}$ as the spectral measures of $V_{[j,k]} = V_j \oplus \cdots \oplus V_k$ with respect to $\mc{H}$, $\hat{\mc{H}}$ respectively. In particular, $\mu$ is at the same time the spectral measure of the first canonical subspace $V_0$ with respect to $\mc{H}$ and $\hat{\mc{H}}$.

To deal with the extension of Theorem \ref{thm:MATRIX-K-2} to the measures \eqref{eq:nu} it only remains to introduce, analogously to the CMV case, block Hessenberg submatrices and block Hessenberg unitary truncations. We define $\mc{H}_{[j,k]}=\mc{H}_{[j,k]}(\alpha_{j-1},\dots,\alpha_k)$ and $\hat{\mc{H}}_{[j,k]}=\hat{\mc{H}}_{[j,k]}(\alpha_{j-1},\dots,\alpha_k)$ respectively as the submatrices of $\mc{H}$ and $\hat{\mc{H}}$ corresponding to the subspace $V_{[j,k]}$. For instance
\[
 \mc{H}_{[1,3]} = \text{\small$
 \begin{pmatrix}
 	-\alpha_0\alpha_1^\dag & -\alpha_0\rho_1^L\alpha_2^\dag &
	-\alpha_0\rho_1^L\rho_2^L\alpha_3^\dag
 	\\ \noalign{\vskip3pt}
 	\rho_1^R & -\alpha_1\alpha_2^\dag & -\alpha_1\rho_2^L\alpha_3^\dag
 	\\ \noalign{\vskip3pt}
 	& \rho_2^R & -\alpha_2\alpha_3^\dag
 \end{pmatrix}$},
 \qquad
 \hat{\mc{H}}_{[1,3]} = \text{\small$
 \begin{pmatrix}
 	-\alpha_1^\dag\alpha_0 & \rho_1^L
	\\ \noalign{\vskip3pt}
	-\alpha_2^\dag\rho_1^R\alpha_0 & -\alpha_2^\dag\alpha_1 & \rho_2^L
	\\ \noalign{\vskip3pt}
	-\alpha_3^\dag\rho_2^R\rho_1^R\alpha_0 & -\alpha_3^\dag\rho_2^R\alpha_1 &
	-\alpha_3^\dag\alpha_2
 \end{pmatrix}$}.
\]
Setting $\alpha_{j-1}\to\1$ and $\alpha_k\to\1$ in these submatrices we get the finite unitary truncations
\begin{equation} \label{eq:HESS-TRUNC}
 \begin{aligned}
 	& \mc{H}_{(j,k)}
 	:= \mc{H}_{[j,k]}(-\1,\alpha_j\dots,\alpha_{k-1},\1)
 	= \mc{H}_{k-j}(\alpha_j,\dots,\alpha_{k-1}) = \mc{H}_{k-j}^{(j)},
 	\\
 	& \hat{\mc{H}}_{(j,k)}
 	:= \hat{\mc{H}}_{[j,k]}(-\1,\alpha_j\dots,\alpha_{k-1},\1)
 	= \hat{\mc{H}}_{k-j}(\alpha_j,\dots,\alpha_{k-1}) = \hat{\mc{H}}_{k-j}^{(j)},
 \end{aligned}
\end{equation}
an example of which is
\[
 \mc{H}_{(1,3)} = \text{\small$
 \begin{pmatrix}
 	\alpha_1^\dag & \rho_1^L\alpha_2^\dag & \rho_1^L\rho_2^L
 	\\ \noalign{\vskip3pt}
 	\rho_1^R & -\alpha_1\alpha_2^\dag & -\alpha_1\rho_2^L
 	\\ \noalign{\vskip3pt}
 	& \rho_2^R & -\alpha_2
 \end{pmatrix}$}
 = \mc{H}^{(1)}_2,
 \qquad
 \hat{\mc{H}}_{[1,3]} = \text{\small$
 \begin{pmatrix}
 	\alpha_1^\dag & \rho_1^L
	\\ \noalign{\vskip3pt}
	\alpha_2^\dag\rho_1^R & -\alpha_2^\dag\alpha_1 & \rho_2^L
	\\ \noalign{\vskip3pt}
	\rho_2^R\rho_1^R & -\rho_2^R\alpha_1 & -\alpha_2
 \end{pmatrix}$}
 = \hat{\mc{H}}^{(1)}_2.
\]
Both kind of matrices are related by
\[
 \begin{aligned}
 	& \mc{H}_{[j,k]} =
	\left(-\alpha_{j-1} \oplus \1_{k-j} \right)
 	\mc{H}_{(j,k)}
	\left(\1_{k-j} \oplus \alpha_k^\dag \right),
	\\
 	& \hat{\mc{H}}_{[j,k]} =
 	\left(\1_{k-j} \oplus \alpha_k^\dag \right)
 	\hat{\mc{H}}_{(j,k)}
 	\left(-\alpha_{j-1} \oplus \1_{k-j} \right).
 \end{aligned}
\]

With this notation, the alluded extension of Theorem \ref{thm:MATRIX-K-2} reads as follows.

\begin{thm} \label{thm:MATRIX-K-4}
{\bf \small
[$2^{\text{nd}}$ generalized Khrushchev's formulas for matrix measures]}
\newline
Let $f$ be the Schur function of a non-trivial matrix measure $\mu$ on $\T$ with Verblunsky coefficients $\alpha_j$, orthonormal polynomials $\varphi_j^{L,R}$ and block Hessenberg matrices $\mc{H}$, $\hat{\mc{H}}$. If $f_j$ and $b_j$ are the iterates and inverse iterates of $f$, then the Schur functions $h_{[j,k]}$, $\hat{h}_{[j,k]}$ of the measures $\nu_{[j,k]}$, $\hat{\nu}_{[j,k]}$ given in \eqref{eq:nu} are the result of substituting {\rm $-\alpha_{j-1}^\dag \to b_j$} and $\alpha_k \to f_k$ into {\rm $\mc{H}_{[j,k]}^\dag$}, {\rm $\hat{\mc{H}}_{[j,k]}^\dag$} respectively.
More explicitly,
{\rm
\begin{equation} \label{eq:h}
 h_{[j,k]} =
 \left(\1_{k-j} \oplus f_k\right)
 \mc{H}_{(j,k)}^\dag
 \left(b_j \oplus \1_{k-j}\right),
 \qquad
 \hat{h}_{[j,k]} =
 \left(b_j \oplus \1_{k-j}\right)
 \hat{\mc{H}}_{(j,k)}^\dag
 \left(\1_{k-j} \oplus f_k\right).	
\end{equation}}
\end{thm}

\begin{proof}
Assume first that $(\alpha_l)\notin\ell^2$ so that $\mc{H}$, $\hat{\mc{H}}$, as well as $\mc{H}^{(j)}$, $\hat{\mc{H}^{(j)}}$, are unitary. Then, $h_{[j,k]}=f_{V_{[j,k]}}$ and $\hat{h}_{[j,k]}=\hat{f}_{V_{[j,k]}}$ are the Schur functions of $V_{[j,k]}$ with respect to $\mc{H}$ and $\hat{\mc{H}}$ respectively.

Applying Theorem \ref{thm:GEN-K} to the subspace $V=V_{[j,k]}$ and the $V_j$-overlapping factorizations given in \eqref{eq:HESS-fac}, we obtain
\[
 f_{V_{[j,k]}} =
 f^{(j)}_{V_{[j,k]}} (b_{V_j} \oplus \1_{k-j}),
 \qquad
 \hat{f}_{V_{[j,k]}} =
 (\hat{b}_{V_j} \oplus \1_{k-j}) \hat{f}^{(j)}_{V_{[j,k]}},
\]
with $f^{(j)}_{V_{[j,k]}}$, $f^{(j)}_{V_{[j,k]}}$ the Schur functions of $V_{[j,k]}$ with respect to $\mc{H}^{(j)}$, $\hat{\mc{H}}^{(j)}$, and $b_{V_j}$, $\hat{b}_{V_j}$ the Schur functions of $V_j$ with respect to $\mc{H}_j$, $\hat{\mc{H}}_j$.

We can also factorize $f^{(j)}_{V_{[j,k]}}$, $f^{(j)}_{V_{[j,k]}}$ by applying the $V_k$-overlapping factorizations \eqref{eq:HESS-fac} to $\mc{H}^{(j)}$, $\mc{H}^{(j)}$, i.e.
\[
 \mc{H}^{(j)} =
 (\mc{H}_{(j,k)}\oplus\1_\infty)(\1_{k-j}\oplus\mc{H}^{(k)}),
 \qquad
 \hat{\mc{H}}^{(j)} =
 (\1_{k-j}\oplus\hat{\mc{H}}^{(k)})(\hat{\mc{H}}_{(j,k)}\oplus\1_\infty).
\]
Then, we get
\[
 f^{(j)}_{V_{[j,k]}} =
 (\1_{k-j} \oplus f^{(k)}_{V_k}) f^{(j,k)}_{V_{[j,k]}},
 \qquad
 \hat{f}^{(j)}_{V_{[j,k]}} =
 \hat{f}^{(j,k)}_{V_{[j,k]}} (\1_{k-j} \oplus \hat{f}^{(k)}_{V_k}),
\]
where $f^{(j,k)}_{V_{[j,k]}}$, $\hat{f}^{(j,k)}_{V_{[j,k]}}$ are the Schur functions of $V_{[j,k]}$ with respect to $\mc{H}_{(j,k)}$, $\hat{\mc{H}}_{(j,k)}$
and $f^{(k)}_{V_k}$, $\hat{f}^{(k)}_{V_k}$ are the Schur functions of $V_k$ with respect to $\mc{H}^{(k)}$, $\hat{\mc{H}}^{(k)}$.

We conclude that
\[
 f_{V_{[j,k]}} =
 (\1_{k-j} \oplus f^{(k)}_{V_k})
 f^{(j,k)}_{V_{[j,k]}}
 (b_{V_j} \oplus \1_{k-j}),
 \qquad
 \hat{f}_{V_{[j,k]}} =
 (\hat{b}_{V_j} \oplus \1_{k-j})
 \hat{f}^{(j,k)}_{V_{[j,k]}}
 (\1_{k-j} \oplus \hat{f}^{(k)}_{V_k}).
\]

From \eqref{eq:fH}, $f^{(j,k)}_{V_{[j,k]}}=\mc{H}_{(j,k)}^\dag$ and $\hat{f}^{(j,k)}_{V_{[j,k]}}=\hat{\mc{H}}_{(j,k)}^\dag$ because $V_{[j,k]}$ is the whole Hilbert space where $\mc{H}_{(j,k)}$ and $\hat{\mc{H}}_{(j,k)}$ act. Besides, since $V_k$ is the first canonical subspace of the Hilbert space $\bigoplus_{l\ge k}V_l$ related to $\mc{H}^{(k)}$ and $\hat{\mc{H}}^{(k)}$, we find that $f^{(k)}_{V_k}$ and $\hat{f}^{(k)}_{V_k}$ are the Schur functions of the corresponding orthogonality measure, which has Verblunsky coefficients $(\alpha_k,\alpha_{k+1},\dots)$. Thus, $f^{(k)}_{V_k}=\hat{f}^{(k)}_{V_k}=f_k$. On the other hand, $V_j$ is the last canonical subspace regarding the Hilbert space $V_{[0,k]}$ associated with $\mc{H}_j$ and $\hat{\mc{H}}_j$. If $B_l=\{e_{ld},e_{ld+1},\dots,e_{ld+d-1}\}$ is the canonical basis of $V_l$, reordering the basis of $V_{[0,j]}$ as $B_j \cup B_{j-1} \cup \cdots \cup B_0$ forces $V_j$ to be the first subspace. From \eqref{eq:HESSfin} we see that this reordering also changes $\mc{H}_j(\{\alpha_l\}) \to \hat{\mc{H}}_j(\{-\alpha_{j-l-1}^\dag\})$ and $\hat{\mc{H}}_j(\{\alpha_l\}) \to \mc{H}_j(\{-\alpha_{j-l-1}^\dag\})$. Therefore, $b_{V_j}$ and $\hat{b}_{V_j}$ are the Schur functions of the measure with Verblunsky coefficients $(-\alpha_{j-1}^\dag,\dots,-\alpha_0^\dag,\1)$, i.e. $b_{V_j}=\hat{b}_{V_j}=b_j$.

The above arguments complete the proof when $(\alpha_l)\notin\ell^2$. To deal with the complementary case $\boldsymbol\alpha=(\alpha_l)\in\ell^2$ we will use a limiting procedure. Any such a square-summable sequence of Schur parameters can be understood as a pointwise limit of non square-summable ones $\boldsymbol\alpha^n=(\alpha^n_l)\notin\ell^2$, i.e. $\lim_{n\to\infty}\alpha^n_l=\alpha_l$ for all $l$. For instance, since $\lim_{l\to\infty}\alpha_l=0$, we can define $\alpha^n_l=\alpha_l+1/n$ for $n>(1-\sup_l\|\alpha_l\|)^{-1}$ so that $\|\alpha^n_l\|<1$, $\lim_{n\to\infty}\alpha^n_l=\alpha_l$ and $\lim_{l\to\infty}\alpha^n_l=1/n\neq0$.
Since $\boldsymbol\alpha^n\notin\ell^2$, we have already proved that the theorem holds for the measure $\mu^n$ with Verblunsky coefficients $\boldsymbol\alpha^n$.
Taking limits on this result as $n\to\infty$ and using Lemma \ref{lem:ASYMP} we find that the theorem is also valid for the measure $\mu$ with Verblunsky coefficients $\boldsymbol\alpha$.
\end{proof}

The last part of the previous proof requires an asymptotic result, proved in the followig lemma, relating the convergence of matrix Schur parameters, matrix measures, matrix orthogonal polynomials and matrix Schur functions. For convenience we will summarize the notation before stating the lemma: $\boldsymbol\alpha=(\alpha_l)$ are the Verblunsky coefficients of the non-trivial matrix measure $\mu$ on $\T$ with Carath\'eodory function $F$, orthonormal polynomials $\varphi_l^{L,R}$ and block Hessenberg matrices $\mc{H}$, $\hat{\mc{H}}$. Also, $f$ is the matrix Schur function of $\mu$, which has Schur parameters $\boldsymbol\alpha$, while $f_l$, $b_l$ are its iterates and inverse iterates. We will consider a sequence $\boldsymbol\alpha^n=(\alpha^n_l)$ of sequences of
Verblunsky coefficients, the related mathematical objects being denoted by adding the superscript $n$ to those of $\boldsymbol\alpha$. Moreover, we will use the following convergence notations
\[
 \begin{aligned}
	& \boldsymbol\alpha^n \to \boldsymbol\alpha
	&\quad& \lim_{n\to\infty}\alpha^n_l=\alpha_l
	\kern7pt \forall l
	&\quad& \text{ pointwise convergence}
	\\
	& \mu^n \overset{\ast}{\to} \mu
	&& \lim_{n\to\infty}\int \kern-3pt h\,d\mu^n = \int \kern-3pt h\,d\mu
	\kern7pt \forall h\in C(\mbb{T})
	&& \ast\text{-weak convergence}
	\\
	& g^n(z) \overset{D}{\rightrightarrows} g(z)
	&& \lim_{n\to\infty} \sup_{z \in K} \|g^n(z)-g(z)\| = 0
	\kern7pt \forall K \subset D \text{ compact}
	&& \text{ uniform convergence}
	\\ \noalign{\vskip-12pt}
	&&&&& \text{ on compact subsets}
 \end{aligned}
\]
where $C(\mbb{T})$ is the set of continuous scalar functions on $\mbb{T}$, and we understand that $g^n$ and $g$ are matrix-valued functions with domain $D \subset \mbb{C}$.

\begin{lem} \label{lem:ASYMP}
If $\boldsymbol\alpha^n \to \boldsymbol\alpha$, then
\[
 \mu^n \overset{\ast}{\to} \mu,
 \qquad
 \varphi^{L,n}_l \overset{\mbb{C}}{\rightrightarrows} \varphi^L_l,
 \qquad
 \varphi^{R,n}_l \overset{\mbb{C}}{\rightrightarrows} \varphi^R_l,
 \qquad
 f^n_l \overset{\mbb{D}}{\rightrightarrows} f_l,
 \qquad
 b^n_l \overset{\mbb{D}}{\rightrightarrows} b_l,
\]
while for the measures defined in \eqref{eq:nu} and the Schur functions given in \eqref{eq:h}
\[
 \nu^n_{[j,k]} \overset{\ast}{\to} \nu_{[j,k]},
 \qquad
 \hat{\nu}^n_{[j,k]} \overset{\ast}{\to} \hat{\nu}_{[j,k]},
 \qquad
 h^n_{[j,k]} \overset{\mbb{D}}{\rightrightarrows} h_{[j,k]},
 \qquad
 \hat{h}^n_{[j,k]} \overset{\mbb{D}}{\rightrightarrows} \hat{h}_{[j,k]}.
\]
\end{lem}

\begin{proof}
Suppose that $\boldsymbol\alpha^n \to \boldsymbol\alpha$. Let us prove first that $f^n_l \overset{\mbb{D}}{\rightrightarrows} f_l$ for any $l$. Let $g$ be any of the limit points of the sequence $f^n$ in the topology of the uniform convergence in $\mbb{D}$, i.e. $f^n \overset{\mbb{D}}{\rightrightarrows} g$ for some subsequence which we omit for convenience. Then, $g$ must be a matrix Schur function too.

Let us prove by induction on $l$ that $f^n_l \overset{\mbb{D}}{\rightrightarrows} g_l$. The result is obviously true for $l=0$. Assuming it for a given index $l$ implies that $\alpha_l=\lim_{n\to\infty}\alpha^n_l=\lim_{n\to\infty}f^n_l(0)=g_l(0)$, hence $f^n_l-\alpha^n_l \overset{\mbb{D}}{\rightrightarrows} g_l-g_l(0)$. Besides, in any compact subset $K\subset\mbb{D}$ containing the origin,
\[
 \begin{aligned}
	& \|(\1 - {\alpha^n_l}^\dag f^n_l)^{-1} - (\1 - g_l(0)^\dag g_l)^{-1}\|
    = \|(\1 - {\alpha^n_l}^\dag f^n_l)^{-1}
 		({\alpha^n_l}^\dag f^n_l - \alpha_l^\dag g_l)
 		(\1 - {\alpha_l}^\dag g_l)^{-1}\|
	\\
	& \le
	  \|(\1 - {\alpha^n_l}^\dag f^n_l)^{-1}\|
 	  \|{\alpha^n_l}^\dag (f^n_l - g_l)
		+ ({\alpha^n_l}^\dag - \alpha_l^\dag) g_l\|
 	  \|(\1 - {\alpha_l}^\dag g_l)^{-1}\|
	\le
	\frac{2\sup_K\|f^n_l-g_l\|}{(1-\|\alpha^n_l\|)(1-\|\alpha_l\|)},
 \end{aligned}
\]
which shows that $(\1 - {\alpha^n_l}^\dag f^n_l)^{-1} \overset{\mbb{D}}{\rightrightarrows} (\1 - g_l(0)^\dag g_l)^{-1}$. Using the explicit form \eqref{eq:matrix-SA} of the matrix Schur algorithm we conclude that $f^n_{l+1} \overset{\mbb{D}}{\rightrightarrows} g_{l+1}$, thus $\alpha_{l+1}=g_{l+1}(0)$.

This identifies $\boldsymbol\alpha$ as the Schur parameters of $g$, proving that $g=f$ because the Schur parameters characterize the Schur function of a non-trivial measure on $\T$. Therefore, $f$ is the only limit point of $f^n$, which means that $f^n \overset{\mbb{D}}{\rightrightarrows} f$. Then, the above inductive proof shows that $f^n_l \overset{\mbb{D}}{\rightrightarrows} f_l$ for any $l$.

The previous arguments also work for finite sequences of Verblunsky parameters. Since $(-\alpha_{l-1}^{n,\dag},-\alpha_{l-2}^{n,\dag},\dots,-\alpha_0^{n,\dag},\1) \to (-\alpha_{l-1}^\dag,-\alpha_{l-2}^\dag,\dots,-\alpha_0^\dag,\1)$, we also find that $b^n_l \overset{\mbb{D}}{\rightrightarrows} b_l$ for any $l$.

On the other hand, $\1+zf^n \overset{\mbb{D}}{\rightrightarrows} \1+zf$, while the relations
\[
 \|(\1-zf^n)^{-1} - (\1-zf)^{-1}\| =  \|(\1-zf^n)^{-1} z(f^n-f) (\1-zf)^{-1}\|
 \le \frac{\|f^n-f\|}{(1-|z|)^2},
 \qquad
 z \in \D,
\]
imply that $(\1-zf^n)^{-1} \overset{\mbb{D}}{\rightrightarrows} (\1-zf)^{-1}$. Therefore, rewriting the relation between Schur and Carath\'eodory functions in \eqref{eq:C-S} as $F = (\1+zf)(\1-zf)^{-1}$, we find that $F^n \overset{\mbb{D}}{\rightrightarrows} F$. If $\nu$ is any limit point of $\mu^n$ in the $\ast$-weak topology, then the relation given in \eqref{eq:C-S} between a measure and its Carath\'eodory function shows that $F^n \overset{\mbb{D}}{\rightrightarrows} G$ for some subsequence, where $G$ is the Carath\'eodory function of $\nu$. Hence, $G=F$ and $\nu=\mu$ because the Carath\'eodory function characterizes a measure on $\T$.

Concerning the orthonormal polynomials, using recurrence \eqref{eq:REC-OP} it is straightforward to prove by induction on $l$ that $\varphi^{n,L}_l \overset{\mbb{C}}{\rightrightarrows} \varphi^L_l$, $\varphi^{n,R}_l \overset{\mbb{C}}{\rightrightarrows} \varphi^R_l$, $\varphi^{n,L,\ast}_l \overset{\mbb{C}}{\rightrightarrows} \varphi^{L,\ast}_l$, $\varphi^{n,R,\ast}_l \overset{\mbb{C}}{\rightrightarrows} \varphi^{R,\ast}_l$.

Finally, $\mu^n_{[j,k]} \overset{\ast}{\to} \mu_{[j,k]}$ follows from the convergence $\mu^n \overset{\ast}{\to} \mu$ and $\varphi^{n,R}_l \overset{\mbb{C}}{\rightrightarrows} \varphi^R_l$, and analogously for $\hat{\mu}^n_{[j,k]} \overset{\ast}{\to} \hat{\mu}_{[j,k]}$ using the convergence of left orthonormal polynomials. The fact that $\lim_{n\to\infty}\mc{H}^n_{(j,k)}=\mc{H}_{(j,k)}$ due to the presence of a finite number of Schur parameters in these truncations, together with the convergence $f^n_k \overset{\mbb{D}}{\rightrightarrows} f_k$, $b^n_j  \overset{\mbb{D}}{\rightrightarrows} b_j$, yields $h^n_{[j,k]} \overset{\mbb{D}}{\rightrightarrows} h_{[j,k]}$. A similar argument using $\lim_{n\to\infty}\hat{\mc{H}}^n_{(j,k)}=\hat{\mc{H}}_{(j,k)}$ shows that $\hat{h}^n_{[j,k]} \overset{\mbb{D}}{\rightrightarrows} \hat{h}_{[j,k]}$.
\end{proof}

It is worth mentioning that Theorem \ref{thm:MATRIX-K-4} also holds for those finitely supported measures with a finite sequence of Verblunsky coefficients $\boldsymbol\alpha=(\alpha_0,\dots,\alpha_{N-1},\1)$. This can be seen by a limiting argument which mimics that one in the proof of the Theorem: such a finite sequence $\boldsymbol\alpha$ can be obtained as a pointwise limit of infinite sequences of Verblunsky coefficients, for instance $\boldsymbol\alpha^n=(\alpha_0,\dots,\alpha_{N-1},\frac{n}{n+1}\1,0,0,\dots)$; since the proof of Lemma \ref{lem:ASYMP} also works in the situation of infinite sequences of Verblunsky coefficients converging pointwise to a finite one, Theorem \ref{thm:MATRIX-K-4} for $\boldsymbol\alpha$ follows by taking $n\to\infty$ in the case $\boldsymbol\alpha^n$ where the result has been already proved.

\section{From matrix to scalar Khrushchev's formulas}
\label{sec:MtoS}

Given a unitary operator $U$, the relation between the $U$-Schur functions $f_{\hat{V}}$ and $f_V$ of different nested subspaces $\hat{V} \subset V$ is not trivial, but can be inferred from the connection between the corresponding spectral measures and Carath\'eodory functions, given in terms of the orthogonal projection $\hat{P}=P_{\hat{V}}$ onto $\hat{V}$ by
\[
 \mu_{\hat{V}} = \hat{P} \mu_V \hat{P},
 \qquad
 F_{\hat{V}} = \hat{P} F_V \hat{P}.
\]
This, together with relation \eqref{eq:C-S} between Schur and Carath\'eodory functions, as well as its inverse, gives a procedure to obtain $f_{\hat{V}}$, once $f_V$ is known. In particular, the identities
\begin{equation} \label{eq:fV-fpsi}
 F_V = (\1+zf_V)(\1-zf_V)^{-1},
 \qquad
 F_\psi = \<\psi|F_V\psi\>,
 \qquad
 f_\psi = \frac{1}{z} \frac{F-1}{F+1},
\end{equation}
provide the scalar $U$-Schur function $f_\psi$ of any vector $\psi\in V$, starting from the matrix $U$-Schur function $f_V$.

We will illustrate this procedure with the computation for any scalar measure $\mu$ on $\T$ of the Schur function for
\begin{equation} \label{eq:SM-qubit}
 |\beta\chi_j+\gamma\chi_{j+1}|^2 \, d\mu,
 \qquad
 \beta,\gamma\in\C,
 \qquad
 |\beta|^2+|\gamma|^2=1,
\end{equation}
$\chi_j$ being the scalar CMV basis related to the scalar orthonormal polynomials $\varphi_j=\varphi_j^L=\varphi_j^R$ by \eqref{eq:OP-OLP}. According to \eqref{eq:fV-fpsi} the corresponding Carath\'eodory function $F_{\beta,\gamma}$ can be expressed as
\begin{equation} \label{eq:F-F}
 F_{\beta,\gamma} =
 \begin{pmatrix} \overline\beta & \overline\gamma \end{pmatrix}
 F_{[j,j+1]}
 \begin{pmatrix} \beta \\ \gamma \end{pmatrix},
\end{equation}
in terms of the matrix Carath\'eodory function $F_{[j,j+1]}$ of the measure $\mu_{[j,j+1]}$ given in \eqref{eq:MM-MEASURE}.

Assume for instance that $j$ is even. The related Schur function $f_{[j,j+1]}$, obtained from Theorem \ref{thm:MATRIX-K-2}, factorizes as
\begin{equation} \label{eq:f[j,j+1]}
 f_{[j,j+1]} = (b_j \oplus f_{j+1}) \, \Theta(\alpha_j)^\dag =
 \begin{pmatrix}
 \alpha_jb_j & \rho_jb_j \\ \rho_jf_{j+1} & -\overline\alpha_jf_{j+1}
 \end{pmatrix},
 \qquad
 \rho_j = \sqrt{1-|\alpha_j|^2},
\end{equation}
with $(\alpha_0,\alpha_1,\dots)$ the Verblunsky coefficients of $\mu$ and $f_j$, $b_j$ the iterates and inverse iterates of the corresponding Schur function. From the first equality of \eqref{eq:fV-fpsi} we find that
\[
\begin{gathered}
 F_{[j,j+1]} = \frac{1}{1 - g_L g_R - \alpha g_L + \overline\alpha g_R}
 \begin{pmatrix}
 	1 + g_L g_R + \alpha g_L + \overline\alpha g_R & 2\rho g_L
	\\
	2\rho g_R & 1 + g_L g_R - \alpha g_L - \overline\alpha g_R
 \end{pmatrix},
 \\ \noalign{\vskip5pt}
 \alpha=\alpha_j,
 \qquad
 \rho=\rho_j,
 \qquad
 g_L(z)=zb_j(z),
 \qquad
 g_R(z)=zf_{j+1}(z).
\end{gathered}
\]
Then, the other two relations in \eqref{eq:fV-fpsi} lead to the following Schur function $f_{\beta,\gamma}$ for the measure \eqref{eq:SM-qubit},
\begin{equation} \label{eq:fbc0}
 f_{\beta,\gamma} =
 \frac{zb_jf_{j+1}+\overline\beta(\beta\alpha_j+\gamma\rho_j)b_j
 +\overline\gamma(\beta\rho_j-\gamma\overline\alpha_j)f_{j+1}}
 {1+z\gamma(\overline\beta\rho_j-\overline\gamma\alpha_j)b_j
 +\beta(\overline\beta\overline\alpha_j+\overline\gamma\rho_j)zf_{j+1}}.
\end{equation}
A similar analysis for an odd index $j$ yields the same result for $f_{\beta,\gamma}$, but changing $\beta\to\overline\beta$ and $\gamma\to\overline\gamma$.

Formula \eqref{eq:fbc0} has a nice interpretation in terms of mappings between Schur functions. To understand this interpretation let us rewrite \eqref{eq:fbc0} as
\begin{equation} \label{eq:fbc}
 f_{\beta,\gamma} = T_{u_j,v_j}(b_j,f_{j+1}),
 \qquad
 \begin{cases}
 	u_j=\overline\beta\beta_j,
 	\\
 	v_j=\overline\gamma\gamma_j,
 \end{cases}
 \quad
 \begin{pmatrix} \beta_j \\ \gamma_j \end{pmatrix} =
 \Theta(\alpha_j)^\dag
 \begin{pmatrix} \beta \\ \gamma \end{pmatrix} =
 \begin{pmatrix}
 \beta\alpha_j+\gamma\rho_j \\ \beta\rho_j-\gamma\overline\alpha_j
 \end{pmatrix},
\end{equation}
where $T_{u,v}$ maps any pair of Schur functions $g(z),h(z)$ into a single one,
\begin{equation} \label{eq:Tuv}
 T_{u,v}(g,h) :=
 \frac{zgh+ug+vh}{1+\overline{v}zg+\overline{u}zh},
 \qquad
 |u|+|v|\le1.
\end{equation}
The parameters $u_j,v_j$ in \eqref{eq:fbc} satisfy the inequality $|u_j|+|v_j| \le \sqrt{|\beta|^2+|\gamma|^2} \sqrt{|\beta_j|^2+|\gamma_j|^2} = 1$ which guarantees that $T_{u_j,v_j}$ maps Schur functions into Schur functions.

The fact that the binary operation \eqref{eq:Tuv} is a consistent composition law for Schur functions follows from the analysis of the two-variable analytic function
\begin{equation} \label{eq:2-MOBIUS}
 (z,w) \mapsto \frac{zw+uz+vw}{1+\overline{v}z+\overline{u}w},
 \qquad
 z,w\in\D,
 \qquad
 |u|+|v|\le1.
\end{equation}
Since
\[
 \left|
 \frac{zw+uz+vw}{1+\overline{v}z+\overline{u}w}
 \right| =
 \left|
 zw \frac{1+u\overline{w}+v\overline{z}}{1+\overline{v}z+\overline{u}w}
 \right| = 1,
 \qquad
 z,w\in\T,
 \qquad
 |u|+|v|<1,
\]
the maximum modulus principle implies that \eqref{eq:2-MOBIUS} maps $\D^2$ onto the closed unit disk $\overline\D$ for $|u|+|v|<1$, and by continuity also for $|u|+|v|=1$. Then, Schwarz's lemma applied to the result of substituting $z \to zg(z)$ and $w \to zh(z)$ into \eqref{eq:2-MOBIUS} ensures that $T_{u,v}(g,h)$ is a Schur function whenever $g$ and $h$ are Schur functions.

The transformations \eqref{eq:2-MOBIUS} can be considered as a two-variable version of the M\"obius transformations mapping $\D$ onto itself,
\begin{equation} \label{eq:MOBIUS}
 z \mapsto \frac{z+\alpha}{1+\overline{\alpha}z},
 \qquad
 |\alpha|<1,
\end{equation}
which, together with Schwarz's principle, constitute the essence of the scalar Schur algorithm \eqref{eq:SA}. Also, the operation $T_{u,v}$ defined by \eqref{eq:Tuv} is a binary analog of the backward step for the Schur algorithm, $f_j = T_{\alpha_j}(f_{j+1})$, given by the transformation
\[
 T_\alpha(f) := \frac{zf+\alpha}{1+\overline{\alpha}zf},
 \qquad
 |\alpha|<1.
\]
Actually, the identities $T_\alpha(f) = T_{\alpha,0}(1,f) = T_{0,\alpha}(f,1)$ show that $T_\alpha$ comes from particular choices in $T_{u,v}$.

Summarizing, the Schur function $f_{\beta,\gamma}=T_{u_j,v_j}(b_j,f_j)$ of $|\beta\chi_j+\gamma\chi_{j+1}|^2\,d\mu$ is given by a binary version $T_{u_j,v_j}$ of the backward Schur algorithm step acting on the inverse iterate $b_j$ and the iterate $f_j$ of the Schur function of $\mu$. These iterates encode the dependence of $f_{\beta,\gamma}$ on the ``left" and ``right" Schur parameters $(\alpha_k)_{k<j}$, $(\alpha_k)_{k>j}$ respectively. The dependence on the coordinates $\beta,\gamma$ and the remaining Schur parameter $\alpha_j$ is concentrated on the coefficients $u_j$, $v_j$ of the binary operation.

In spite of the absence of a factorization, this result can be viewed as a Khrushchev formula for $|\beta\chi_j+\gamma\chi_{j+1}|^2\,d\mu$ because it expresses the Schur function of this modified measure in terms of iterates and inverse iterates of the Schur function for the original measure $\mu$. In the extreme cases
\[
\begin{aligned}
 & \beta=1, \; \gamma=0 \kern7pt \Rightarrow \kern7pt
 T_{u_j,v_j}(b_j,f_{j+1}) = T_{\alpha_j,0}(b_j,f_{j+1})
 = b_j\,T_{\alpha_j}(f_{j+1}) = b_j f_j,
 \\
 & \beta=0, \; \gamma=1 \kern7pt \Rightarrow \kern7pt
 T_{u_j,v_j}(b_j,f_{j+1}) = T_{0,-\overline\alpha_j}(b_j,f_{j+1})
 = f_{j+1}\,T_{-\overline\alpha_j}(b_j) = f_{j+1} b_{j+1},
\end{aligned}
\]
we recover the known scalar Khrushchev formulas for $|\chi_k|^2\,d\mu=|\varphi_k|^2\,d\mu$ with $k=j,j+1$.

A similar analysis of the measure $|\beta\varphi_j+\gamma\varphi_{j+1}|^2\,d\mu$ using the results of Theorem \ref{thm:MATRIX-K-4} yields the following expression for the corresponding Schur function $h_{\beta,\gamma}$,
\[
 h_{\beta,\gamma} =
 \frac{zb_jf_{j+1} + \overline\beta(\beta\alpha_jb_j+\gamma\rho_j)
 + \overline\gamma(\beta\rho_jb_j-\gamma\overline\alpha_j)f_{j+1}}
 {1 + \gamma(\overline\beta\rho_j-\overline\gamma\alpha_jb_j)z
 + \beta(\overline\beta\overline\alpha_j+\overline\gamma\rho_jb_j)zf_{j+1}},
\]
which should be compared with the Schur function $f_{\beta,\gamma}$ given in \eqref{eq:fbc0}.

These examples illustrate the usefulness of matrix Khrushchev's formulas for the computation of scalar Schur functions. This kind of results are of interest, not only from the mathematical point of view, but also for the study of return properties of quantum states which are defined by certain superpositions  \cite[Section 6]{GVWW}.

\section*{Acknowledgements}

C. Cedzich and R.F. Werner acknowledge support from the ERC grant DQSIM and the European project SIQS.

A.H. Werner acknowledges support from the ERC grant TAQ.

F.A. Gr\"{u}nbaum acknowledges support from the Applied Math. Sciences subprogram of the Office of Energy Research, US Department of Energy, under
Contract DE-AC03-76SF00098, and from AFOSR grant FA95501210087 through a subcontract to Carnegie Mellon University.

The work of L. Vel\'azquez is partially supported by the research project MTM2011-28952-C02-01 from the Ministry of Science and Innovation of Spain and the European Regional Development Fund (ERDF), and by Project E-64 of Diputaci\'on General de Arag\'on (Spain).

\end{document}